\documentclass{amsart}

\usepackage[english]{babel}
\usepackage[utf8]{inputenc}
\usepackage[T1]{fontenc}
\usepackage{textcomp}

\usepackage[a4paper]{geometry}

\usepackage{amsmath}
\usepackage{amssymb}
\usepackage{amsthm}
\usepackage{amsfonts}
\usepackage{enumitem}
\usepackage{graphicx}
\usepackage{subcaption}
\usepackage{mathtools}
\usepackage{tikz-cd}
\usepackage{bm}
\usepackage{float} 
\usepackage{wasysym} 
\usepackage{rotating} 
\usepackage{tcolorbox} 
\usepackage{listings} 
\usepackage{caption}
\usepackage{nicefrac}
\usepackage{xcolor}
\usepackage[colorlinks=true, linkbordercolor = {gray}, allcolors=blue]{hyperref}
\usepackage{booktabs} 
\usepackage{multirow}
\usepackage{tabularx} 
\usepackage{colortbl} 
\usepackage[export]{adjustbox}
\usepackage{wrapfig}
\usepackage{relsize}

\makeatletter
\def\@seccntformat#1{\csname the#1\endcsname\quad}

\renewcommand{\section}{\@startsection{section}{1}{\z@}%
  {-3.5ex \@plus -1ex \@minus -.2ex}%
  {2.3ex \@plus .2ex}%
  {\normalfont\Large\bfseries}}

\renewcommand{\subsection}{\@startsection{subsection}{2}{\z@}%
  {-3.25ex \@plus -1ex \@minus -.2ex}%
  {1.5ex \@plus .2ex}%
  {\normalfont\large\bfseries}}

\renewcommand{\subsubsection}{\@startsection{subsubsection}{3}{\z@}%
  {-3.25ex \@plus -1ex \@minus -.2ex}%
  {1.5ex \@plus .2ex}%
  {\normalfont\normalsize\bfseries}}
\makeatother

\makeatletter
\renewcommand{\paragraph}{\@startsection{paragraph}{4}{\z@}%
  {3.25ex \@plus1ex \@minus.2ex}%
  {-1em}%
  {\normalfont\normalsize\bfseries}}
\makeatother

\hypersetup{
    pdftitle={Computations of the slice genus and the unknotting number of links via machine learning},
    pdfauthor={Yutong Dai, Oliver Hayman, Andr\'as Juh\'asz, Ludovico Morellato},
    pdfkeywords={knot theory, machine learning, reinforcement learning, Bayesian optimisation}
}

\graphicspath{{./}}

\usepackage[style = alphabetic]{biblatex}
\usepackage{csquotes}
\numberwithin{equation}{section}
\theoremstyle{plain}
\newtheorem{theorem}{Theorem}[section]
\newtheorem*{theorem*}{Theorem}
\newtheorem{proposition}[theorem]{Proposition}
\newtheorem{corollary}[theorem]{Corollary}
\newtheorem*{corollary*}{Corollary}
\newtheorem{lemma}[theorem]{Lemma}
\newtheorem{conjecture}[theorem]{Conjecture}

\theoremstyle{definition}
\newtheorem{definition}[theorem]{Definition}
\newtheorem{remark}[theorem]{Remark}

\newtheorem{example}[theorem]{Example}

\renewcommand{\epsilon}{\varepsilon}
\renewcommand{\S}{S}

\def\R{{\mathbb R}}
\def\C{{\mathbb C}}

\def\N{{\mathbb N}}

\newcommand\abs[1]{\left\lvert #1 \right\rvert}
\newcommand\norm[1]{\left\lVert #1 \right\rVert}
\newcommand{\ktparen}[1]{\IfNoValueF{#1}{\if\relax\detokenize{#1}\relax\else(#1)\fi}}
\newcommand{\ktsub}[1]{\IfNoValueF{#1}{\if\relax\detokenize{#1}\relax\else_{#1}\fi}}
\NewDocumentCommand{\unk}{!g}{\ensuremath{{u}\ktparen{#1}}}
\NewDocumentCommand{\genSeif}{!g}{\ensuremath{{g}_3\ktparen{#1}}}
\NewDocumentCommand{\gentop}{!g}{\ensuremath{{g}_4^{top}\ktparen{#1}}}
\NewDocumentCommand{\gen}{!g}{\ensuremath{{g}_4\ktparen{#1}}}
\NewDocumentCommand{\sgen}{!g}{\ensuremath{{g}_4^{\ast}\ktparen{#1}}}
\NewDocumentCommand{\rgen}{!g}{\ensuremath{{g}_r\ktparen{#1}}}
\NewDocumentCommand{\rsgen}{!g}{\ensuremath{{g}_r^{\ast}\ktparen{#1}}}
\NewDocumentCommand{\ualg}{!g}{\ensuremath{\mathnormal{u}_a\ktparen{#1}}}
\NewDocumentCommand{\lk}{!g}{\ensuremath{\mathrm{lk}\ktparen{#1}}}
\NewDocumentCommand{\abslk}{!g}{\ensuremath{\abs{\mathrm{lk}}\IfNoValueF{#1}{\if\relax\detokenize{#1}\relax\else\!(#1)\fi}}}
\NewDocumentCommand{\mir}{!g}{\ensuremath{\operatorname{m}\ktparen{#1}}}
\NewDocumentCommand{\torus}{!g}{\ensuremath{\operatorname{T}\ktsub{#1}}}

\DeclareMathOperator{\FM}{FM}

\definecolor{light-gray}{gray}{0.95}

\newcommand{\greyrule}{\arrayrulecolor{black!25}\hline\arrayrulecolor{black}}

\definecolor{codegreen}{rgb}{0,0.6,0}
\definecolor{codegray}{rgb}{0.5,0.5,0.5}
\definecolor{codepurple}{rgb}{0.58,0,0.82}
\definecolor{backcolour}{rgb}{0.95,0.95,0.92}

\lstdefinestyle{mystyle}{
    backgroundcolor=\color{light-gray},
    commentstyle=\color{codegreen},
    keywordstyle=\color{magenta},
    numberstyle=\tiny\color{codegray},
    stringstyle=\color{codepurple},
    basicstyle=\ttfamily\footnotesize,
    breakatwhitespace=false,
    breaklines=true,
    captionpos=b,
    keepspaces=true,
    numbers=left,
    numbersep=5pt,
    showspaces=false,
    showstringspaces=false,
    showtabs=false,
    tabsize=2,
    frame=lines,
    rulecolor=\color{gray}
}

\usetikzlibrary{shapes.geometric}
\tikzset{
    mynode/.style={rounded corners, inner sep=4pt, fill=#1, draw=none}
}
\title[Computations of the slice genus and the unknotting number of links]{Computations of the slice genus and the unknotting number of links via machine learning%
}

\author{Yutong Dai}
\address{Department of Mathematics, Humboldt-Universit\"at zu Berlin, Unter den Linden 6, 10099 Berlin, Germany}
\email{yutong.dai@student.hu-berlin.de}

\author{Oliver Hayman}
\address{DRW Holdings, 540 West Madison Street, Suite 2500, Chicago, IL 60661, US}
\email{ohayman@drwholdings.com}

\author{Andr\'{a}s Juh\'{a}sz}
\address{Mathematical Institute, University of Oxford, Andrew Wiles Building,
	Radcliffe Observatory Quarter, Woodstock Road, Oxford, OX2 6GG, UK}
\email{juhasza@maths.ox.ac.uk}

\author{Ludovico Morellato}
\address{Department of Mathematics, Università di Padova, via Trieste, 63, 35121 Padova, Italy}
\email{ludovico.morellato@phd.unipd.it}

\begin{document}
\renewcommand{\thefootnote}{\fnsymbol{footnote}}
\renewcommand{\thefootnote}{\arabic{footnote}}
\setcounter{footnote}{0}

\begin{abstract}
    Links are disjoint unions of circles smoothly embedded in $\S^3$. We use reinforcement learning and Bayesian optimisation to obtain new upper bounds on several link invariants that are not known to be algorithmically computable: the slice genus and the unknotting number for links, and the strong slice genus for algebraically split links. We also compute lower bounds using known invariants. Combining the upper and lower bounds, we obtain new exact values in many cases. Our unknotting agents can reproduce the non-additivity of the unknotting number for several counterexamples due to Brittenham and Hermiller, in some cases finding new unknotting trajectories.
\end{abstract}

\maketitle


\section{Introduction}
\label{sec:intro}

The unknotting number and the smooth slice genus are topological invariants of links that are not known to be algorithmically computable, as they are defined as minima over infinite sets. 
The \emph{unknotting number} of a link is the minimal number of crossing changes needed to turn a diagram of the link into a diagram of the unlink. 
The difficulty is that every link admits infinitely many diagrams.
For example, out of the 249 prime knots with up to ten crossings, we do not know the unknotting number of ten.

The (weak) \emph{slice genus} of a link $L$ is the minimal genus of a compact, connected, oriented surface in the 4-ball that bounds $L$.
If the components of $L$ are $K_1, \dots, K_\ell$, the \emph{strong slice genus} of $L$ is the minimum of $\sum_{i=1}^\ell g(F_i)$, where $F_1,\dots,F_\ell$ are pairwise disjoint compact, connected, oriented surfaces in the 4-ball with $\partial F_i = K_i$. 
Such a collection of surfaces exists if and only if the link is \emph{algebraically split}; i.e., $\lk(K_i,K_j) = 0$ for $i \neq j$. 
Surfaces bounded by $L$ in the 4-ball can be represented by movies starting with $L$ and ending with the empty link that consist of isotopies, births and deaths of unlinked unknot components, and band attachments. 
The set of possible movies is infinite.
If we disallow births of unlinked unknot components, then we obtain what is known as a \emph{ribbon surface}, and minimising the genus over such surfaces gives rise to the notions of the \emph{weak and strong ribbon genera}.
The ribbon genus is at least the slice genus.

The application of machine learning to knot-theoretic problems has developed rapidly over the last decade. 
On the supervised side, Hughes~\cite{hughes2020neuralnetwork} trained neural networks to predict knot invariants directly from a diagram, and Jejjala, Kar, and Parrikar~\cite{jejjala2019volume} to recover the hyperbolic volume of a knot from its Jones polynomial. 
More broadly, Davies et al.~\cite{davies2021advancing} used machine learning feature attribution techniques to uncover previously unknown relationships between invariants, leading to new conjectures in knot theory.
On the reinforcement learning side, several works have framed the simplification of a knot as a sequential decision problem. Gukov, Halverson, Ruehle, and Su\l{}kowski~\cite{gukov2021unknot} trained an agent to recognise and untie the unknot via Markov moves and braid relations, while Applebaum et al.~\cite{applebaum2025unknottingrl} used a reinforcement learning agent to bound the unknotting number, producing along the way a large dataset of hard unknot diagrams.

Closer to our work, Gukov, Halverson, Manolescu, and Ruehle~\cite{gukov2025ribbon} looked for band moves exhibiting knots as ribbon. 
They framed the search for a ribbon disc as a game aimed at finding a band path very close to the one we describe in Section~\ref{sec:env}. 
The present work follows this line, unifying the computation of the unknotting number and the weak and strong slice genera (including the ribbon versions) into a single band path framework.
To obtain upper bounds on the unknotting number and the various slice genera, we implemented three agents: a random walker tuned via Bayesian optimisation, a neural network trained with reinforcement learning, and a mixed agent that combines the two.
While in the work of Gukov et al.\ the random walker outperformed the RL agent, the opposite holds in our case.
The novelty of our work is that we study links instead of just knots and all the aforementioned variants of the slice genus, not only ribbonness (i.e., whether a knot has ribbon genus zero). 
Furthermore, our unified framework allows us to bound the unknotting number.

While finding an upper bound on the unknotting number or slice genus involves a search over an infinite space of moves, verifying that a proposed sequence achieves a claimed upper bound is fast.
We phrase the computation of both invariants as a single combinatorial game on a diagram, and we let an agent play it. 
A won episode returns a certificate: an explicit sequence of moves that provides a verifiable upper bound for the invariant.
The environment never \emph{predicts} an invariant; it produces a sequence of moves that one can verify by hand. 
We also compute existing lower bounds obtained via invariants such as various signatures, the invariants $\tau$ and $\nu$ due to Ozsv\'ath and Szab\'o, and the Rasmussen $s$-invariant.
When the upper bound obtained via an agent matches the resulting lower bound, the invariant is determined exactly.
The lower bounds we match against are computed independently of the agents, using the results collected in Appendix~\ref{sec:bounds}. 

The numerical results we obtained with the trained reinforcement learning agents and the tuned random walkers can be found in Section~\ref{sec:results}.
Prior to our work, no unknotting number data were available for the 1{,}007 unoriented 11-crossing, at least two-component prime links.
We determine the exact value for 690 and obtain lower and upper bounds for the remaining ones.

When studying the slice genus, allowing the birth of unlinked unknot components did not improve our upper bounds, which is consistent with the higher-genus slice--ribbon conjecture stating that the slice and ribbon genera agree.
As the slice genus is at most the ribbon genus and all the invariants that we compute are lower bounds on the slice genus, our results apply to both the slice and ribbon genera.
Hence, below, we report only on the slice genus.
Of all oriented, at least two-component prime links with up to 11 crossings, the weak slice genus was previously undetermined for 407 links, and the strong slice genus for 732 algebraically split links.
We determine the weak slice genus for 144 of them.
The remaining 263 turn out to all have weak slice genus zero or one.
We determine the strong slice genus for 676 algebraically split links and strictly improve the bounds for the remaining 56.
Some of our results have already been incorporated into the \cite{linkinfo} dataset.

We also studied the slice genus and unknotting number of 14-crossing prime knots, of which there are $46{,}972$.
Dunfield and Gong \cite{dunfield2025ribbonconslice} showed $1{,}194$ have slice genus zero, and $408$ have slice genus at least one.
We determine a further $44{,}838$ values exactly and strictly narrow the interval for $940$ more, bringing the census to $46{,}032$ exact values.
For the unknotting number, the upper bounds of all 14-crossing prime knots come from \cite{brittenham2017unknottingsearch}, and the new contribution is purely on the lower bound. The obstructions give a non-trivial lower bound for $1{,}288$ knots, of which $1{,}097$ become exactly determined, with the new bound meeting Brittenham's upper bound.

When starting this project in 2023, our main motivation for developing the unknotting agents was to find counterexamples to the additivity of the unknotting number.
The knots $K$ and $K^{\prime}$ are called \emph{symbionts} if $\unk{K \sharp K^{\prime}} < \unk{K} + \unk{K^{\prime}}$.
The first symbionts were produced by Brittenham and Hermiller~\cite{brittenham2025additivity, brittenham2026unknottingnumberconnectedsums4_1_5_1}.
We tested the unknotting agents on some of their connected sums; see Section~\ref{subs:unknotting-additivity}.
The agents were able to independently reproduce that $\unk{K \sharp \mathtt{9_{10}}} \leq 4$ for ten symbionts $K$ of $\mathtt{9_{10}}$, and that $\unk{\mathtt{10_{41}} \sharp \mir{\mathtt{9_{10}}}} \leq 4$; in each case, the two summands have unknotting numbers $2$ and $3$, adding up to $5$.
None of these connected sums appear in the datasets we used to train, tune, or test the agents.
Four of the sequences found by the agents reproduce almost exactly the unknotting sequence in the proof of~\cite[Theorem~1.1]{brittenham2026unknottingnumberconnectedsums4_1_5_1}, and a fifth reproduces its mirror image; the others are original.
Two of them pass through $\mathtt{4_1} \sharp \mathtt{9_{10}}$ and then through $\mathtt{K15n4005}$, which does not appear in the sequence of Brittenham and Hermiller.
This gives a second unknotting sequence of three crossing changes for $\mathtt{4_1} \sharp \mathtt{9_{10}}$. 
The new unknotting sequences are depicted in Appendix~\ref{app:sequences}.

\subsubsection*{Data availability}
The code, the datasets, the trained agents, and the evaluations are available at 
\[
\texttt{\small https://github.com/andras-juhasz/slice-genus-unknotting-rl/releases/tag/v1.0-preprint}
\]

\subsubsection*{AI declaration}
The first and second authors wrote large parts of the codebase and did not use AI to assist with coding.
The third author used ChatGPT 5.6 and 6 Pro to proofread the manuscript, audit the code, and fix minor bugs.
The fourth author used Claude Opus 5 to proofread the manuscript and assist with code development.

\label{subsubs:acknowledgements}

\section{Basics: links, diagrams, the band game and RL}
\label{sec:background}

We work in the $3$-sphere $\S^3$.
A smooth $\ell$\emph{-component link} in $\S^3$ is a smooth embedding 
\[
L \colon \bigcup_{i=1}^{\ell} \S^1 \times \{i\} \hookrightarrow \S^3, 
\]
and a \emph{knot} is a $1$-component link.
We usually identify $L$ with its image, which is a smooth, oriented one-dimensional submanifold of $\S^3$, and we denote its connected components by $K_1, \dots, K_{\ell}$.
Links are regarded up to \emph{ambient isotopy}, that is, up to a smooth family of diffeomorphisms of $\S^3$ starting at the identity and carrying one link onto the other.

Although links live in $\S^3$, they can be faithfully encoded by $2$-dimensional pictures, called \emph{link diagrams}. A diagram $D$ of $L$ is a projection of $L$ onto a plane, which is an immersion with only transverse double points, called \emph{crossings}, at each of which we record which of the two strands lies above the other.
The over/under decoration is exactly what is needed to recover $L$ from $D$ up to ambient isotopy, so a diagram encodes the link up to equivalence.
By Reidemeister's theorem \cite[Theorem 3.8.1]{cromwell2004knots}, two diagrams represent equivalent links if and only if one can be transformed into the other via a sequence of planar isotopies and \emph{Reidemeister moves}, of which there are three.

Diagrams are also convenient from a computational perspective, since they admit a compact encoding known as a \emph{PD code}.
Let $D$ be an oriented link diagram with $n$ crossings.
Label each crossing with an integer from $1$ to $n$. Then start from a crossing and walk along the link following the orientation, labelling each segment between two crossings with integers starting from $1$ in increasing order, and repeating for every component, always increasing the segment labels.
The resulting \emph{Planar Diagram code (PD code)} is the ordered list of $n$ 4-tuples, where the $i$-th tuple corresponds to the $i$-th crossing and encodes the labels of the four segments around the crossing, starting from the incoming under-strand and going around counter-clockwise.
The isotopy type of a connected link diagram in $S^2$ is uniquely determined by its PD code.
In $\R^2$, the only ambiguity when reconstructing a connected diagram from a PD code is the choice of the unbounded region.
A given diagram admits several PD codes, depending on the choice of labels.
For multi-component links with few crossings on some components, it might happen that the orientation is not uniquely determined, and this problem is avoided by adding the data of crossing signs ($\pm 1$) to the PD code.
A PD code for the trefoil can be seen in Figure~\ref{fig:PD_trefoil}.

\begin{figure}[H]
    \centering
    \includegraphics[height=3.4cm]{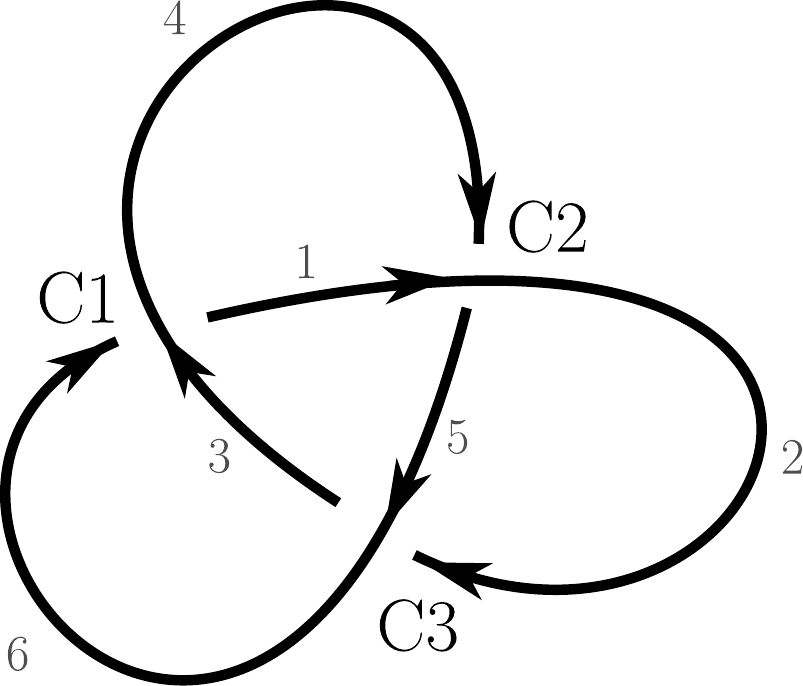}
    \caption{A PD code of the trefoil is $[(6,3,1,4), (4,1,5,2), (2,5,3,6)]$, where the crossings are labelled as $Ci$ for $i \in \{1,2,3\}$.}
    \label{fig:PD_trefoil}
\end{figure}

\subsection{The unknotting number and the band path}
\label{subs:unknotting-num}

The unknotting number, denoted by $\unk$, is a numerical invariant measuring how knotted a link is.
To define it, we need the notion of \emph{crossing change}, an elementary move that lets the link pass through itself.
It is a regular homotopy of the link which is an embedding except for a single time when there is exactly one transverse double point, where two strands of the link pass through each other.
By iterating this move finitely many times, any link can be turned into the unknot; see~\cite[Section~3.1]{adams2004knotbook}.
We are interested in the minimum number of times we need to perform this move.

\begin{definition}[Unknotting number]
\label{def:unknotting-number-link}
    Let $L$ be a link in $\S^3$ with $\ell \geq 1$ components.
    The \emph{unknotting number} $\unk{L}$ is the minimum number of crossing changes required to turn $L$ into the $\ell$-component unlink.
\end{definition}

For $\ell \geq 2$, the unknotting number is also called ``unlinking number''. We will call it the unknotting number for both knots and links.
Computing this value is very difficult, and no general algorithm is currently known.
Even though the problem is $3$-dimensional in nature, we can recast it in the completely two-dimensional framework of link diagrams.
Given a crossing $c$ of a diagram $D$, by a crossing change at $c$, we mean the local modification that swaps which of the two strands at $c$ lies above the other, leaving the rest of $D$ unchanged.
Every diagram $D$ can be turned into a diagram of an unlink by changing at most $c(D)$ crossings, where $c(D)$ is the number of crossings of $D$.
Write $\unk{D}$ for the minimal number of crossing changes needed to obtain the unlink from $D$.
This is only an upper bound on $\unk{L}$, as a different diagram may admit a shorter sequence.
Indeed, for any link $L$, one has
\[
\unk{L} = \min_{\substack{D \text{ diagram} \\ \text{for } L }} \unk{D}.
\]
While $u(D)$ can be computed by changing all subsets of crossings of size at most $c(D)/2$ (which is still computationally expensive), every link has infinitely many diagrams, and there are links for which the minimum is not achieved by a minimal crossing number diagram.

To remove this dependence on a specific diagram while still working with combinatorial moves in dimension 2, we package the data of ``deforming first, then performing a crossing change'' into a single elementary object called a \emph{band path}.

\begin{definition}[Band path]
\label{def:band-path}
    A \emph{band path} on a diagram of $L$ consists of
    \begin{itemize}
        \item an edge of the diagram between two neighbouring crossings of $L$;
        \item a sequence of crossing-increasing Reidemeister~II moves we want to perform on that edge, possibly interleaved with Reidemeister~I moves. This is the ``path'' along which we slide the chosen edge to reach a target edge of the diagram.
    \end{itemize}
\end{definition}

At the end of the band path, we are going to perform the crossing change at one of the two new crossings created by the last Reidemeister~II move.
An example can be seen in Figure~\ref{fig:unknotting_number_trefoil_band}.

\begin{figure}[ht]
    \centering
    \includegraphics[width=0.20\linewidth, valign=t]{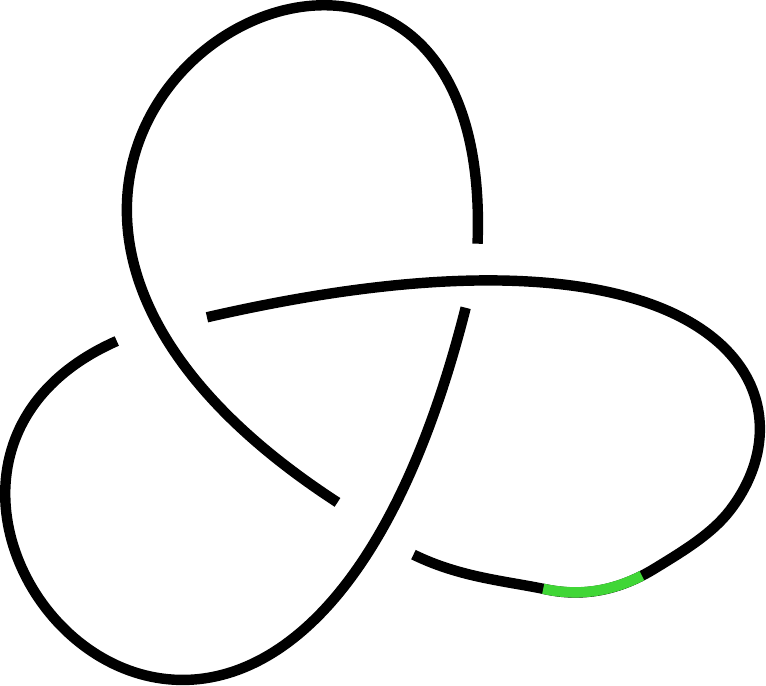}
    \hspace{15pt}
    \includegraphics[width=0.20\linewidth, valign=t]{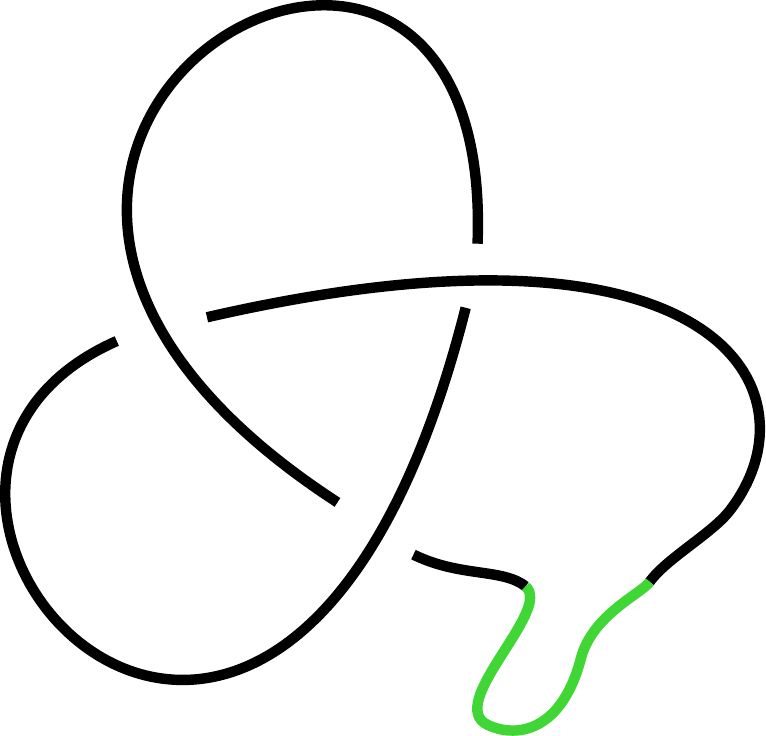}
    \hspace{15pt}
    \includegraphics[width=0.20\linewidth, valign=t]{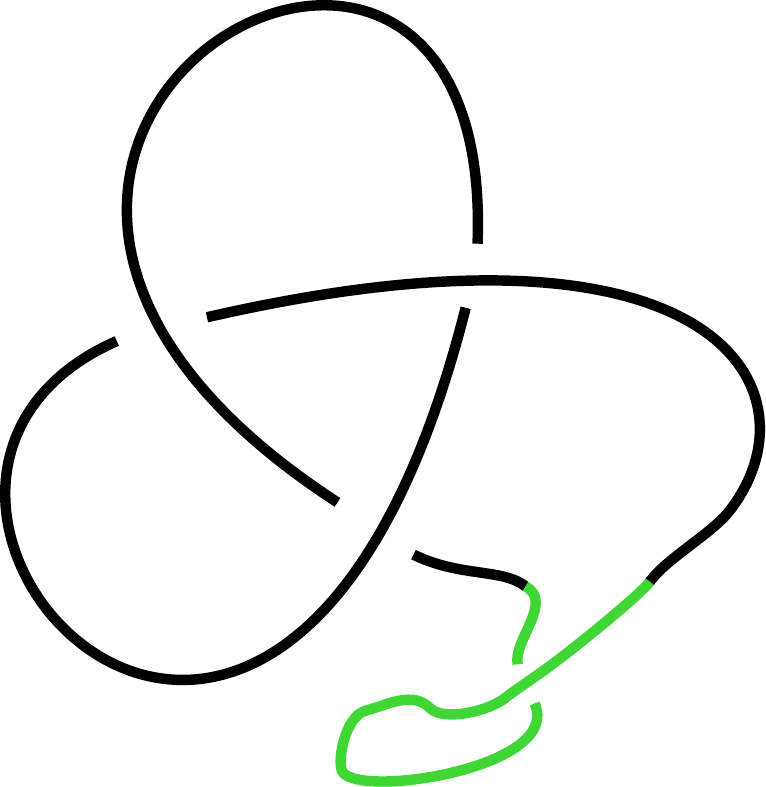}
    \\[10pt]
    \includegraphics[width=0.20\linewidth, valign=b]{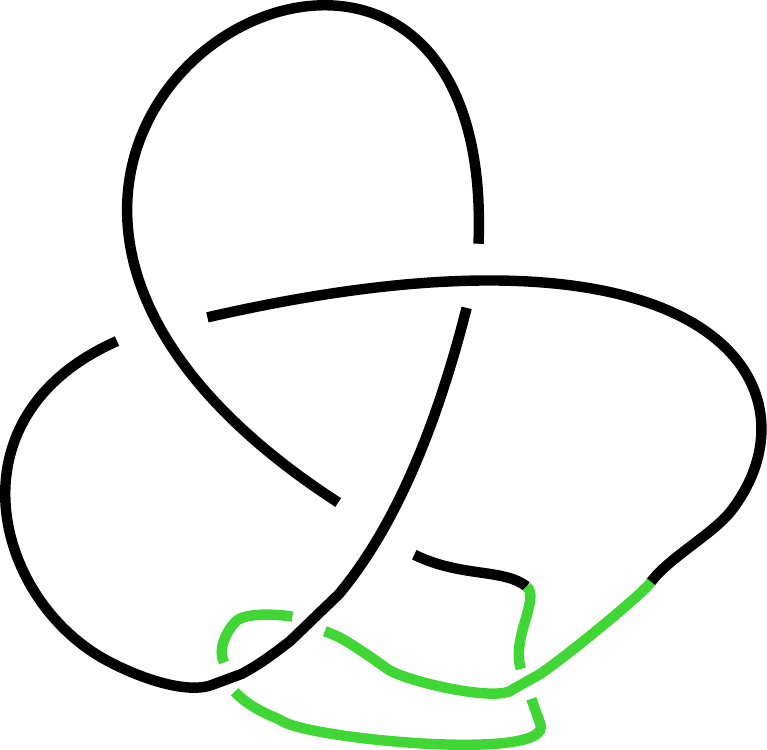}
    \hspace{15pt}
    \includegraphics[width=0.20\linewidth, valign=b]{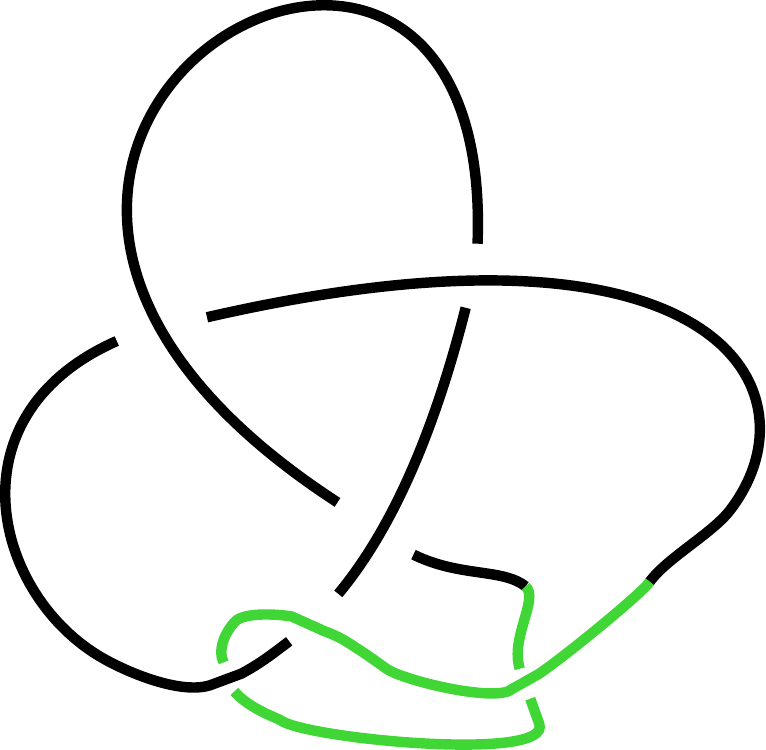}
    \hspace{15pt}
    \includegraphics[width=0.19\linewidth, valign=b]{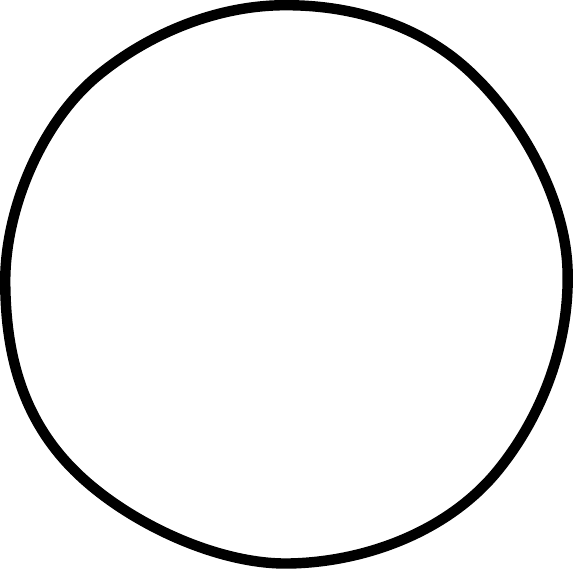}
    \caption{The trefoil has unknotting number at most 1: movie with a band.}
    \label{fig:unknotting_number_trefoil_band}
\end{figure}

A crossing change on a diagram can be represented by a band path. Indeed, we can perform a single Reidemeister~II move (and possibly one Reidemeister~I move), where the initial and target edges meet at the crossing we want to change, followed by changing one of the two new crossings and then locally simplifying using Reidemeister moves. Again, see Figure~\ref{fig:unknotting_number_trefoil_band}.
Therefore, from now on, when we say ``we perform a crossing change'', we actually mean that we choose a band path and change one of the two crossings created by the final Reidemeister~II move.
The unknotting number $\unk{L}$ of a link $L$ is the minimal number of band paths needed to turn a diagram of $L$ into a diagram of the unlink.
This allows us to work with a fixed diagram instead of having to explore all diagrams of $L$.
However, there are infinitely many band paths. Our implementation only allows paths that do not intersect themselves and visit each edge, including the initial and target edges, at most once (visiting the same component of the complement of the diagram multiple times is allowed).

\begin{remark}[Orientation independence of $\unk{L}$]
\label{rmk:u-orientation-independent}
    The unknotting number is defined purely in terms of crossing changes, and a crossing change simply swaps the overstrand and the understrand at a crossing, independently of any orientation.
    Therefore, $\unk{L}$ is well defined for \emph{unoriented} links.
    This is in sharp contrast with many of the invariants we will use to \emph{bound} $\unk{L}$ in Appendix~\ref{sec:bounds}, which do require a choice of orientation.
\end{remark}

\subsection{The slice genus}
\label{subs:slice-genus}

A central theme in low-dimensional topology is the study of surfaces bounded by a given knot.
Seifert~\cite{seifert1935geschlecht} gave an algorithm producing, for any knot $K$, a compact, connected, oriented surface $S \subset \S^3$ with $\partial S = K$. The minimal genus of such a surface is an invariant of $K$, called the \emph{Seifert genus} of $K$ and denoted by $\genSeif{K}$.
Both the algorithm and the definition apply verbatim to a link $L$.

Let $B^4 = \{ x \in \R^4 \mid \norm{x} \leq 1 \}$, so that $\S^3 = \partial B^4$, and fix a knot $K$ in $\S^3$.

\begin{definition}[Slice surface and slice genus]
\label{def:slice-genus}
    A \emph{slice surface} for a knot $K$ is a compact, connected, oriented surface $F$ smoothly embedded in $B^4$ such that $\partial F = K$.
    The \emph{slice genus} $\gen{K}$ of $K$ is the minimal genus of a slice surface for $K$.
\end{definition}

Note that $\gen{K} \leq \genSeif{K}$, as a Seifert surface can be embedded in $B^4$ to become a slice surface.
A useful way to visualise a slice surface is as a movie of links.
Let $r \colon B^4 \to \R$ be the radial function.
We regard $t := 1 - r$ as time and intersect the surface with successive $3$-dimensional level sets.
At a regular time, we obtain a link in a copy of $\S^3$, which we represent by a diagram in the plane.
As $t$ increases, these diagrams change by planar isotopies, Reidemeister moves, and a non-trivial local change called a \emph{band move}, which records a saddle point of the surface with respect to $t$.

The movie starts at $t=0$ with a diagram of our knot $K$ and ends with a diagram of the unlink. Then we glue discs to the components of the final unlink, corresponding to local maxima of $t$ on the surface.
The whole trace swept out by the moving diagram, together with these final caps, is a compact, oriented surface $F \subset B^4$ with $\partial F = K$; i.e., a slice surface for $K$.
Its genus is determined entirely by the moves we performed during the movie, so every such movie gives an explicit upper bound on $\gen{K}$.

The slice genus and the unknotting number are related by the inequality $g_4(K) \le u(K)$. This can be seen by smoothing the double points of the immersed surface in $B^4$ that the regular homotopy in $S^3$ from $K$ to the unknot traces out, and noting that each smoothing increases the genus by one.
It should now be clear that the band framework that we developed for $\unk{K}$ in Section~\ref{subs:unknotting-num} can be reused almost verbatim to produce diagrammatic upper bounds on $\gen{K}$ as well.
The key is to perform a band move instead of crossing changes at the end of the band: instead of swapping the over/under information at one of the new crossings produced by the final Reidemeister~II move of a band, we identify a target edge, cut the target and the band, and reconnect them in a way that agrees with the orientation of the two edges; see Figure~\ref{fig:band_move_oriented}.

\begin{figure}[ht]
    \hspace{30pt}
    \hfill
    {\begin{minipage}[c]{0.30\textwidth}
        \centering
        \includegraphics[width=0.8\linewidth]{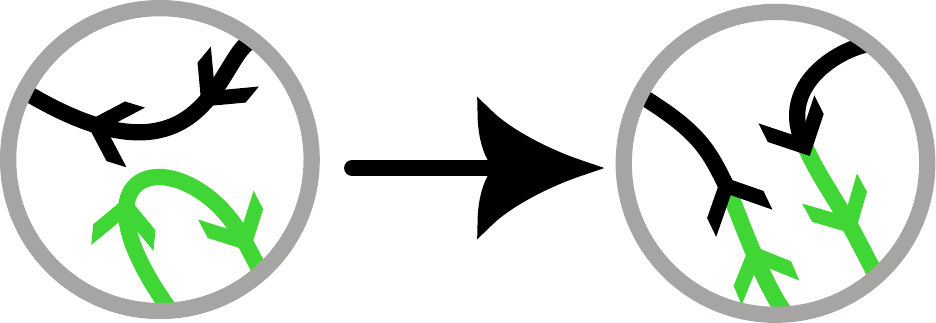}
    \end{minipage} \hspace{15pt}
    \begin{minipage}{0.30\textwidth}
        \centering
        \includegraphics[width=0.8\linewidth]{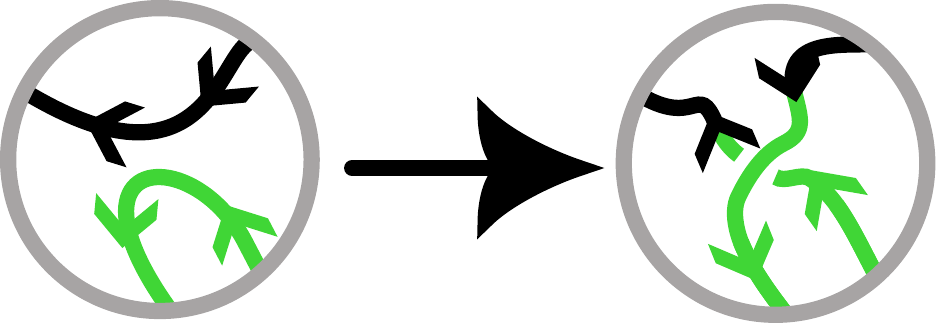} \\

        \vspace{5pt}

        \includegraphics[width=0.8\linewidth]{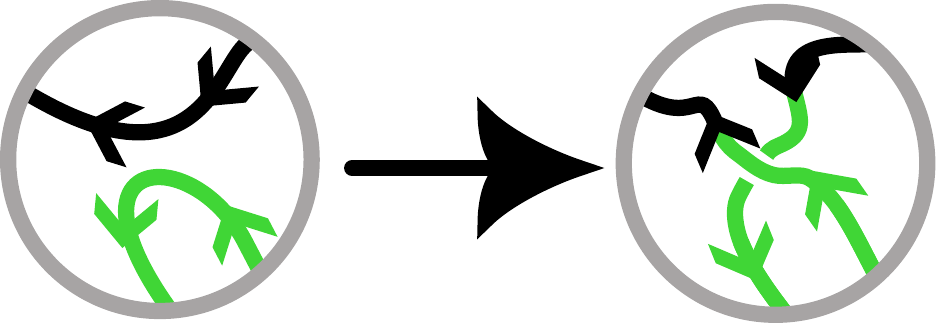}
    \end{minipage}}
    \hfill
    \hspace{30pt}
    \caption{Band move.
    The top strand (in black) is the target edge and the bottom one (in green) is the final part of the band before the band move takes place.
    Orientation is crucial. On the left, no twist is needed. On the right, we have two possible cases, which we can think of as performing a final Reidemeister~I move before the band move takes place. This Reidemeister~I move is right-handed in the top picture and left-handed in the bottom picture.}
    \label{fig:band_move_oriented}
\end{figure}

We can therefore play exactly the same game as before on a fixed diagram of $K$ (choose a strand to deform, slide it next to a target strand via Reidemeister~II moves, possibly interleaved with Reidemeister~I moves) but now we conclude each band with a band move rather than a crossing change.
A sequence of such band paths taking a diagram of $K$ to the unknot produces a slice surface for $K$ whose genus is an explicit upper bound on $\gen{K}$.

This unifies the two computations into a single combinatorial game on a fixed diagram of $K$. 
We choose a sequence of band paths, each ending either in a crossing change or in a band move, that simplifies the diagram to the unknot or an unlink, and try to minimise the upper bound we obtain.
This is exactly the setup for our agents in Section~\ref{sec:env}.

\begin{example}[Slice surface of genus one for the trefoil]
\label{ex:trefoil-slice}

As a concrete example, we study the case of the left-handed trefoil $K$.
Figures~\ref{fig:slice_example1} and~\ref{fig:slice_example2} show a \emph{slice movie} building a slice surface of genus one for the trefoil.
At time $t \in [0,1]$, we draw the portion $F^t$ of the surface $F$ that we have constructed so far, and above it the link $L_t = \partial F^t \setminus K$ (i.e., the ``bottom boundary'' with respect to the radial function on $B^4$).
While we represent $L_{t}$ with its actual link diagram, we represent the ``horizontal sections'' of $F^t$ as circles.

\begin{figure}[ht]
    \centering
    \begin{subfigure}[b]{0.17\textwidth}
        \centering
        \includegraphics[scale=0.14, valign=t]{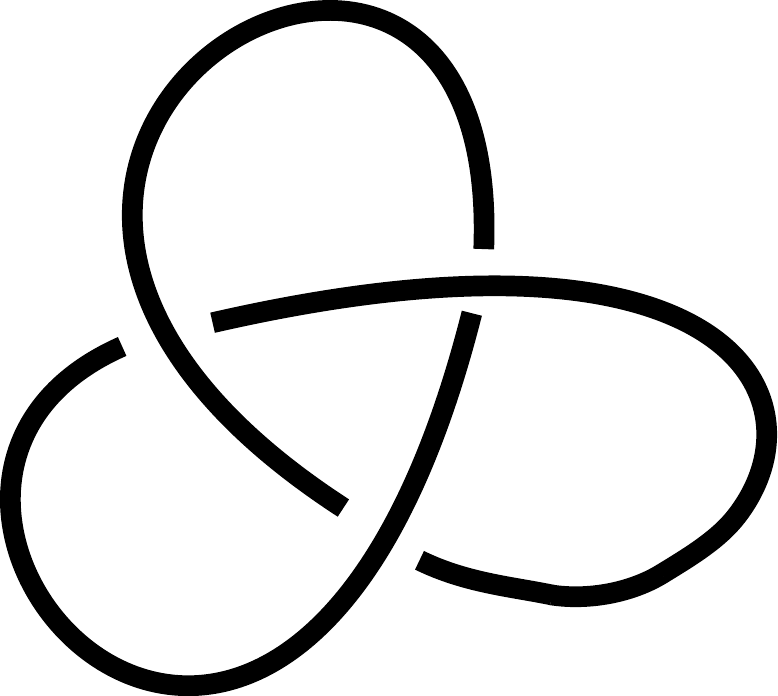}\rule[-4.7cm]{0pt}{0pt}
        \caption{Time $t_a = 0$}
        \label{subfig:slice1-a}
    \end{subfigure}
    \hspace{3pt}
    \begin{subfigure}[b]{0.17\textwidth}
        \centering
        \includegraphics[scale=0.14, valign=t]{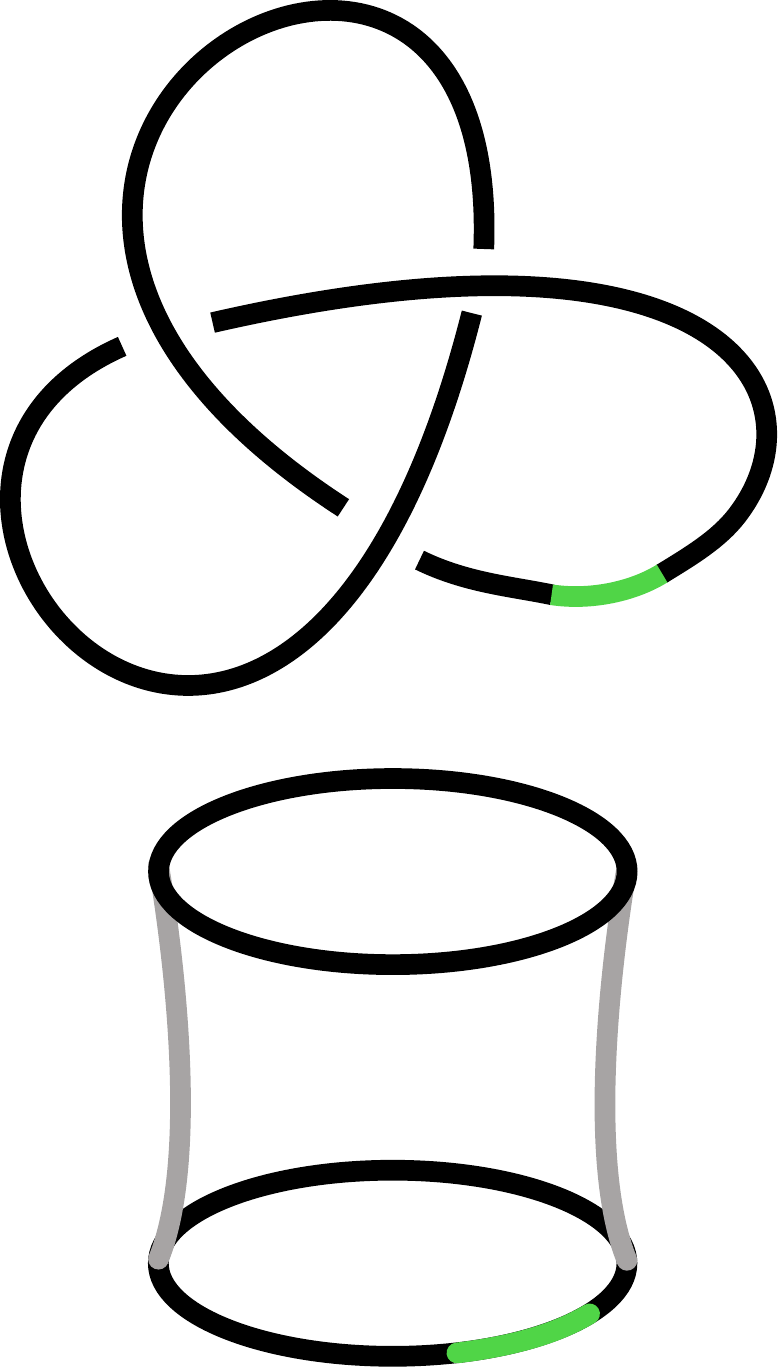}\rule[-4.7cm]{0pt}{0pt}
        \caption{Time $t_b > t_a$}
        \label{subfig:slice1-b}
    \end{subfigure}
    \hspace{3pt}
    \begin{subfigure}[b]{0.17\textwidth}
        \centering
        \includegraphics[scale=0.14, valign=t]{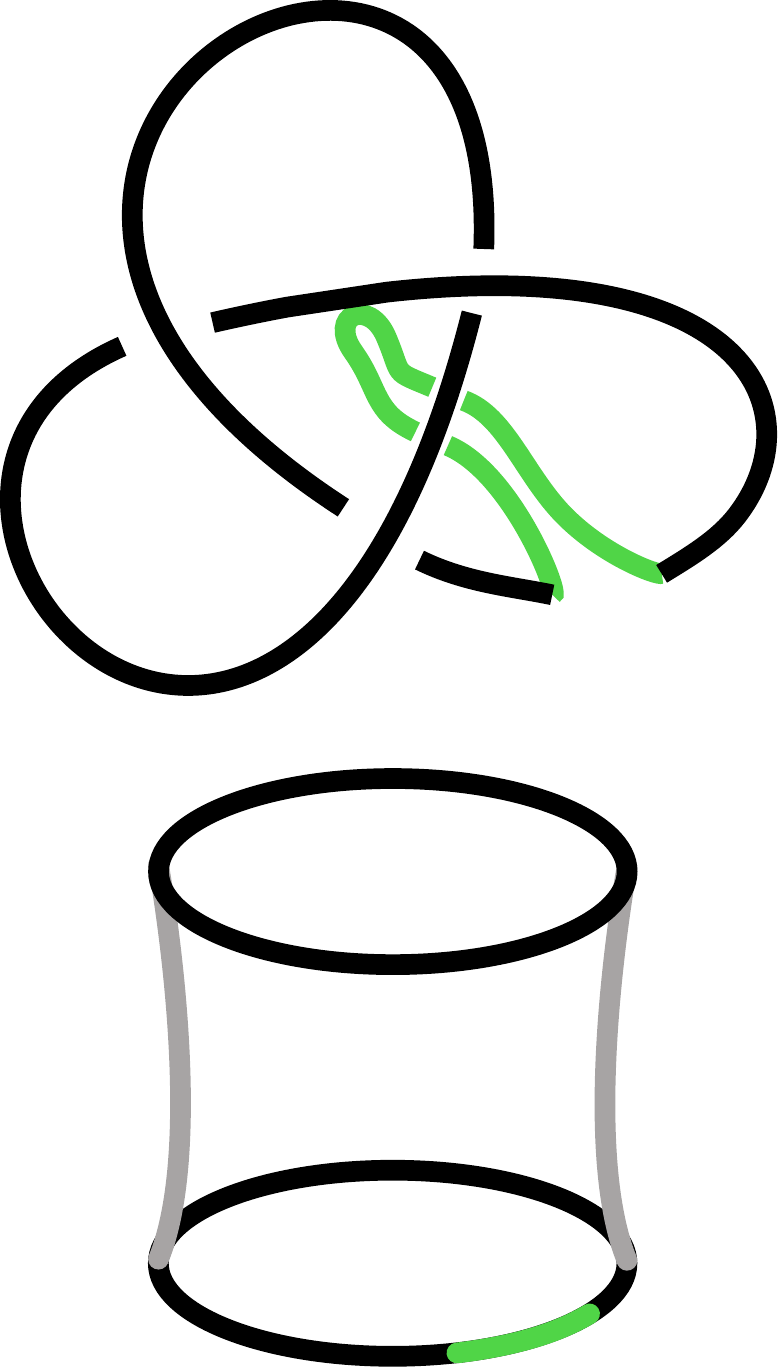}\rule[-4.7cm]{0pt}{0pt}
        \caption{Time $t_c > t_b$}
        \label{subfig:slice1-c}
    \end{subfigure}
    \hspace{0pt}
    \begin{subfigure}[b]{0.17\textwidth}
        \centering
        \includegraphics[scale=0.14, valign=t]{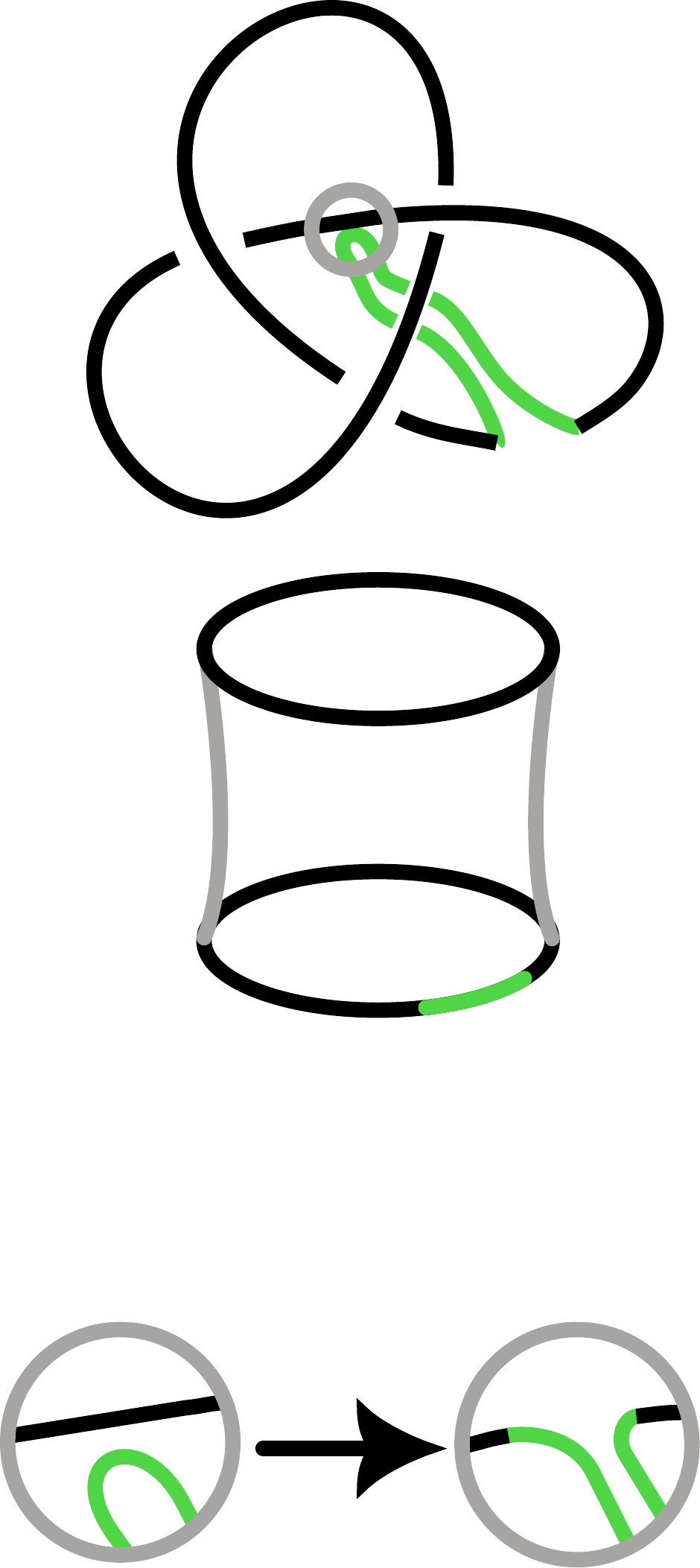}\rule[-4.7cm]{0pt}{0pt}
        \caption{Time $t_d > t_c$}
        \label{subfig:slice1-d}
    \end{subfigure}
    \hspace{6pt}
    \begin{subfigure}[b]{0.17\textwidth}
        \centering
        \includegraphics[scale=0.14, valign=t]{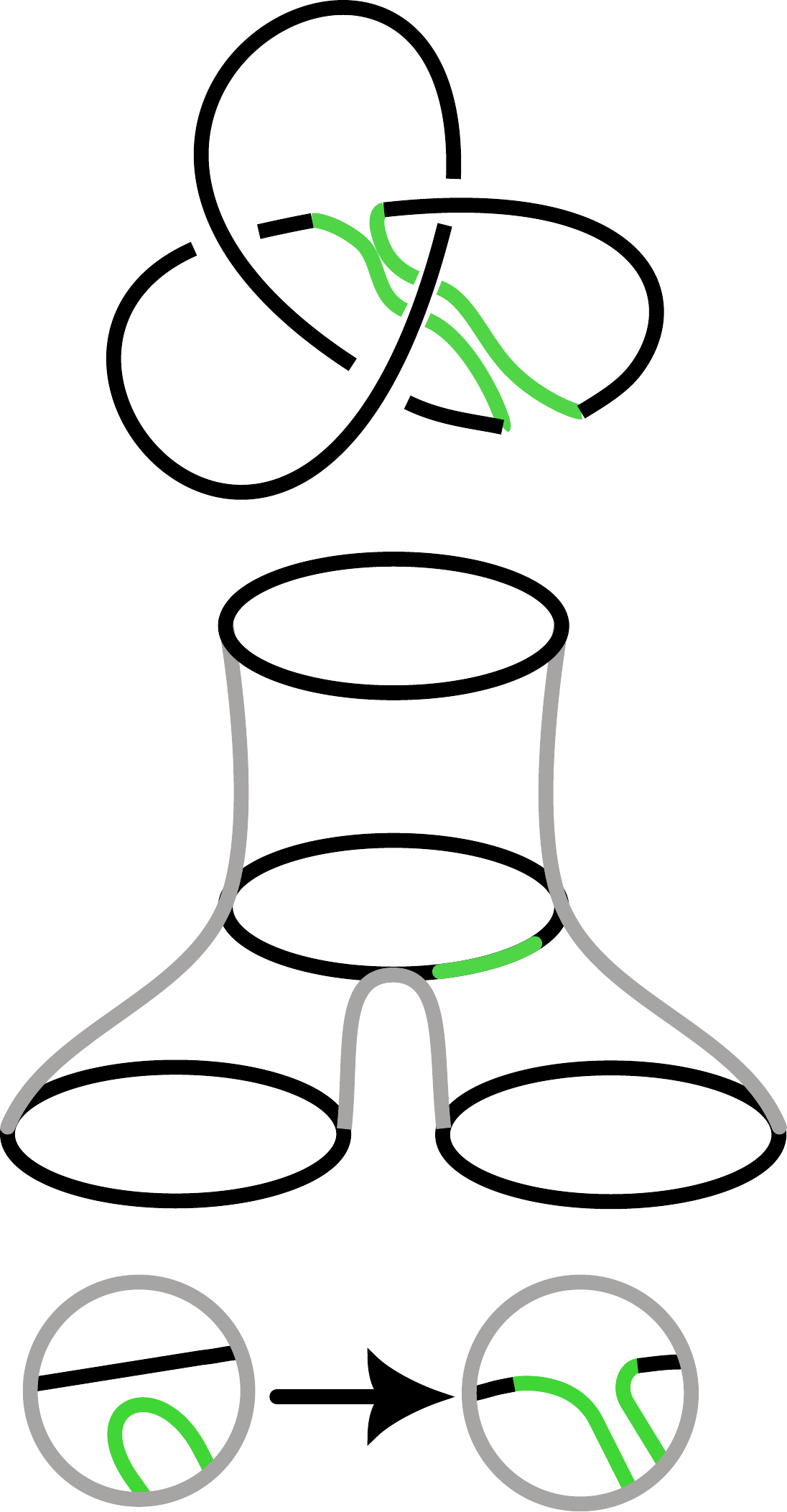}\rule[-4.7cm]{0pt}{0pt}
        \caption{Time $t_e > t_d$}
        \label{subfig:slice1-e}
    \end{subfigure}
    \caption{Building a slice surface for the trefoil: first part.}
    \label{fig:slice_example1}
\end{figure}

Let us start with the trefoil $K = L_{0}$ depicted in Figure~\ref{subfig:slice1-a}. Recall that $K$ lives in $\S^3$; i.e., the boundary of $B^4$.
We push the trefoil from $\S^3$ inside the ball $B^4$, obtaining at time $t_b$ a cylinder with the trefoil as its base; see Figure~\ref{subfig:slice1-b}.
We select a portion of $L_{t_b}$ highlighted in green in the figure.
We continue pushing down the link while deforming it with a Reidemeister~II move until we reach the state depicted in Figure~\ref{subfig:slice1-c}.
At this point, we can perform the \emph{band move} shown in Figure~\ref{subfig:slice1-d}, which corresponds to adding a saddle to $F^{t_d}$ and results in the link $L_{t_e}$ shown in Figure~\ref{subfig:slice1-e}.

\begin{figure}[ht]
    \centering
    \begin{subfigure}[b]{0.17\textwidth}
        \centering
        \includegraphics[scale=0.12, valign=t]{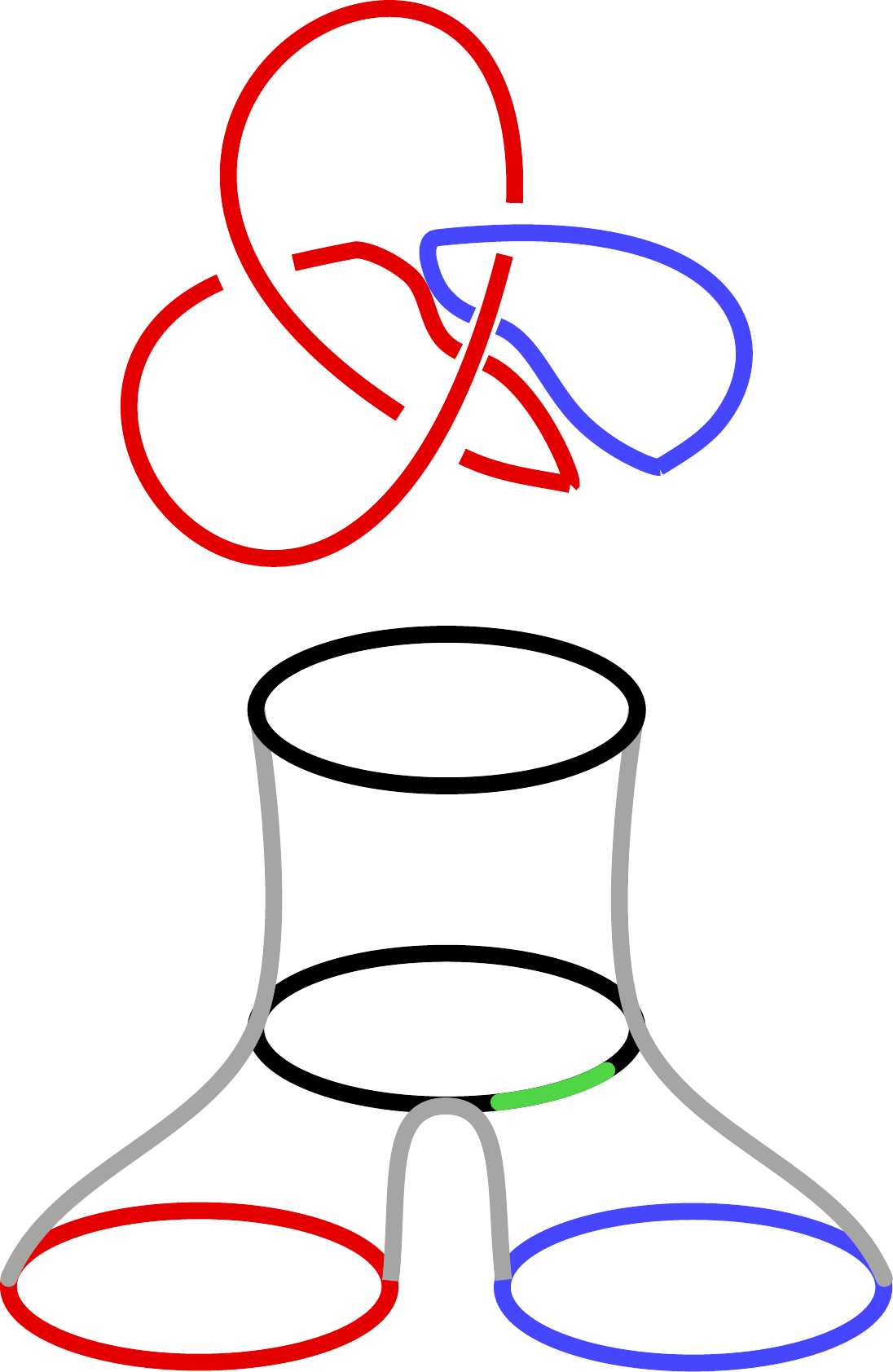}\rule[-5.45cm]{0pt}{0pt}
        \setcounter{subfigure}{5}
        \caption{Time $t_f > t_e$}
        \label{subfig:slice2-f}
    \end{subfigure}
    \hspace{5pt}
    \begin{subfigure}[b]{0.17\textwidth}
        \centering
        \includegraphics[scale=0.12, valign=t]{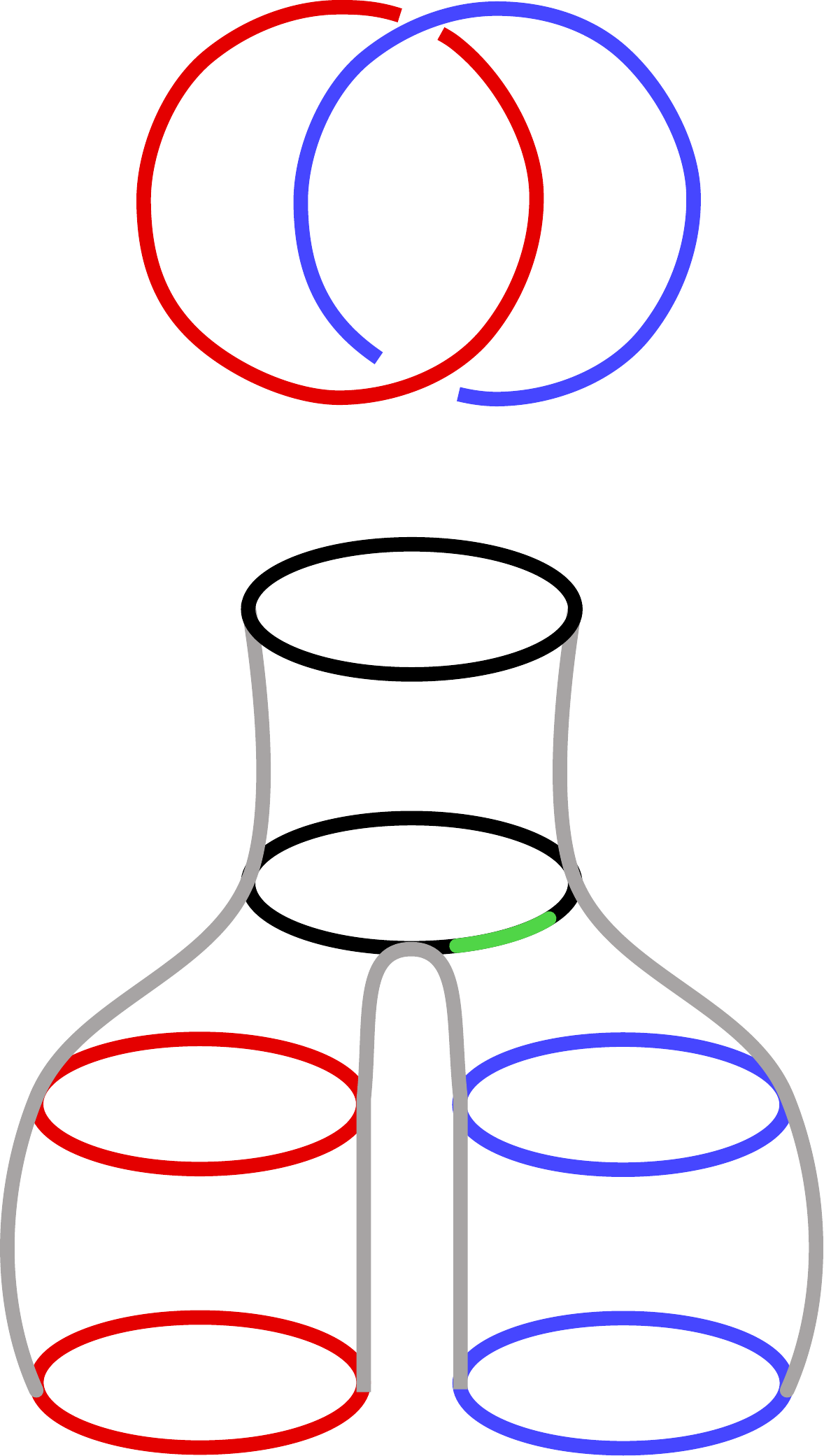}\rule[-5.45cm]{0pt}{0pt}
        \caption{Time $t_g > t_f$}
        \label{subfig:slice2-g}
    \end{subfigure}
    \hspace{5pt}
    \begin{subfigure}[b]{0.17\textwidth}
        \centering
        \includegraphics[scale=0.12, valign=t]{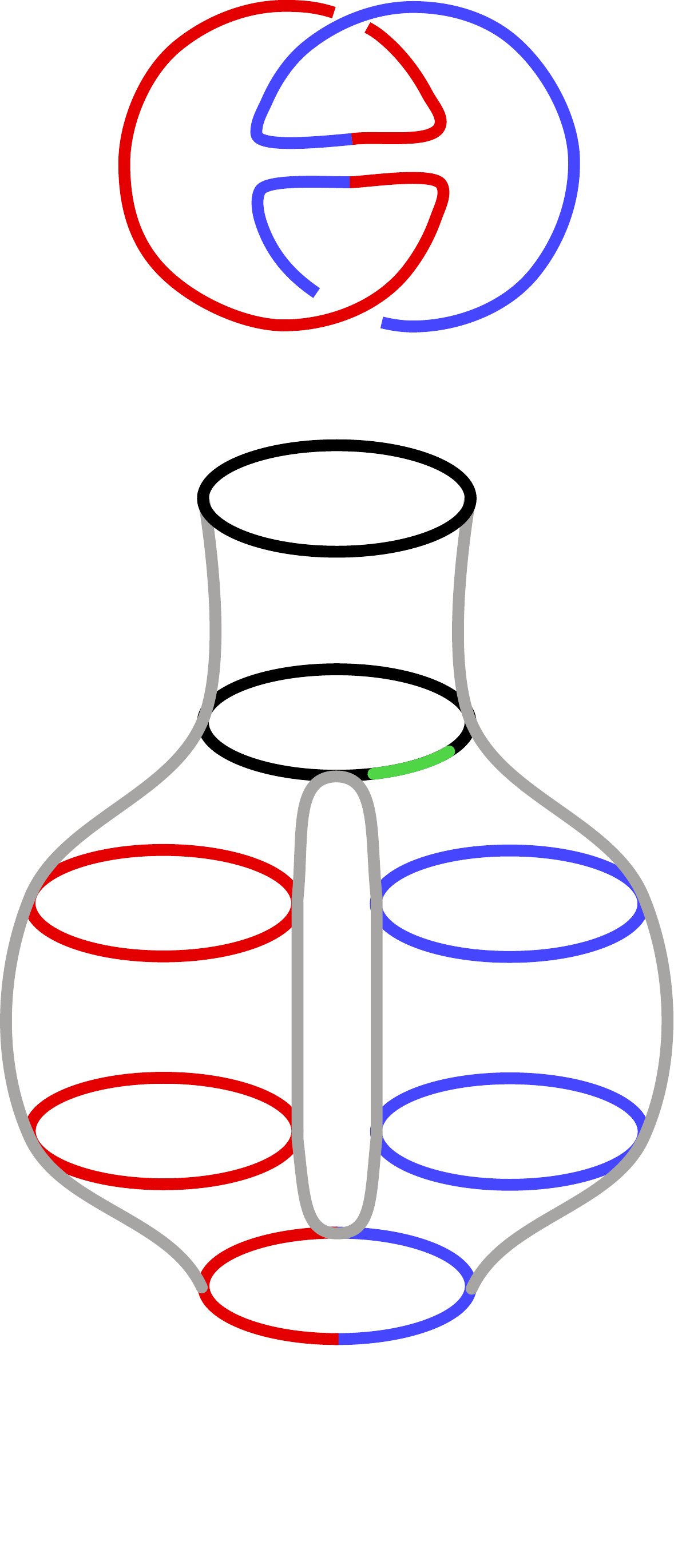}\rule[-5.45cm]{0pt}{0pt}
        \caption{Time $t_h > t_g$}
        \label{subfig:slice2-h}
    \end{subfigure}
    \hspace{5pt}
    \begin{subfigure}[b]{0.17\textwidth}
        \centering
        \includegraphics[scale=0.12, valign=t]{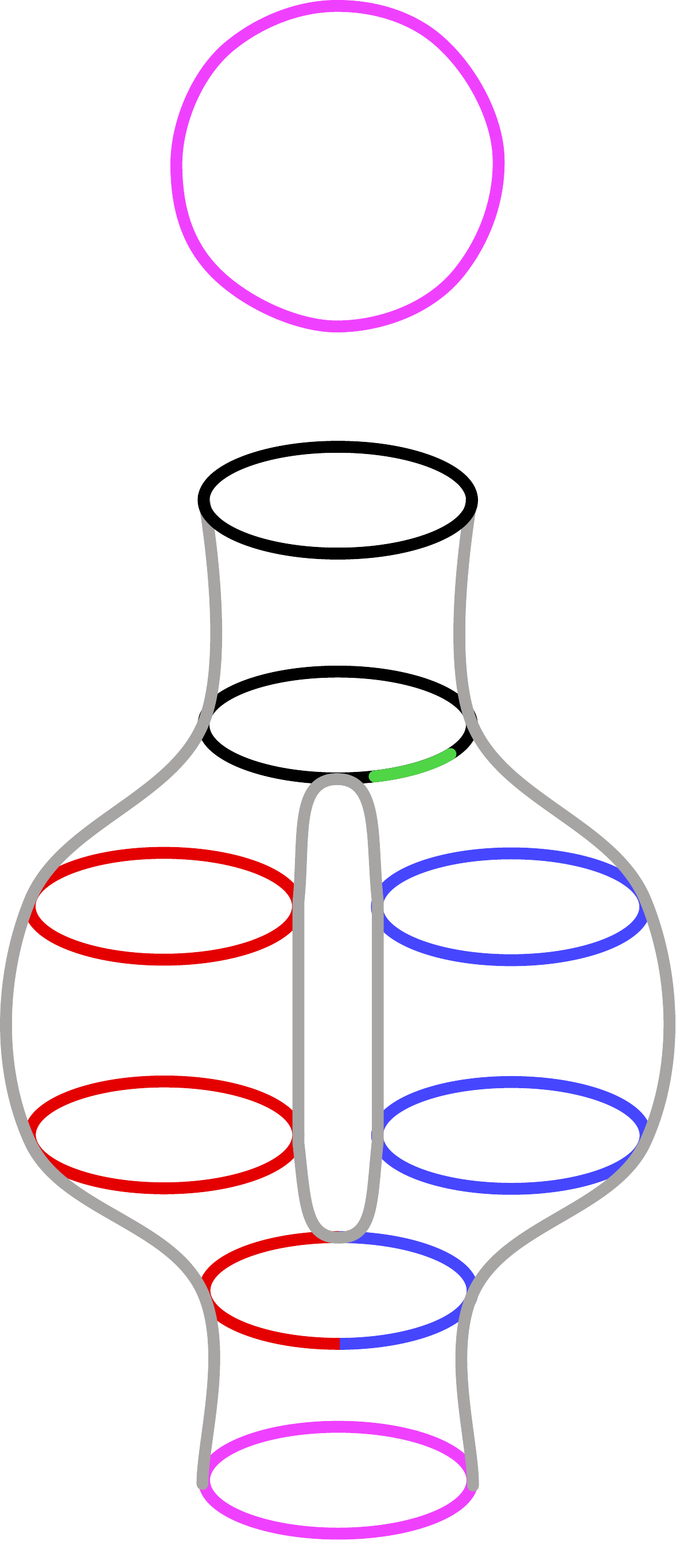}\rule[-5.45cm]{0pt}{0pt}
        \caption{Time $t_i > t_h$}
        \label{subfig:slice2-i}
    \end{subfigure}
    \hspace{5pt}
    \begin{subfigure}[b]{0.17\textwidth}
        \centering
        \includegraphics[scale=0.12, valign=t]{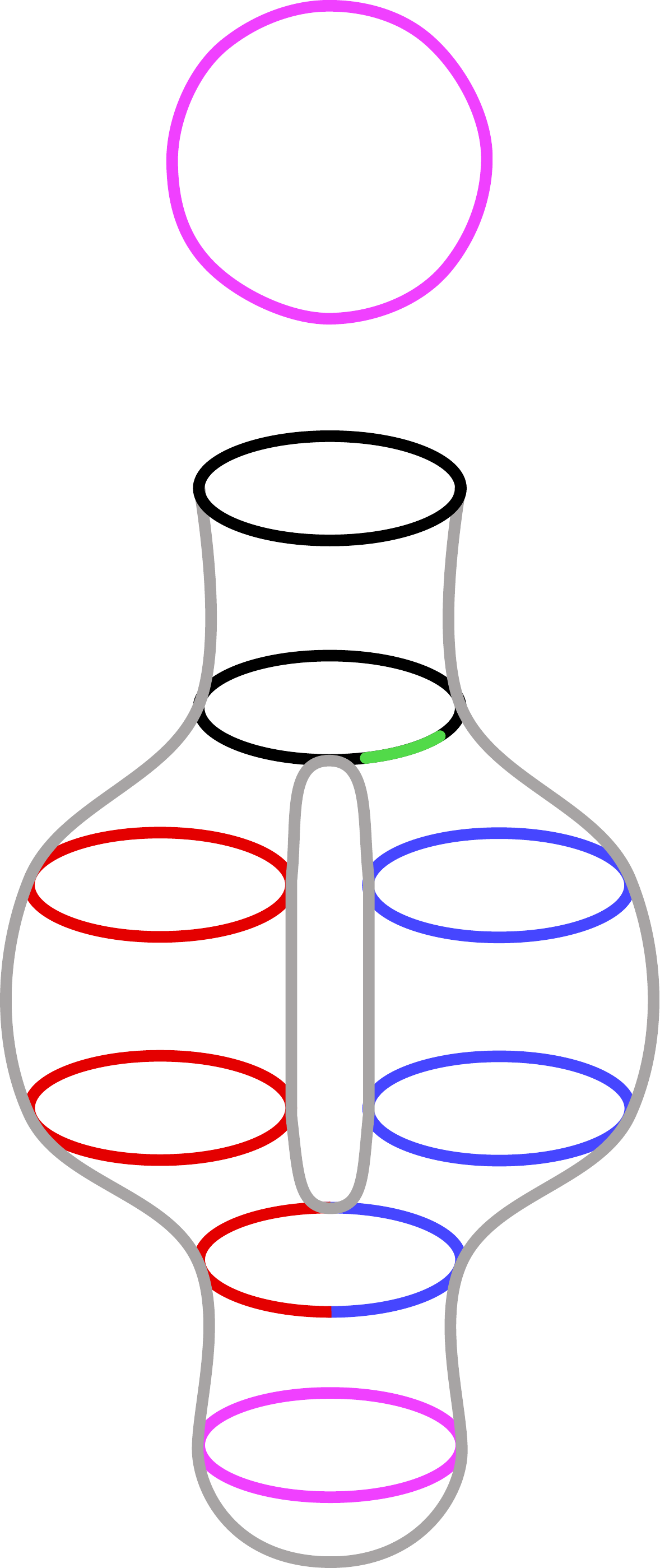}\rule[-5.45cm]{0pt}{0pt}
        \caption{Time $t_j > t_i$}
        \label{subfig:slice2-j}
    \end{subfigure}
    \caption{Building a slice surface for the trefoil: second part.}
    \label{fig:slice_example2}
\end{figure}

The link $L_{t_e}$ has two components that we colour in red and blue, as shown in Figure~\ref{subfig:slice2-f}.
As time increases, we simplify $L_{t_e}$ using a Reidemeister~II move and a Reidemeister~I move to obtain $L_{t_g}$, which is the standard diagram of the Hopf link in Figure~\ref{subfig:slice2-g}.
We now perform another band move on the link, which results in another saddle on the surface; see Figure~\ref{subfig:slice2-h}.
Note that this saddle fuses the two components into a knot $L_{t_h}$. We can simplify $L_{t_h}$ via two Reidemeister~I moves, obtaining $L_{t_i}$, the unknot shown in Figure~\ref{subfig:slice2-i}.
At this point, we can cap off the final boundary component by gluing a disc to it, obtaining a surface $F$ embedded in $B^4$ for which $\partial F = K$; this can be seen in Figure~\ref{subfig:slice2-j}. This shows that the slice genus of the trefoil is at most one, concluding our example.

\end{example}

\subsubsection{Weak and strong slice genera for links}
\label{subsubs:slice-genus-links}

There are two extensions of the notion of slice genus to links.

\begin{definition}[Weak and strong slice genera]
\label{def:weak-strong-genus}
    Let $L = \cup_{i=1}^{\ell} K_i$ be a link.
    A \emph{weak slice surface} for $L$ is a compact, connected, oriented surface $F \subset B^4$ such that $\partial F = L$.
    The \emph{weak slice genus} $\gen{L}$ is the minimal genus of a weak slice surface for $L$.

    A \emph{strong slice surface} for $L$ is a compact, oriented surface $F \subset B^4$ with $\ell$ disjoint components $F_1, \dots, F_\ell$ such that $\partial F_i = K_i$ for $i \in \{1,\dots,\ell\}$.
    The \emph{strong slice genus} $\sgen{L}$ is the minimal \emph{total} genus of a strong slice surface for $L$; i.e.,
    \[
    \sgen{L} = \min \left\{ \sum_{i=1}^{\ell} g(F_i) : F = \bigsqcup_{i=1}^{\ell} F_i \subset B^4 \text{ is a strong slice surface for } L \right\}.
    \]
\end{definition}

We could allow disconnected surfaces in the definition of the weak slice genus without changing $\gen{L}$, since one can take the connected sum of the components of a disconnected slice surface without changing the total genus.
In the literature, the term ``slice genus'' is used for either the weak or the strong slice genus. 
In this paper, we refer to the weak slice genus as the ``slice genus''.

Note that $\gen{L} \leq \sgen{L}$ for every link $L$.
Indeed, if we take a strong slice surface for $L$ and remove the discs we attached to the unlink ``on the bottom'', we can attach $\ell-1$ bands between unknots from different connected components and make it connected. This process does not change the boundary or the total genus.
Not every link bounds a strong slice surface. Indeed, if $\bigsqcup_{i=1}^\ell F_i$ is a strong slice surface for $L$, then $\lk{K_i, K_j} = F_i \cdot F_j = 0$ for any $i$, $j \in \{1,\dots,\ell\}$ with $i \neq j$; see Cavallo~\cite[Lemma~5.1]{cavallo2015links}.
We call a link $L$ \emph{algebraically split} if $\lk{K_i, K_j} = 0$ for all $i \neq j$, and we set $\sgen{L} = + \infty$ for a link $L$ that is not algebraically split (using $\min(\emptyset) = +\infty$).
With this convention, $\sgen{L}$ is defined for every link, and it is finite exactly for the algebraically split ones. If $L'$ is obtained from $L$ by reversing the orientation of a component, then $\sgen{L} = \sgen{L'}$; see \cite[Proposition~5.2]{cavallo2015links}.

\subsubsection{Ribbon surfaces and the slice--ribbon conjecture}
\label{subsubs:ribbon}

As we watch  a movie as in Section~\ref{subs:slice-genus}, only two kinds of moves change the topology of the surface we are building: a band move, which adds a saddle, and the birth of a new unknotted component, which adds a local maximum of the radial function on $B^4$.
Forbidding the second kind singles out a smaller class of surfaces and therefore yields a possibly larger invariant.

\begin{definition}[Ribbon surface and ribbon genus]
\label{def:ribbon-genus}
    A \emph{ribbon surface} for $K$ is a slice surface $F \subset B^4$ for $K$ with no local maxima with respect to the radial function of $B^4$.
    The \emph{ribbon genus} $\rgen{K}$ of $K$ is the minimal genus of a ribbon surface for $K$.
    \smallskip
    A knot with $\rgen{K} = 0$ is called a \emph{ribbon knot}; equivalently, it bounds a \emph{ribbon disc} in $B^4$.
\end{definition}

Every ribbon surface is, in particular, a slice surface, so the minimum defining $\gen{K}$ is taken over a larger set, and
\[
\gen{K} \leq \rgen{K}.
\]
Whether this inequality can be strict when $\gen{K}=0$ is the content of a long-standing open problem, raised by Fox in 1961.

\begin{conjecture}[Slice--ribbon conjecture, {\cite[Problem 25]{fox1961problems}}]
\label{conj:slice-ribbon}
    Every slice knot is ribbon; that is, $\gen{K} = 0$ implies $\rgen{K} = 0$.
\end{conjecture}

The conjecture is still open. Families of potential counterexamples have been proposed, but none of them have been confirmed.
See~\cite[Section~8]{gompf2010sliceribbon} for details. The higher-genus version of the conjecture is also open, which states that $\gen{K} = \rgen{K}$ for every knot $K$.

Definition~\ref{def:ribbon-genus} of ribbon surfaces for knots extends to links exactly as Definition~\ref{def:weak-strong-genus} does, by requiring the ribbon surface to be connected in the weak case and to have one component per link component in the strong case. We write $\rgen{L}$ and $\rsgen{L}$ for the two resulting invariants.
We say that the link $L$ is \emph{ribbon} if $\rsgen{L} = 0$; i.e., when its components bound disjoint ribbon discs.

Conjecture~\ref{conj:slice-ribbon} has a weak and a strong version for links, and both are open.
The strong genus-zero statement, that every slice link is ribbon, is the one usually meant in the literature (e.g., in~\cite[Section~8]{gompf2010sliceribbon}). It has been verified by Aceto, Kim, Park, and Ray~\cite{aceto2021pretzel} for $4$-stranded $2$-component pretzel links.

Upper bounds on both $\rgen$ and $\rsgen$ are computed by the agents of Section~\ref{sec:env}, where the ribbon frameworks are obtained from the slice ones by removing the birth of an unknotted component from the set of allowed moves.

\subsubsection{Computation of the slice genus from a movie}
\label{subsubs:computation-slice-genus-from-movie}

We explain how to compute the genus of a slice surface from a movie.
For a compact, connected, orientable surface $F$ with $\abs{\partial F}$ boundary components,
\begin{equation}
    \label{eq:genus-bound-surface}
    g(F) = \frac{2 - \chi(F) - \abs{\partial F}}{2};
\end{equation}
see~\cite[Section 1.1.1]{farb2011mappingclass}.
Start with an $\ell$-component link $L$ for which we want to build the slice surface, and let $F^t$ be the surface we have built until time $t$.
Before any births or band attachments, $\chi(F^t) = 0$, as $F^t$ is a disjoint union of annuli.
Assume we have reached the end of the slice movie, obtaining a compact, connected, orientable surface $F$, which is a cobordism between $L$ and the $m$-component unlink. 
Therefore, the surface has $\ell+m$ boundary components.
Each birth increases the Euler characteristic by one, while each band attachment decreases it by one. Denoting by $b$ the number of births and by $\beta$ the number of band attachments, the Euler characteristic of the surface is $b - \beta$, and equation~\eqref{eq:genus-bound-surface} gives
\begin{equation}\label{eq:genus-link-slice-surf}
    g(F) = \frac{2 - (b - \beta) - \left( \ell + m \right)}{2}.
\end{equation}
By capping the $m$ components of the unlink, we increase the Euler characteristic by $m$ and remove $m$ boundary components; hence, the genus remains the same.

We can actually use the above formula during an episode to compute the ``partial genus'' of the surface we are building.
This is useful for understanding ``how much genus has been created so far'', since this quantity can only increase as the episode continues.
When the intermediate surface $F^t$ is disconnected, with components $F_1^t, \dots, F_{c_t}^t$, we define its total genus as $g(F^t) := \sum_{i=1}^{c_t} g(F_i^t)$. 
By equation~\eqref{eq:genus-link-slice-surf}, if $L_t$ has $m_t$ components and we have encountered $b_t$ births and $\beta_t$ band attachments by time $t$, then
\begin{equation}\label{eq:genus}
g(F^t) = \frac{2c_t + \beta_t - b_t - \ell - m_t}{2}.
\end{equation}

Let $K_1,\dots,K_\ell$ be the components of the link $L$. A strong slice surface $F$ for $L$ is disconnected by definition. Write $F = \sqcup_{j=1}^\ell F_j$, where the $F_j$ are the connected components of the surface and $\partial F_j = K_j$. We compute $g(F_j)$ for each $j \in \{1,\dots,\ell\}$ using equation~\eqref{eq:genus-link-slice-surf}, and the total genus of $F$ as 
\begin{equation}
    g(F) := \sum_{j=1}^\ell g(F_j) = \frac{2\ell + \beta - b - \ell - m}{2}.
\end{equation}
Note that a movie of $F$ fuses every unknot component that is born into a surface component that intersects $L$. In our implementation, we track every unknot component born separately and discard it if it does not fuse into a component that intersects $L$ by the end of the episode.

\subsection{RL basics}
\label{subs:rl-basics}

Reinforcement learning (RL) is a machine learning framework for sequential decision problems.
The first ideas behind RL, the principle of learning from rewards and punishments, go back to the 1950s, but the most striking demonstrations are much more recent.
In the 1990s, \emph{TD-Gammon}~\cite{TD-gammon} reached near world-class play in backgammon, the first sign that this kind of training could lead to genuine expertise.
In 2016, \emph{AlphaGo}~\cite{Silver2016} defeated world champion Lee Sedol at Go, a game long considered too complex for computers to play at the highest level.
In 2017, \emph{AlphaGo Zero}~\cite{Silver2017} learned Go from scratch, without ever looking at a human game, suggesting that the method can discover its own playing style entirely from experience.

The basic picture is that of an \emph{agent} that interacts with an environment by repeatedly choosing actions. Good moves are rewarded with a positive score, while bad moves receive a negative one, and the agent's goal is to maximise its total score over time.
Through trial and error, an initially random agent gradually learns strategies that lead to higher cumulative rewards.

The interaction is organised as a loop, sketched in Figure~\ref{fig:rl_loop}.
From a given state, the agent observes the environment and selects one of the available actions according to its current strategy.
The environment then carries out the action and returns both a reward, quantifying how good that move was, and a new state, which plays the role of ``present state'' in the following iteration. The process repeats until the episode ends.
The rewards collected along the way are what the agent uses to revise its strategy between episodes, rather than a rule for choosing the next move.

\begin{figure}[ht]
    \centering
    \includegraphics[width=\linewidth]{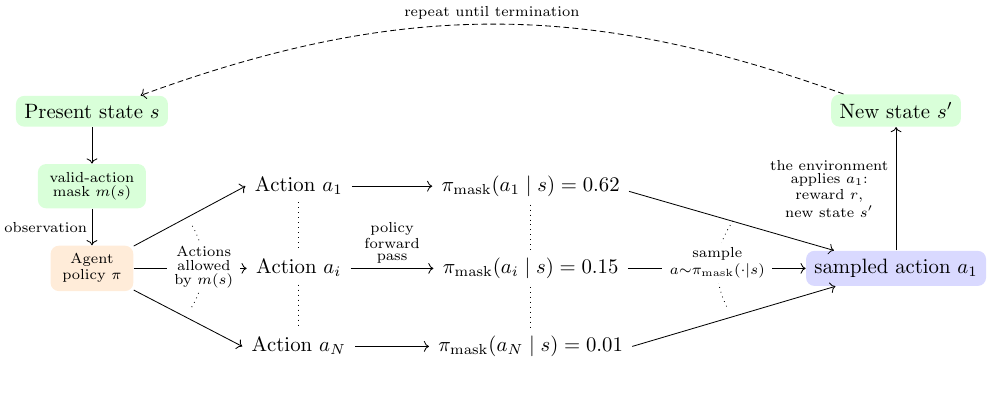}
    \caption{Schematic of the RL loop.}
    \label{fig:rl_loop}
\end{figure}

To make this picture more precise, the agent's strategy is encoded by a \emph{policy} $\pi(a \mid s)$, which is a typically probabilistic prescription assigning to each state $s$ a distribution over the available actions $a$.
The agent does \emph{not} simply pick the action with the highest immediate reward. Rather, its goal is to maximise the \emph{scalar discounted return}
\[
\mathbb{E}_{\pi}\!\left[\, \sum_{t \geq 0} \gamma^{t} \, r_t \,\right]
\]
collected over an entire episode, where $r_t$ is the reward received at time $t$, and $\gamma \in (0,1]$ is a discount factor.
When $\gamma < 1$, it weights immediate rewards more than future ones.
To reason about these long-term consequences, the agent typically estimates a \emph{value function} $V^{\pi}(s)$ (or, alternatively, an \emph{action-value function} $Q^{\pi}(s,a)$), defined as the expected cumulative reward obtained by starting from $s$ (and taking action $a$) and following the policy $\pi$ from then on.
We instead set $\gamma = 1$.
Since the setting is episodic with a maximum number of actions, an episode is guaranteed to terminate, and this value for $\gamma$ is allowed.
Defining $\gamma = 1$ permits creating a reward function that always gives a higher return for a smaller invariant value found.
More details can be found in Section~\ref{subs:env-reward}.
During evaluation, the smallest invariant found is retained separately.

A core challenge in reinforcement learning is the \emph{exploration versus exploitation trade-off}: the agent must \emph{exploit} the actions that currently look best while still \emph{exploring} less familiar ones to avoid getting stuck in suboptimal habits.
The policy is not given in advance, but is \emph{learned} over many episodes.
In the \emph{deep reinforcement learning} setting used in this work, $\pi$ and $V^{\pi}$ are represented by \emph{neural networks} and updated via gradient-based optimisation.
This is the approach that powered AlphaGo and underlies the work described in the rest of this article.

\paragraph{Using RL to compute the unknotting number and the slice genus}
\label{par:rl-and-knot-theory}
We can phrase the computation of the unknotting number and the slice genus as a game.
The RL framework is set up as follows.

\begin{itemize}
    \item \emph{Episode}: starting from a link and trying to reach the unlink.
    \item \emph{State}: the link diagram, the band path being built, and additional data about the episode in progress.
    \item \emph{Action}: extending a band path, performing a crossing change (for the unknotting number), or a band move (for the slice genus).
    \item \emph{Reward}: the smaller the resulting unknotting number or slice genus, the higher the score.
\end{itemize}

\section{The environment}
\label{sec:env}

In this section, we present the environment we have built.
The environment encodes the game the agent will play via states, actions, rewards, and terminal conditions. 
We therefore deliberately describe it without any reference to the agent, which is nothing but a choice of a policy $\pi$ inside the environment.
Indeed, the same environment is shared by the three agents of Section~\ref{sec:agents}: the RL policy, the random walker (RW), and the mixed RL$+$RW agent.

Before any episode takes place, we need to fix two fundamental parameters:
\begin{itemize}
    \item The \textbf{framework} selects the set of moves, the terminal condition, and the invariant that the episode computes (Section~\ref{subs:env-frameworks}).
    \item The \textbf{shape} $(c_{\max}, \ell_{\max})$ is a pair of upper bounds on the number of crossings and link components a diagram is allowed to have during an episode. 
    It is what makes all observations live in a common, fixed-dimensional space.
    The default we use is $(c_{\max}, \ell_{\max}) = (100,16)$.
\end{itemize}
Throughout this section, we denote by $c$ the number of crossings of the current diagram and by $\ell$ the number of link components.
An outline of an episode is given in Appendix~\ref{app:episode} alongside a worked example.

\paragraph{Implementation.}
The environment is implemented as a hybrid Nim/Python system. 
The diagram modifications, the construction of the band, the surface bookkeeping, the invariant features, and the random walker are written in Nim and compiled into shared libraries, while the RL layer (the Gymnasium environments, the datasets, and the experiment drivers) is written in Python.
Link manipulation and simplification rely on \texttt{spherogram-nim}~\cite{dai2025spherogramnim}, the Nim port of \texttt{spherogram} (the SnapPy module~\cite{culler2024snappy} that manages links).

\subsection{Frameworks}
\label{subs:env-frameworks}

With this environment, we can play five different games by slightly changing the rules and the move set.
These are: unknotting, ribbon, slice, strong ribbon, and strong slice.

\paragraph{Unknotting.}
This is the framework in which we can apply crossing changes to a link diagram.
A crossing change is applied via a band, which starts at an edge of the diagram and may travel over or under other edges.
The crossing change itself takes place when the band ``passes through'' a final edge.
A won episode returns an upper bound on the unknotting number of the input link.

\paragraph{Ribbon and slice.}
These are the frameworks for finding upper bounds on the ribbon genus and the slice genus of a link.
They operate almost identically to the unknotting framework, but we apply a band move instead of a crossing change when we end the band.
We equip every link component with a \emph{connected component label} (CCL for short).
Two current link components have the same CCL precisely when they belong to the same connected component of the evolving surface $F^t$.
This is used to keep track of the component count in the genus computation of equation~\eqref{eq:genus}.
In the slice frameworks, it is also possible, before starting a band, to add a crossingless unknotted component to the diagram; i.e., a local maximum of the radial function on the surface. 
The ribbon frameworks forbid this. 
Internally, we apply a reverse Reidemeister~II move to the crossingless unknot born and a neighbouring edge in order to provide crossing strands along which band sums can later be performed.
The unknot also receives a new provisional CCL marked as containing no original link component. 
When a fusion joins different CCL classes, these classes and their associated surface bookkeeping records (see Section~\ref{subs:env-mdp}) are merged. 
The implementation retains the larger CCL for the merged class. 
If a crossingless unknot born reaches the terminal state of an episode, it is discarded. 

\paragraph{The strong ribbon and strong slice frameworks.}
The strong ribbon and strong slice frameworks find upper bounds on the strong ribbon genus and the strong slice genus, respectively.
In the strong ribbon framework, a band may only join strands with the same CCL. 
In the strong slice framework, it may also join different CCL classes provided at least one contains no original link component. 
Thus, no connected surface component can contain two different original link components.
In the strong slice case, where \texttt{unknot\_birth} is allowed, the crossingless unknotted components introduced are managed as in the slice framework. 
These frameworks apply only to algebraically split links, and links outside this family are rejected when the episode is created.

\subsection{States and actions}
\label{subs:env-mdp}

\paragraph{State.}
\label{par:env-mdp-state}
A state of the game can be thought of as a tuple $(D, B, \FM)$, where:
\begin{itemize}
    \item $D$ is the working diagram; i.e., its PD code, together with the vector of crossing signs, the number of split unknotted components, and the CCL data for each link component (relevant only in the genus frameworks).
    If the diagram is split, then it is made into a connected diagram via inverse Reidemeister II moves.
    \item $B$ is the partially built band, represented by a list of quadruples 
    \[
    [\mathrm{crossing},\mathrm{strand},\mathrm{over/under},\mathrm{twists}].
    \] 
    The pair $[\mathrm{crossing},\mathrm{strand}]$ encodes an edge of the diagram and the side into which the band extends.
    The entry $\mathrm{over/under}$ specifies whether the band goes over or under this edge.
    The fourth entry records the signed number of half-twists inserted in the band segment following that crossing strand. 
    Thus, the internal state records both the total number of twists and their placement along the band.
    \item $\FM$ is the framework metadata.
    For all frameworks, it stores the number of crossings of the initial diagram of an episode. 
    In the unknotting framework, it stores the number of crossing changes performed.
    In the genus frameworks, it is a table indexed by CCL used for surface bookkeeping. 
    For each surface component, it records the number of bands attached to it, the number of initial link components (including unknots introduced by birth) and the number of current link components belonging to it,  and the number of original link components it contains (newborn unknots excluded).
\end{itemize}

\paragraph{Actions.}
An action is encoded as a string that can be of three different forms:
\begin{itemize}
    \item \texttt{ccw\_twist}/\texttt{cw\_twist}: insert a counter-clockwise (resp.\ clockwise) half-twist into the band;
    \item \texttt{unknot\_birth}: add a new unlinked trivial component (slice frameworks only);
    \item \texttt{C\{c\}S\{i\}\_\{type\}}: here $(c,i)$ identifies a crossing strand and \texttt{\{type\}} is one of the following:
        \begin{itemize}
            \item \texttt{start}: begin a band at the crossing strand $(c,i)$;
            \item \texttt{over}/\texttt{under}: route the band into the next face, passing over (resp.\ under) the strand $(c,i)$;
            \item \texttt{end1}/\texttt{end2}: in the unknotting framework, these describe whether the band goes over or under the final strand; in the genus frameworks, they specify whether we attach the band at $(c,i)$ with a positive or negative twist inserted just before closing it;
            \item \texttt{end0}: attach the band at $(c,i)$ with no extra twist (genus frameworks only).
        \end{itemize}
\end{itemize}
Note that only the \texttt{end} actions and \texttt{unknot\_birth} modify the link. An \texttt{end} builds the new PD code and the new crossing signs, and then calls a simplification routine.
All the other actions only extend the band path.

\paragraph{Size of the action space.}
For each of the $c_{\max}$ crossings and each of their $4$ strands, the environment emits \texttt{start}, \texttt{over}, \texttt{under}, \texttt{end1}, \texttt{end2}, plus the two global twist actions: $20\,c_{\max} + 2$ actions.
Genus frameworks add one \texttt{end0} per crossing strand, giving $24\,c_{\max}+2$, and the slice frameworks append \texttt{unknot\_birth} as the last index.
At $c_{\max} = 100$, this is $2002$, $2402$, and $2403$ discrete actions, respectively.

\paragraph{Legality and masking.}
Most of these actions are illegal in any given state, and the illegal ones are \emph{masked out} before the agent selects an action so that a policy never assigns a positive probability to an illegal move.
Illegal actions consist of:
\begin{itemize}
    \item starting a band while one is already in progress, or on a crossing strand that does not exist;
    \item routing the band over or under a strand that does not lie on the band's current face;
    \item re-crossing an already visited edge;
    \item reaching a position admitting no legal continuation;
    \item mixing twists of opposite signs within the same band or exceeding the maximum allowed number of twists \texttt{max\_twists} (we use $\texttt{max\_twists} = 5$ for the RL agent);
    \item ending a band before it has started;
    \item ending a band on a strand not reachable from the current face, or, in the strong frameworks, joining different CCL classes that both contain an original link component;
    \item performing a band move that is incompatible with the orientations;
    \item giving birth to an unknot outside the slice frameworks or while a band is in progress.
\end{itemize}
In the strong slice framework, a surface component containing no original link component may merge with another CCL class, subject to the other legality conditions. 
This remains possible after earlier fusions with other birth components.

\subsection{The observation}
\label{subs:env-obs}

The observation is a dictionary of named NumPy arrays of type float64, whose shapes depend only on $(c_{\max},\ell_{\max})$. 
Observation values are not clipped.
Each array is represented by a Gymnasium Box with unbounded numerical limits and a matching dtype. 
The value $-1000$ is used both as raw padding and as a sentinel for quantities that cannot be computed on a given diagram.
The environment replaces raw padding by $0$, except in \texttt{band\_matrix}, where it uses $-3$.
Unavailable invariant values are also replaced by $0$.
Each invariant that may be unavailable has a flag equal to $0$ when its value is available and $1$ otherwise.
The observation common to all the frameworks consists of the thirteen entries of Table~\ref{tab:observation}, which at $(c_{\max},\ell_{\max}) = (100,16)$ amount to $1923$ floats.
The base observation for the unknotting framework additionally carries the number of crossing changes performed, for a total of $1924$ floats.
In the genus frameworks, the base observation includes the last three entries of Table~\ref{tab:observation-specific}, for a total of $2020$ floats.
Note that some of these tensors are intentionally made redundant in order to make the base observation as explicit as possible for the agent.

\begin{table}[ht]
    \centering
    \small
    \setlength{\tabcolsep}{4pt}
    \renewcommand{\arraystretch}{1.2}
    \renewcommand{\tabularxcolumn}[1]{m{#1}}
    \begin{tabularx}{\textwidth}{|l|c|>{\raggedright\arraybackslash}X|}
        \hline
        Tensor Name & Tensor Shape & Tensor Content \\
        \hline
        \texttt{PD\_code}                     & $c_{\max} \times 4$       & PD code of the current diagram \\ \greyrule
        \texttt{crossing\_signs}              & $c_{\max}$                & sign of each crossing \\ \greyrule
        \texttt{writhe}                       & $1$                       & writhe of the diagram \\ \greyrule
        \texttt{gauss\_code}                  & $2c_{\max} + \ell_{\max}$ & Gauss code \\ \greyrule
        \texttt{num\_link\_components}        & $1$                       & number of link components that have crossings\\ \greyrule
        \texttt{unlinked\_unknot\_components} & $1$                       & number of crossingless unknot components \\ \greyrule
        \texttt{num\_connected\_components}   & $1$                       & number of connected components of the diagram containing crossings, counted before the reverse Reidemeister~II moves that connect them \\ \greyrule
        \texttt{band\_path}                   & $2c_{\max} \times 3$      & the band built so far: each row is a triple (crossing, strand, over/under), where the first two entries determine a crossing strand and the last one is the sign of the action (0 for start, $+1$ if the band goes over the strand, $-1$ if it goes under). \\ \greyrule
        \texttt{total\_twists}                & $1$                       & signed number of twists in the current band \\ \greyrule
        \texttt{band\_matrix}                 & $c_{\max} \times 4$       & the band path re-indexed per crossing strand: each position corresponds to a crossing strand and each entry is $0$/$1$/$-1$ if that strand lies on the band, according to its start/over/under sign, $-2$ if the crossing exists but the strand was never visited, and $-3$ if the crossing strand does not exist (padding) \\ \greyrule
        \texttt{band\_twists}                 & $2c_{\max}$               & signed number of twists in each segment of the band, row-by-row with \texttt{band\_path} \\ \greyrule
        \texttt{simplified\_num\_crossings}   & $1$                       & number of crossings of the current diagram after simplification \\ \greyrule
        \texttt{init\_num\_crossings}         & $1$                       & number of crossings of the initial diagram of the episode \\
        \hline
    \end{tabularx}
    \caption{The elements of the base observation common to all frameworks (they describe the diagram and the band being built).}
    \label{tab:observation}
\end{table}

\begin{table}[ht]
    \centering
    \small
    \setlength{\tabcolsep}{4pt}
    \renewcommand{\arraystretch}{1.2}
    \renewcommand{\tabularxcolumn}[1]{m{#1}}
    \begin{tabularx}{\textwidth}{|l|c|>{\raggedright\arraybackslash}X|}
        \hline
        Tensor Name & Tensor Shape & Tensor Content \\
        \hline
        \texttt{crossing\_changes}            & $1$                       & number of crossing changes performed so far (unknotting only) \\ \hline
        \texttt{component\_ccl}               & $\ell_{\max}$             & CCL of each link component (genus frameworks only) \\ \greyrule
        \texttt{genera\_info}                 & $\ell_{\max} \times 5$    & one row per CCL, recording the CCL, the number of bands attached to that surface component, the number of initial link components (including unknots introduced by birth) and the number of current link components belonging to it, and the number of original link components it contains (genus frameworks only) \\ \greyrule
        \texttt{current\_genus}               & $1$                       & genus of the surface built so far (genus frameworks only) \\
        \hline
    \end{tabularx}
    \caption{The elements of the base observation specific to each framework: \texttt{crossing\_changes} is used in the unknotting framework and the last three in the genus frameworks.}
    \label{tab:observation-specific}
\end{table}

To this base observation, we may concatenate a number of topological invariants that the agent might not be able to extract directly from the diagram.
We encoded some of them (the linking matrix, the determinant, the signature, the Alexander and Jones polynomials, and a few others) in order to determine which ones help the agent improve its accuracy.
Which of them actually help is an experimental question, and is the subject of Section~\ref{subs:RL-policy-features}.

\begin{remark}
\label{rmk:env-limitations}
    The shape is a hard upper bound on diagram size. 
    A diagram exceeding $c_{\max}$ crossings or $\ell_{\max}$ components during an episode cannot be represented (illegal state), and the episode is lost rather than truncated.
    Since band moves may increase the number of crossings, some bands are inadmissible.
\end{remark}

\subsection{Terminal states: won, lost, and truncated episodes}
\label{subs:env-terminal}

An episode finishes when it reaches a terminal state.
This happens on three different occasions: either the episode is \emph{won} (i.e., the diagram represents an unlink), or it is \emph{lost} (i.e., an illegal state has been reached), or it is \emph{truncated} (i.e., the step budget \texttt{max\_actions}, which we set to $80$ in training, is exhausted).

A state can be illegal for different reasons:
\begin{itemize}
    \item the agent picks an illegal action (the action mask makes this unreachable, but the check is kept as a guard);
    \item the diagram no longer fits the shape $(c_{\max}, \ell_{\max})$;
    \item a band in progress terminates prematurely with no legal continuation (which can only happen in the strong frameworks because of the CCL constraint).
\end{itemize}

If an episode is won, the agent has produced a certificate for an upper bound.
In the \texttt{unknotting} case, the agent's answer coincides with the number of crossing changes performed during the episode; i.e., the number of \texttt{end} actions.
In the other frameworks, we need to compute the genus of the surface built as indicated in Section~\ref{subsubs:computation-slice-genus-from-movie}.
After the episode terminates in any way, we move to the next link in the dataset and start a new one.

\subsection{Rewards}
\label{subs:env-reward}

All frameworks use \emph{potential-based} reward shaping.
This means that the reward $R$ splits into an immediate reward term and a potential reward $\Phi$.
The potential reward is a function of the state, while the immediate term also depends on the action.
We denote by $s$ the present state, $a$ the action applied, and $s^{\prime}$ the next state.

\paragraph{The immediate term.}
The immediate term rewards winning, punishes errors, and charges a price for every other action.
This splits again into two terms: the action reward $r(a, s^{\prime})$, a function of the action and the next state, and the answer charge $C(s, s^{\prime})$, a function of the present and next states.
In Table~\ref{tab:rewards}, we report all the possible action reward values.

\begin{table}[H]
    \centering
    \small
    \begin{tabular}{|l|c|}
        \hline
        Event & $r(a, s^{\prime})$ \\
        \hline
        the diagram becomes terminal (episode won)                   & $+R_{\text{win}}$ \\
        an error state is entered (episode lost)                     & $-R_{\text{win}}$ \\
        a non-terminal \texttt{end} action (a band move or crossing change)     & $-5$ \\
        an \texttt{unknot\_birth}                                   & $-1$ \\
        any other action, with pre-action band length $L \geq 9$ & $-\min(L - 7,\, 5)$ \\
        otherwise & $-1$ \\
        \hline
    \end{tabular}
    \caption{The action rewards.}
    \label{tab:rewards}
\end{table}
Conditions are tested in the following order: error, win, birth, end, and other.
The band length test $L \ge 9$ uses the pre-action band path length $L$, and twist actions do not increase that length.
All non-winning actions are charged.
An \texttt{end} action that reaches a terminal state receives a base reward $+R_{\text{win}}$ without an additional $-5$ charge.

The answer charge is designed to give negative rewards for any increase in the returned invariant value, so that lower answers are preferred.
We first need to define an invariant cap $i_{\text{cap}}$. 
We guarantee that the reward of a won episode is higher than that of any episode that is not won and that lower answers are better than higher ones, both up to $i_{\text{cap}}$.
In Remark~\ref{rem:answer-charge-cap}, we discuss the reward behaviour past this cap.
We set $i_{\text{cap}} = 8$ as it is the highest invariant value in the datasets we are interested in.
Denote by $i(s)$ the current invariant value at the state $s$.
In the unknotting framework, this is the number of crossing changes made so far. 
In the genus frameworks, it is the current genus of the surface being built; see Section~\ref{subsubs:computation-slice-genus-from-movie}.
Note that for the initial state $s_0$, we have $i(s_0) = 0$.
The answer charge depends on the framework.
In the case of ribbon and slice, the cost of increasing the genus is charged immediately by defining $C$ as
\[
   C\left(s,s^{\prime}\right) = W \bigl( i(s^{\prime}) - i(s) \bigr),
\]
where $W \in \N$ is an integer factor called ``answer weight''.
For unknotting, strong ribbon, and strong slice, the cost of increasing the invariant value is charged only at the end of an episode. 
$C$ is defined as
\[
    C\left(s,s^{\prime}\right) =
    \begin{cases}
    W \min\bigl(i(s^{\prime}),\, i_{\text{cap}}\bigr), & \text{if $s^{\prime}$ is won,} \\
    0, & \text{otherwise.} \\
    \end{cases}
\]
This difference is justified in Remark~\ref{rem:answer-charge-cap}.

\paragraph{The potential term.}
The potential $\Phi$ brings the number of crossings of the diagram into the reward.
The potential contributes to the final reward as $\gamma\, \Phi\left(s^{\prime}\right) - \Phi\left(s\right)$, generally discounting the potential of the new state by the factor $\gamma \in (0, 1]$.
As already stated in Section~\ref{subs:rl-basics}, we use $\gamma = 1$.
This shaping term telescopes along an episode and it guides the agent towards promising states with fewer crossings~\cite{ng1999shaping}.
More specifically,
\[
\Phi(s) =
\begin{cases}
    0 \qquad & \text{if $s$ is terminal or an error state}, \\
    -\kappa \, c(s) \qquad & \text{otherwise},
\end{cases}
\]
where $c(s)$ is the number of crossings of the simplified current diagram, and $\kappa = 65/(1+\sqrt{c_0})$ scales the crossing term by the number of crossings $c_0$ of the initial diagram.
In particular, an error transition receives no potential shaping.
The numerator of $\kappa$ is an empirical constant left unchanged in all our experiments.
The total shaping accumulated over a successful episode is
\[
\sum_{t=0}^{T-1} \left( \Phi(s_{t+1}) - \Phi(s_t)\right) = -\Phi(s_0) + \Phi(s_T).
\]
For a successful episode, this is $\kappa c_0$.
Note that
\[
\kappa c_0 = \frac{65 c_0}{1 + \sqrt{c_0}} = 65\left( \sqrt{c_0} - 1 + \frac{1}{1 + \sqrt{c_0}} \right),
\]
which is asymptotic to $65\sqrt{c_0}$.
Links whose crossing numbers differ by an order of magnitude still receive shaping of comparable size.

\paragraph{The complete reward.}
The complete reward is defined as
\[
    R\left(s,a,s^{\prime}\right) =
    \begin{cases}
    -\sigma R_{\text{win}}, &s^{\prime} \text{ is an error state,} \\
    \sigma \bigl( r\left(a,s^{\prime}\right) - C\left(s,s^{\prime}\right) + \Phi\left(s^{\prime}\right) - \Phi\left(s\right) \bigr), &\text{otherwise,}\\
    \end{cases}
\]
where $\sigma \in (0,1)$ is a rescaling factor.

It remains to choose the values of $R_{\text{win}}$, $W$, and $\sigma$.
The first two depend on three factors: the maximum number of actions allowed in an episode, the invariant cap defined above, and the rewards of non-terminal actions of Table~\ref{tab:rewards}, which are between $-1$ and $-5$.
We choose $W = 5 \cdot 80 = 400$ to guarantee that an episode returning a lower invariant is better rewarded than one returning a higher value.
Indeed, a won episode has at most $79$ non-terminal actions (plus one winning end action), which cost at most $5 \cdot 79 = 395 < W$ in total, so a single unit of invariant outweighs any difference in the length of the path.
Then, to guarantee that a won episode is rewarded more than a truncated one, we choose $R_{\text{win}} = 5 \cdot 80 (i_{\text{cap}} + 1) = 3600$.
Lastly, the choice of $\sigma = 0.01$ rescales the rewards to the order of $10$ so that they do not saturate the value network's tanh units.

\begin{remark}[On the answer weight $W$ and the invariant cap $i_{\text{cap}}$]
    \label{rem:answer-charge-cap}
    Recall that $i_{\text{cap}} = 8$ is the highest invariant value in our datasets.
    We chose $W$ so that one unit of answer costs more than all the actions of a won episode together, and $R_{\text{win}} = W (i_{\text{cap}} + 1)$ so that it exceeds the largest charge, $W i_{\text{cap}}$, by $W$ as well.

    In the unknotting, strong ribbon, and strong slice frameworks, the answer is charged once, when the episode is won.
    In the strong frameworks, a band can reach a state with no legal move, and the episode then ends on an error.
    When we charged the genus band by band in these frameworks, ending an episode early on an error was cheaper than playing on and paying for the genus built, and the agents learned to end their episodes on an error.
    In the unknotting framework, charging every crossing change taught the agent to avoid crossing changes, although they are its only \texttt{end} actions.
    Without the cap, the first wins found in training used many crossing changes and were worth less than a truncated episode, and the episodes of the agent went back to the time limit.
    Since the charge is at most $W i_{\text{cap}}$, every won episode is better rewarded than every episode that is not won, whatever its answer.
    The better reward of a lower answer is guaranteed below the cap. 
    Answers from $i_{\text{cap}}$ and above receive the same charge and are compared only through the length of the path.
    If the initial diagram has at least two crossings, a truncated episode loses less than $R_{\text{win}}$ through its actions and the crossings it adds, so an error is worse rewarded than any other outcome.

    In the ribbon and slice frameworks, the genus is charged as the surface is built, with no cap.
    Here a lower genus is always better rewarded. 
    However, the reward of a won episode is guaranteed to be higher than that of every episode that is not won only up to genus $i_{\text{cap}}$.
    A truncated episode has already paid for the genus it built, so an error can be cheaper than a long truncated episode.
    Since ribbon and slice always have a legal move, they reach an error only when the diagram outgrows the shape.
\end{remark}

\section{The agents}
\label{sec:agents}

Three agents play the game of Section~\ref{sec:env}: a trained RL policy, a Bayesian-optimised random walker, and a mixture of the two.
This section describes each of them and how we selected the best one for each framework.

\paragraph{Accuracy.}
To compare the performance of different agents, we measure their accuracy.
This is defined as the percentage of knots and links for which the agent finds a sequence of moves that matches or improves the best known upper bound.

\paragraph{Working datasets.}
For each framework, we created the following four disjoint datasets that have been used at various stages:
\begin{itemize}
    \item The \textbf{training dataset} is used to train the RL policy.
    \item The \textbf{validation dataset} is used to select the best RL checkpoint and the best parameters for the random walker and the mixed agent. A fixed subset of $500$ of its links, the \emph{coarse validation set}, is used for the cheaper passes.
    \item The \textbf{development dataset} is used to choose the input features of the RL policy.
    \item The \textbf{test dataset} is used only to evaluate the final agents after every selection and tuning step has been completed and is never used to select or tune anything.
\end{itemize}
The datasets were constructed to provide the models with a diverse collection of samples.
In them, we can find prime knots and prime multi-component links, connected sums of knots and links, and, in the strong frameworks, Whitehead doubles of links, with varying invariant values and link sizes.
We split the samples by underlying primes. 
A prime knot or link, together with all its orientations and mirror images, lies in exactly one dataset, and the connected sums and Whitehead doubles of a dataset are built from the primes of that dataset only.
Hence, no prime link appears in two datasets in any role.
The assignment is the same for every framework, so a link is in the test dataset of every framework or of none.
Roughly $70\%$ of the links go to training and $10\%$ to each of the other three datasets.
Moreover, we split the samples so that the proportions of types, sizes, and invariant values are similar across the four datasets.
Given the similarity between the \texttt{ribbon} and \texttt{slice} frameworks, and the \texttt{strong ribbon} and \texttt{strong slice} frameworks, these pairs of frameworks share the same datasets.
The resulting dataset sizes are reported in Table~\ref{tab:datasets-train-validation-test}.

\begin{table}[H]
    \centering
    \begin{tabular}{|lccccc|}
        \hline
        Framework & Training & Validation & Development & Test & Total \\
        \hline
        \texttt{unknotting}                       & 6310 & 838 & 838 & 841 & 8827 \\
        \texttt{ribbon}, \texttt{slice}          & 5701 & 800 & 801 & 807 & 8109 \\
        \texttt{strong ribbon}, \texttt{strong slice} & 5153 & 731 & 868 & 690 & 7442 \\
        \hline
    \end{tabular}
    \caption{Dataset sizes. The coarse validation set consists of $500$ of the validation links.}
    \label{tab:datasets-train-validation-test}
\end{table}

\paragraph{Allocated agent budget.}
We quote evaluation budgets as $B/N$ budget units per link, where $B$ is the nominal total allocation and $N$ is the size of the dataset. 
An attempted episode containing $a$ actions is charged $a+1$ units: one unit per action and one unit accounting for reset per attempted episode, including unsuccessful and zero-action attempts. 
The evaluator cycles through the links, starting another episode only while its remaining budget is positive.
Once started, an episode runs until an ordinary stopping condition, including pruning or the episode action limit.
It is not interrupted by exhaustion of the evaluation budget. 
Consequently, the final episode can exceed the nominal allocation. 
The recorded answer for each link is the best found across its attempts, with only strict improvements replacing the stored answer.

For parallel evaluation, the dataset is partitioned into shards $S_j$. Writing $d_j = \sum_{i \in S_j} c_i^{3/2}$, where $c_i$ is the number of crossings of the diagram of link $i$, shard $j$ receives a budget of
\[
B_j = \left\lceil B\,\frac{d_j}{\sum_k d_k}\right\rceil.
\]
Each shard applies the episode-accounting rule independently. 
Upward rounding and the final episode in each shard can therefore make the consumed budget exceed $B$. 
The quoted per-link figures are nominal average allocations, not exact environment-action counts or hard per-link limits.

\subsection{The RL policy and the choice of input features}
\label{subs:RL-policy-features}

For our first agent, we trained an RL policy to interact with the environment described in Section~\ref{sec:env}.
Indeed, its observation contains every part of the state on which the legal actions, the next state, or the reward depend; these are the values reported in Tables~\ref{tab:observation} and~\ref{tab:observation-specific}.
The observation is fed into a pair of neural networks: a policy network that outputs a distribution over actions, and a value network that estimates the expected cumulative reward of the current state.
Both networks are built on top of Stable-Baselines3's MultiInputPolicy~\cite{stable-baselines3}.
For training, we use the \texttt{SB3-Contrib} implementation of MaskablePPO, an implementation of Proximal Policy Optimization (PPO) with invalid action masking.
PPO is an on-policy actor-critic policy gradient algorithm introduced by Schulman, Wolski, Dhariwal, Radford and Klimov~\cite{schulman2017ppo}, and is currently among the most widely used RL algorithms in practice.

The fixed-size tensors forming the dictionary observation are flattened and concatenated, and every entry is brought to order one by a fixed map with no trainable parameters: a division by a fixed scale or, for larger entries, a signed logarithm.
Without this scaling, large entries saturate the first $\tanh$ layer, and the policy stops depending on the state.
The resulting vector is passed to separate policy and value networks, each consisting of two fully connected layers with 256 units each with $\tanh$ activations.
We use a discount factor $\gamma=1$ (since the reward of Section~\ref{subs:env-reward} is designed for undiscounted returns), a learning rate $3\times10^{-4}$, a batch size $64$, ten optimisation epochs per update, a GAE parameter $\lambda=0.95$, and a PPO clipping range $0.2$.
The remaining PPO hyperparameters are left at their library defaults.

At inference time, the policy defines a categorical distribution over the fixed action space. 
Before each prediction, actions that are illegal in the current state are masked out, and the distribution is renormalised over the remaining actions. 
We sample from this masked distribution and repeat the procedure until the episode terminates or reaches its step limit.

\subsubsection{The feature-selection protocol}
The first question we needed to address was which features would provide the RL agent with the most informative representation of the current state.
The base observation for each framework is shown in Tables~\ref{tab:observation} and~\ref{tab:observation-specific}.
Several additional \emph{candidate features} can be added to this observation, but it is not clear a priori which of them would provide the agent with the most useful information.
We tested the following features:
\begin{itemize}
    \item \texttt{Baseline} (base observation only);
    \item \texttt{MT}: the Murasugi--Tristram bound;
    \item \texttt{Seif}: the upper bound on the Seifert genus obtained using Seifert's algorithm;
    \item \texttt{LM}: the Linking Matrix;
    \item \texttt{CCLM}: the CCL Linking Matrix, where entry $(i,j)$ is the total linking number of the link components with Connected Component Labels $i$ and $j$; 
    \item \texttt{CompM}: the Components Matrix, where entry $(i,j)$ records on which link component (labelled by integers) the $j$-th strand of the $i$-th crossing lies;
    \item \texttt{DetSig}: the link determinant and signature;
    \item \texttt{Alex}: the Alexander polynomial;
    \item \texttt{Jones}: the Jones polynomial (only for links up to 16 crossings, since the cost of computing it grows exponentially with the number of crossings);
    \item \texttt{Theta}: evaluations of the $\Theta$-invariant of Bar-Natan and van der Veen~\cite{barnatanvanderveen2025theta}, at a single evaluation point (the cost grows linearly with the number of points, and \texttt{Theta} is already the slowest feature to train with).
\end{itemize}
All these quantities can be recomputed from the current diagram every time it changes. 
We did not include invariants such as the Rasmussen invariant $s$, whose computation is far more expensive.
In addition, we trained using pairs of features for each framework, pairing the two features whose models solved the most complementary sets of development links: \texttt{LM}$+$\texttt{MT} (\texttt{unknotting}), \texttt{DetSig}$+$\texttt{LM} (\texttt{ribbon}), \texttt{DetSig}$+$\texttt{Seif} (\texttt{slice}), \texttt{MT}$+$\texttt{Jones} (\texttt{strong ribbon}), and \texttt{CompM}$+$\texttt{DetSig} (\texttt{strong slice}).
This provided no improvement overall. 
Each feature pair scored below the best of the two features on the development dataset, between $3.6\%$ and $12.0\%$ depending on the framework.
The overall score of the feature pair was also below the score of the best model of each framework, so we do not consider them further.

To select the best feature set for each framework, we trained one model per candidate feature and compared them under the following protocol.

\paragraph{Checkpoint selection (validation).}
For each configuration, we trained for roughly $1{,}000{,}000$ timesteps (between about $160$ and $195$ timesteps per training link), saving the policy parameters every $15{,}000$ timesteps. 
We call each saved snapshot a \emph{(training) checkpoint}, obtaining $66$ checkpoints in total.
The budget is deliberately generous, so that training can continue beyond the peak in validation performance.
This matters because the best-performing checkpoint often occurs early in training.

Let us denote by $C_1,\dots,C_{66}$ the checkpoints obtained during training.
A naive way to select one would be to validate all of them on the validation set and keep the one with the highest accuracy score.
However, this could result in favouring whichever checkpoint was most inflated by noise~\cite{smith2006optimizers}. 
Moreover, the validation curve itself is noisy, and the raw maximum might not be the best stopping point~\cite{prechelt2012early}.
Therefore, the checkpoint validation consists of two phases.
\begin{enumerate}
    \item First, a \textbf{coarse pass}.
    We score the first checkpoint $C_1$ and every third checkpoint ($C_{3n}$) on the coarse validation set, with a nominal evaluation budget of 200 steps per link.
    Writing $a(C_{3n})$ for the corresponding accuracy score, we compute a moving average over adjacent scored checkpoints rather than take the noisy arg-max:
    \[
        s(C_{3n}) = \tfrac{1}{3}\left( a(C_{3n-3}) + a(C_{3n}) + a(C_{3n+3}) \right),
    \]
    where $C_1$ plays the role of $C_0$, and the average is taken over the available neighbours at the two ends of the curve.
    This whole procedure is computationally cheap, and it allows us to locate the plateau of the validation curve by looking at the five checkpoints with the highest smoothed accuracies.
    \item Second, the \textbf{refinement}.
    We then evaluate these checkpoints and their immediate neighbours, $C_{3n-1}$ and $C_{3n+1}$, on the full validation set.
    The best of these is selected as the \emph{best model} for that configuration.
\end{enumerate}

\paragraph{Feature choice (development) and budget.}
The selected checkpoint of each configuration is evaluated exactly once on the development dataset, and the configuration with the highest accuracy there is selected for its framework.
Any configuration within one standard error of the best counts as a tie, and a tie is resolved in favour of the feature that is cheaper to compute.
The test dataset plays no role in this choice.

Training multiple seeds of the same feature showed that the accuracy of the selected checkpoint varies with the seeding by a few percentage points.
Therefore, the above selection process has to be considered a mix of model selection and feature selection.

Every evaluation presented above (coarse, full-validation, and development) gives the RL agent a nominal evaluation budget of about 200 units per link.
This allocation is a compromise between computational feasibility and model accuracy. 
On one hand, each link needs several search attempts because the policy is stochastic, and a valid sequence of moves may not be found on the first try. 
On the other hand, it would take too long to evaluate the models with a substantially larger budget.
The development evaluation was performed at both 200 and 800 units per link, the latter being the budget given for the test evaluation of Section~\ref{subsec:best-agent-rl-vs-rw}.
The winning configuration is the same in both.

\paragraph{Longer training.}
The validation curves of several configurations were still rising at the end of training, so we trained ten configurations, chosen among the best of their frameworks, for $2{,}000{,}000$ timesteps: \texttt{DetSig} and \texttt{LM} on \texttt{ribbon}; \texttt{Seif}, \texttt{Jones}, and \texttt{Alex} on \texttt{slice}; \texttt{MT} on \texttt{strong ribbon}; and \texttt{CompM}, \texttt{LM}, \texttt{Seif}, and \texttt{DetSig} on \texttt{strong slice}.
No \texttt{unknotting} configuration was included, since the best ones reach their peak within the first $75{,}000$ timesteps.
Only one of these runs improved on the model selected for its framework: the \texttt{MT} configuration of the \texttt{strong ribbon} framework, whose development accuracy went from $12.7\%$ to $18.4\%$, with the best checkpoint at $1{,}755{,}000$ timesteps.
We therefore trained it for $1{,}000{,}000$ timesteps more, with the new best checkpoint being the one at $2{,}775{,}000$ timesteps, where the new development accuracy is $23.3\%$.
We use this model for the \texttt{strong ribbon} framework.

\paragraph{Best feature set per framework.}
The selected feature configuration for each framework is given in Table~\ref{tab:feature-selection}.

\begin{table}[H]
    \centering
    \small
    \setlength{\tabcolsep}{5pt}
    \begin{tabular}{|l|l|c|c|}
        \hline
        Framework & Selected feature & Accuracy & \texttt{Baseline} \\
        \hline
        \texttt{unknotting}     & \texttt{Baseline}                                            & $66.8\%$ & $66.8\%$ \\
        \texttt{ribbon}         & \texttt{DetSig} (determinant and signature)                  & $63.3\%$ & $56.4\%$ \\
        \texttt{slice}          & \texttt{Seif} (Seifert genus upper bound)                    & $60.2\%$ & $57.9\%$ \\
        \texttt{strong ribbon}  & \texttt{MT} (Murasugi--Tristram bound)                       & $23.3\%$ & $7.7\%$ \\
        \texttt{strong slice}   & \texttt{CompM} (components matrix)                           & $14.1\%$ & $6.2\%$ \\
        \hline
    \end{tabular}
    \caption{Selected input feature per framework, with its accuracy on the development set at $200$ units per link and, for comparison, the accuracy of the bare state observation.
    Since each selected accuracy is the best of ten, it is biased upwards. 
    The unbiased figures are those measured on the test set in Section~\ref{subsec:best-agent-rl-vs-rw}.}
    \label{tab:feature-selection}
\end{table}

\subsection{The random walker and its Bayesian optimisation}
\label{subs:rw}

Alongside the RL agent, we also built a random walker (RW) agent that is able to interact with the environment.
The random walker is governed by two relative action weights and three additional parameters.
The weights \texttt{p\_twist} and \texttt{p\_end} control the relative frequencies of twist and end actions, respectively.
The parameter \texttt{p\_birth} is the conditional probability of adding an unknot when no band is in progress and is used only in the slice frameworks.
Finally, \texttt{max\_twists} and \texttt{max\_actions} bound the number of twists per band and the number of actions per episode, respectively.

At each step, the random walker first removes action types for which no legal action exists. 
The remaining types \texttt{start}, \texttt{over}, \texttt{under}, \texttt{twist}, and \texttt{end} receive relative weights $1$, $1$, $1$, \texttt{p\_twist}, and \texttt{p\_end}, respectively. 
An action type is sampled proportionally to these weights, after which a concrete legal action of that type is selected uniformly. 
In particular, the twist direction is chosen uniformly among the legal directions, while an end action is obtained by sampling a legal final strand and then a legal end sign uniformly. 
In the slice frameworks, when no band is in progress, an \texttt{unknot\_birth} action is first taken with probability \texttt{p\_birth}; otherwise, the ordinary weighted sampling rule is used.

\paragraph{Bayesian optimisation.}
For each framework, these parameters were tuned via Bayesian optimisation (BO) using the accuracy on the same coarse validation sets as the metric employed for the coarse validation of the RL models.
Bayesian optimisation is well suited to this problem, as the objective is a black box with no accessible gradient, evaluations are expensive, and the space of weights is low-dimensional. 
The optimiser builds a probabilistic model of the accuracy over that space and uses it to choose which weights to try next, spending its budget where the model is promising or still uncertain rather than on a blind grid.
We used the Bayesian optimisation implementation developed by Nogueira~\cite{nogueira2014bayesianBO}, with its default configuration.\footnote{Version $2.0.0$: a Gaussian process surrogate with a Mat\'ern $\nu = 5/2$ covariance, and an upper confidence bound acquisition with $\kappa = 2.576$.}

The BO was performed with a budget of 300 iterations (50 initial exploratory iterations plus 250 tuning iterations), each using a nominal evaluation budget of $500{,}000$ units ($1{,}000$ per link).
We first searched the space $\texttt{p\_twist} \in [0.1, 5.0]$, $\texttt{p\_end} \in [0.5, 10.0]$, $\texttt{max\_twists} \in [1, 5]$, $\texttt{max\_actions} \in [5, 50]$, and $\texttt{p\_birth} \in [0.0, 1.0]$ for the two slice frameworks.
On four frameworks, the optimum of \texttt{p\_end} lay at the upper edge of this box, and that of \texttt{p\_twist} close to its lower edge.
We therefore repeated the search with $\texttt{p\_twist} \in [0.01, 5.0]$ and $\texttt{p\_end} \in [0.5, 100]$, both sampled on a logarithmic scale, running two independent optimisations with different random seeds for each framework.
Instead of adopting the best configuration found by the optimiser, we re-evaluated the ten best configurations of each run five times on the full validation set, with the nominal budget of $800$ units per link used in the final benchmark, and adopted the one with the highest mean accuracy.
The tuned defaults are given in Table~\ref{tab:rw-bo}, where \texttt{max\_twists} and \texttt{max\_actions} are converted to integer values.
Since \texttt{p\_twist} and \texttt{p\_end} are relative weights, the table also reports the probabilities they induce at a step inside a band where every action type is legal:
\[
    P(\mathrm{end}) = \frac{\texttt{p\_end}}{2 + \texttt{p\_twist} + \texttt{p\_end}}, \qquad P(\mathrm{twist}) = \frac{\texttt{p\_twist}}{2 + \texttt{p\_twist} + \texttt{p\_end}}.
\]

\begin{table}[H]
    \centering
    \small
    \setlength{\tabcolsep}{4pt}
    \begin{tabular}{|l|c|c|c|c|c|c|c|c|}
        \hline
        \multirow{2}{*}{Framework} & \multirow{2}{*}{\texttt{p\_twist}} & \multirow{2}{*}{\texttt{p\_end}} & \multirow{2}{*}{$P(\mathrm{end})$} & \multirow{2}{*}{$P(\mathrm{twist})$} & \texttt{max\_} & \texttt{max\_} & \multirow{2}{*}{\texttt{p\_birth}} & \multirow{2}{*}{Accuracy} \\
            & & & & & \texttt{twists} & \texttt{actions} & & \\
        \hline
        \texttt{unknotting}     & $0.061$ & $90.44$ & $97.8\%$ & $0.07\%$ & $1$ & $10$ & $-$     & $58.2 \pm 0.3\%$ \\
        \texttt{ribbon}        & $0.351$ & $98.31$ & $97.7\%$ & $0.35\%$ & $1$ & $20$ & $-$     & $56.9 \pm 0.6\%$ \\
        \texttt{slice}         & $0.012$ & $86.63$ & $97.7\%$ & $0.01\%$ & $1$ & $21$ & $0.005$ & $56.4 \pm 0.2\%$ \\
        \texttt{strong ribbon} & $0.023$ & $90.02$ & $97.8\%$ & $0.03\%$ & $1$ & $17$ & $-$     & $10.6 \pm 0.2\%$ \\
        \texttt{strong slice}  & $0.022$ & $89.80$ & $97.8\%$ & $0.02\%$ & $4$ & $17$ & $0.001$ & $10.7 \pm 0.3\%$ \\
        \hline
    \end{tabular}
    \caption{Tuned random walk parameters per framework, with the probabilities of ending a band and of inserting a twist that they induce. The last column is the validation accuracy, as mean and standard error over five evaluations at $800$ units per link.}
    \label{tab:rw-bo}
\end{table}

In all ten optimisation runs, the best configuration ends a band with probability between $95.3\%$ and $97.9\%$, and the adopted walkers almost never twist. 
So, the walker's best strategy is to close a band as soon as it legally can.
The optimum of \texttt{p\_end} is again close to the upper edge of the box, but this is saturation rather than truncation: setting $\texttt{p\_end} = 10^4$ so that $P(\mathrm{end}) \approx 99.98\%$ changes the validation accuracy of the tuned walkers by between $-0.1\%$ and $+1.6\%$, never by more than $2.2$ standard errors (five paired evaluations at $800$ units per link).
The parameter that separates the frameworks is \texttt{max\_actions}: $10$ for \texttt{unknotting}, $17$ for the strong frameworks, $20$ for \texttt{ribbon}, and $21$ for \texttt{slice}.
On \texttt{unknotting}, a cap of $10$ allows at most five crossing changes, so the walker gives up the links whose best known upper bound is at least six, about $7\%$ of each dataset, in exchange for more attempts on the others. 
The two optimisations chose $10$ and $11$ out of a range reaching $50$.

We now remark on the last column of Table~\ref{tab:rw-bo}.
Each evaluation of an objective is a single stochastic run over the coarse validation set, so the accuracy it returns is noisy, and what the optimiser reports is the largest of a few hundred such draws. 
As with the raw validation maximum over checkpoints in Section~\ref{subs:RL-policy-features}, this is biased upwards by an amount that grows with the budget \cite{smith2006optimizers}.
This is why we re-evaluated the candidates, and on three frameworks, the optimiser's own best configurations scored up to $3.5\%$ below the adopted ones.
The figures in the table are still tuning figures, since the coarse validation set is part of the validation set, and are not themselves performance claims.
Every accuracy we compare across agents is measured on the held-out test sets.

\subsection{The mixed RL+RW agent}
\label{sec:mixed}

As a final agent, we built a \emph{mixed agent} that follows the trained RL policy most of the time and hands control to the random walker (RW) for the remaining steps.
At the start of each episode, we draw the initial probability of selecting the random walker and then decay this probability as the episode proceeds.
The mixed agent is governed by the following parameters:
\begin{itemize}
    \item \texttt{p\_twist}: the relative weight assigned to twist actions when the random walker is in control;
    \item \texttt{p\_end}: the relative weight assigned to end actions when the random walker is in control;
    \item \texttt{p\_birth}: the conditional probability of adding an unknot when the random walker is in control and no band is in progress (this parameter is meaningful only for the slice frameworks);
    \item \texttt{p\_rw\_range}: an interval $[a,b] \subseteq [0,1]$, with $a \leq b$, from which the initial probability of following the random walker is drawn at the start of each episode;
    \item \texttt{min\_p\_rw}: a floor on that probability, leaving the random walker a residual chance of intervening late in the episode, once the decay has taken effect;
    \item \texttt{step\_decay} and \texttt{end\_decay}: the multiplicative factors at which the random walker's influence fades within an episode; the former is applied after an ordinary action, and the latter after a band is ended.
\end{itemize}

Compared with the pure random walker, \texttt{max\_twists} and \texttt{max\_actions} are absent, and those limits are inherited from the RL environment rather than tuned for each agent.

\paragraph{How the mixed agent works.}
The two strategies are combined as follows.
At the start of each episode, an initial probability $p^{(0)}_{\mathrm{rw}}$ is sampled uniformly from the interval \texttt{p\_rw\_range}.
This probability is multiplied by \texttt{end\_decay} after every end action and by \texttt{step\_decay} after every other action. 
We write $p^{(t)}_{\mathrm{rw}}$ for its value at step $t$. For instance, after $t$ non-end actions, we will get $p^{(t)}_{\mathrm{rw}} = p^{(0)}_{\mathrm{rw}} \cdot \left( \texttt{step\_decay} \right)^{t}$.
At each step, the action is taken by the random walker with probability $\max(\texttt{min\_p\_rw},\, p^{(t)}_{\mathrm{rw}})$, and by the trained RL policy with the complementary probability.

When the random walker is selected, it performs an unknot birth with probability \texttt{p\_birth}, provided the framework permits it and no band has started.
Otherwise, it samples an action type among those currently legal, with relative weights $1$ for \texttt{start}, \texttt{over}, and \texttt{under}, \texttt{p\_twist} for a twist, and \texttt{p\_end} for an end action, and then chooses the crossing strand (or the twist direction) uniformly among the legal ones.
The random walker, therefore, injects stochastic exploration early in the episode and gradually hands control over to the RL policy as the episode develops.

\subsubsection{Bayesian optimisation of the mixed agent}
\label{subsec:mixed-bo}

We tuned these parameters via Bayesian optimisation, using each framework’s best RL model as the policy component.
We used the same coarse validation sets employed for the coarse validation of the RL models, with a budget of $200$ objective evaluations per optimisation ($40$ initial points plus $160$ optimisation iterations) and a nominal evaluation budget of 200 units per link for each objective evaluation.
As for the random walker, we searched \texttt{p\_twist} and \texttt{p\_end} on the wider logarithmic box of Section~\ref{subs:rw}, ran two independent optimisations per framework, and re-evaluated the five best configurations of each run five times on the full validation set at $800$ units per link, adopting the one with the highest mean accuracy.
We re-evaluated five configurations per run rather than the ten of the random walker because a step of the mixed agent costs several times a step of the walker.

As we show below, the parameter that matters most is the range \texttt{p\_rw\_range} from which the mixing rate is drawn.
From a practical standpoint, we encoded \texttt{p\_rw\_range} by two endpoints $[\texttt{p\_rw\_range}_{\mathrm{low}}, \texttt{p\_rw\_range}_{\mathrm{high}}]$, sampled independently from $[0,1]$ and then sorted. 
We parametrised \texttt{min\_p\_rw} as a fraction $\epsilon \in [0,1]$ of the lower endpoint, so that
\[
    \texttt{min\_p\_rw} \;=\; \epsilon \cdot \texttt{p\_rw\_range}_{\mathrm{low}}.
\]
This parametrisation guarantees that every configuration proposed by the optimiser is admissible, so that no evaluation is wasted on an invalid one. Sampling \texttt{min\_p\_rw} independently could place it above the sampled range and make it override both the range and the decay.

\paragraph{Which parameters actually matter.}
\label{subsec:mixed-sensitivity}

To identify which parameters carry a detectable signal, we use a simple descriptive signal-to-noise comparison.
We first estimate a noise floor from independent repeated evaluations of identical parameter configurations. 
For each pair of repeats, we compute the absolute difference between their accuracies and define the noise floor as the median of these differences.
This captures the run-to-run variation due to the stochasticity of the agent and the evaluation procedure.

We then estimate the marginal effect of each parameter. 
We split its sampled values into five equally populated groups, compute the mean accuracy within each group, and define its marginal accuracy range as the difference between the largest and the smallest of these five means.
For \texttt{p\_rw\_range}, we use the midpoint of the range and all the evaluations.
For the other parameters, we only use the evaluations in the three lowest groups of \texttt{p\_rw\_range}, whose means differ by less than the noise floor, so that the effect of the mixing rate does not mask theirs.

Table~\ref{tab:mixed-sensitivity} compares these accuracy ranges with the noise floor.
We call a parameter significant when its accuracy range is substantially larger than the noise floor.
Across all five frameworks, \texttt{p\_rw\_range} meets this criterion, and at this budget, its effect is negative. 
The accuracy is flat over the lower mixing rates and drops for the higher ones.
The ranges of all the other parameters are below the noise floor, with one exception, \texttt{p\_birth} on the \texttt{slice} framework, which we discuss below.
We therefore conclude that \texttt{p\_rw\_range} is the only parameter that exhibits a consistent and substantial marginal effect on accuracy.

\begin{table}[H]
    \centering
    \begin{tabular}{|l|c||c|c|c|c|c|}
        \hline
        \multirow{2}{*}{Framework}
            & \multirow{2}{*}{\shortstack{Noise floor}}
            & \multicolumn{5}{c|}{Accuracy range of} \\
            & & \texttt{p\_rw\_range} & \texttt{p\_twist} & \texttt{p\_end} & \texttt{min\_p\_rw} & \texttt{p\_birth} \\
        \hline
        \texttt{unknotting}     & $1.10$ & $18.3$ & $0.32$ & $0.69$ & $0.61$ & $-$ \\
        \texttt{ribbon}        & $1.10$ & $9.4$  & $0.91$ & $0.80$ & $0.89$ & $-$ \\
        \texttt{slice}         & $1.00$ & $10.4$ & $0.92$ & $0.80$ & $0.56$ & $2.01$ \\
        \texttt{strong ribbon} & $1.20$ & $12.6$ & $0.70$ & $0.30$ & $0.83$ & $-$ \\
        \texttt{strong slice}  & $0.80$ & $6.7$  & $0.34$ & $0.20$ & $0.39$ & $0.37$ \\
        \hline
    \end{tabular}
    \caption{Accuracy ranges associated with each parameter, computed on the evaluations of the Bayesian optimisation, compared with the noise floor (coarse validation set, $200$ units per link). All values are expressed in percentage points.}
    \label{tab:mixed-sensitivity}
\end{table}
 
A plausible explanation for this behaviour is the following.
The parameters \texttt{p\_twist}, \texttt{p\_end}, and \texttt{p\_birth} determine the behaviour of the random walker conditional on that component being selected, and under the tuned schedules of Table~\ref{tab:mixed-bo}, the random walker controls only a handful of the agent's steps per episode.
The RW-specific parameters, therefore, have few opportunities to affect the final accuracy.
The exception is \texttt{p\_birth} in the \texttt{slice} framework, where more unknot births steadily lower the accuracy, from $59.7\%$ in the lowest group to $57.7\%$ in the highest.
This agrees with the tuned random walker, whose \texttt{p\_birth} is pinned near zero.
In the \texttt{strong slice} framework, \texttt{p\_birth} has no measurable effect.
We nevertheless keep a small \texttt{p\_birth} in the \texttt{slice} framework ($0.05$ in the adopted configuration) so that the moves of the random walker also include the unknot birth, the move that distinguishes the \texttt{slice} framework from the \texttt{ribbon} one.
The effect of \texttt{min\_p\_rw} is also below the noise floor, so we kept the value of the adopted configuration, which is at most $0.023$.

\paragraph{Optimal mixing rate and decay.}
\label{subsec:mixed-dose}
\label{subsec:mixed-decay}

The mixing rate decides how much of the search goes into random exploration.
This pays off only when the budget leaves room for many attempts per link, so the best mixing rate depends on the budget, while the other parameters only shape the few moves that the random walker makes.
Indeed, among the ten configurations re-evaluated above (the five best of each optimisation run), the accuracy at $800$ units per link grows with the mixing rate, with correlations between $+0.42$ and $+0.90$, whereas at the $200$ units per link of the optimisation, high mixing rates lower it.

The parameters \texttt{p\_twist}, \texttt{p\_end}, \texttt{p\_birth}, and \texttt{min\_p\_rw} at $800$ units per link made no difference.
We therefore chose the mixing rate by a grid search at the budget of the final benchmark, together with \texttt{step\_decay}, which sets how quickly the random walker hands control over to the policy and may interact with it. 
The optimisation had held both decays fixed, at $\texttt{step\_decay} = 0.9$ and $\texttt{end\_decay} = 0.75$.
Every episode starts with the same mixing rate 
\[
k \in \{0.05, 0.15, 0.30, 0.60, 0.80, 1.00\}, 
\]
crossed with 
\[
\texttt{step\_decay} \in \{0.15, 0.30, 0.60, 0.90\},
\]
except $0.90$ for $k \geq 0.80$, since at $k = 0.60$ it was already the worst decay rate on every framework, and every other parameter is held at the value selected above. Each of the resulting $22$ grid points was evaluated at least $4$ times on the coarse validation set at $800$ units per link.

The two rates interact: a faster decay helps more at a high mixing rate than at a low one.
On \texttt{unknotting}, \texttt{ribbon}, and on the two strong frameworks, the gain from the fastest decay is between $2.2\%$ and $9.4\%$ larger at $k = 0.60$ than at $k = 0.05$, by at least $3$ standard errors. On \texttt{slice}, it goes in the same direction, by about $1.5\%$, but within $2$ standard errors.

We then re-evaluated the two best grid points (with the range $\texttt{p\_rw\_range} = [k,k]$ of zero width) of each framework, together with the configuration selected by the optimiser, five times on the full validation set with fresh random seeds, and adopted a grid point only if it beat that configuration by more than one standard error.
This happened on all five frameworks, with gains between $+2.2\%$ and $+6.1\%$.

As a last step to find the best mixed parameters, we tested ranges $[a,b]$ with $a<b$, letting \texttt{step\_decay} vary as well.
The accuracies of the ranges are close to the mean of the grid accuracies over $[a,b]$ at the same decay rate; there is however a small gain growing with the width of the range.
We evaluated twelve ranges per framework four times on the coarse validation set at $800$ units per link, and re-evaluated the three best ones, against the grid point selected above with $\texttt{p\_rw\_range} = [k,k]$, five times on the full validation set with fresh random seeds.
On four frameworks, the best of these configurations was more accurate than that grid point, by between $0.2\%$ and $0.9\%$, and on \texttt{strong ribbon} the two tied ($-0.03\%$), always within two standard errors.
The two choices are therefore equivalent, and we adopt the ranges.
In the final configuration, \texttt{step\_decay} is $0.15$, the lowest value of the grid, on \texttt{unknotting} and \texttt{ribbon}, and $0.30$ on \texttt{slice}, \texttt{strong ribbon}, and \texttt{strong slice}.
Testing \texttt{step\_decay} separately at $0$ and $0.05$ on the first two showed no effect, and we keep $0.15$ as the tuned value.
The \texttt{end\_decay} parameter was tested separately at $0.15$, $0.45$, $0.75$, and $1.0$, and showed no significant effect on the score of the mixed agent. 
It therefore stayed at $0.75$ throughout.

\paragraph{Tuned defaults.}
\label{subsec:mixed-defaults}

Table~\ref{tab:mixed-bo} collects the tuned parameters of the mixed agent for each framework.

\begin{table}[ht]
    \centering
    \begin{tabular}{|l|c|c|c|c|c|}
        \hline
        Framework & \texttt{p\_rw\_range} & \texttt{step\_decay} & \texttt{min\_p\_rw} & \texttt{p\_birth} & Accuracy \\
        \hline
        \texttt{unknotting}     & $(0.70,\,1.00)$ & $0.15$ & $0.022$  & $-$     & $86.8 \pm 0.3\%$ \\
        \texttt{ribbon}        & $(0.90,\,1.00)$ & $0.15$ & $0.006$  & $-$     & $81.9 \pm 0.3\%$ \\
        \texttt{slice}         & $(0.80,\,0.90)$ & $0.30$ & $0.0003$ & $0.050$ & $81.0 \pm 0.4\%$ \\
        \texttt{strong ribbon} & $(0.75,\,1.00)$ & $0.30$ & $0.0025$ & $-$     & $49.0 \pm 0.6\%$ \\
        \texttt{strong slice}  & $(0.70,\,1.00)$ & $0.30$ & $0.023$  & $0.007$ & $34.0 \pm 0.7\%$ \\
        \hline
    \end{tabular}
    \caption{Tuned mixed agent parameters per framework.
    The accuracy is measured on the validation set at $800$ units per link, as mean and standard error over five evaluations with fresh random seeds.}
    \label{tab:mixed-bo}
\end{table}

Three remarks on these values are in order. First, \texttt{p\_twist} and \texttt{p\_end} are omitted because they showed no detectable effect under this analysis (Table~\ref{tab:mixed-sensitivity}). Any value in the search range performs equally well, and the optimiser's choice carries no information.
As remarked above, \texttt{end\_decay} stays at $0.75$ throughout.

Second, under the tuned schedules, the random walker acts almost only at the start of an episode.
The random walker takes the first action of an episode with probability $k$, the second with probability $k \cdot \texttt{step\_decay}$, and so on, until the floor \texttt{min\_p\_rw} is reached, so the expected number of random actions per episode, $k / (1 - \texttt{step\_decay})$, ranges from $1.0$ for \texttt{unknotting} to $1.25$ for \texttt{strong ribbon}.
The tuned mixed agent therefore randomises the first one or two actions of an episode and then follows the RL policy.
Since the first action of an episode can only start a band or, in the slice frameworks, add an unknot, this also explains why \texttt{p\_twist} and \texttt{p\_end} carry no signal.

Lastly, the accuracies in the last column are not comparable across frameworks, since they are measured on different validation sets and refer to different invariants.
The strong frameworks score much lower than the others here, as they do in the pure RL and pure RW cases.

\paragraph{A possible explanation.}
\label{subsubsec:mixed-explanation}
The above discussion suggests using the mixed RL+RW agent on every framework.
Indeed, this was also confirmed by the analysis we report in Section~\ref{subsec:best-agent-rl-vs-rw}.
We can offer a possible explanation for this.
It is instructive to compare our findings with the results reported by Gukov, Halverson, Manolescu, and Ruehle~\cite{gukov2025ribbon}. There, the best-performing agent at finding ribbon discs was a Bayesian-optimised random walker, with the trained RL policy slightly behind it.
In our setting, the opposite holds when the agents are given the same number of steps. As we discuss in Section~\ref{subsec:best-agent-rl-vs-rw}, our RL agents then outperform the optimised random walkers on every framework.
Given the same computing time instead, the walker overtakes the pure policy on the \texttt{ribbon} and \texttt{slice} frameworks and draws level with it on the others, and only the mixed agent stays ahead of the random walker on \texttt{unknotting} and on the strong frameworks (Section~\ref{subsec:equal-time}).
Note that the two benchmarks measure different quantities (theirs is an accuracy computed on knots known to be ribbon), so the comparison should be read qualitatively only.
We believe the difference is due to the wider action set available to our agents, which lets the RL policy choose the right sequence of moves more reliably.

The benefit of mixing the two agents can be explained in an analogous way by claiming that the random walker helps the RL agent where it struggles the most.
It is most visible where the policy is weakest.
On the links with the lowest invariant value of every framework except \texttt{strong ribbon},
the pure RL policy is less accurate than the tuned random walker (for instance, $44.0\%$ against $56.6\%$ on the \texttt{ribbon} test links of genus zero), and the mixed agent beats both ($64.6\%$).
These are the links on which many cheap random attempts might pay off.
On \texttt{strong ribbon}, the pure RL policy is already more accurate than the random walker on the links of genus one ($35.2\%$ against $30.2\%$), and mixing helps more on the links of higher genus.
After training, the policy tends to retry the same first band on a large share of the episodes of a link, and a single random action at the start is enough to diversify its attempts (Table~\ref{tab:diversity}).
Under the tuned schedule, indeed, the random walker takes only one or two actions per episode before handing control back to the RL policy.
Its benefit therefore appears to come from an early stochastic intervention, which can move the episode into a state from which the RL policy is more likely to find a successful trajectory.

\begin{table}[ht]
    \centering
    \begin{tabular}{|l|>{\centering\arraybackslash}p{5em}>{\centering\arraybackslash}p{5em}|>{\centering\arraybackslash}p{6.5em}>{\centering\arraybackslash}p{6.5em}|}
        \hline
        \multirow{2}{*}{Framework}
            & \multicolumn{2}{c|}{Most common first band}
            & \multicolumn{2}{c|}{Distinct first bands per episode} \\
            & \texttt{rl} & \texttt{rl\_rw} & \texttt{rl} & \texttt{rl\_rw} \\
        \hline
        \texttt{unknotting}    & $43\%$ & $9\%$  & $0.11$ & $0.35$ \\
        \texttt{ribbon}        & $63\%$ & $6\%$  & $0.12$ & $0.64$ \\
        \texttt{slice}         & $62\%$ & $9\%$  & $0.13$ & $0.68$ \\
        \texttt{strong ribbon} & $79\%$ & $9\%$ & $0.15$ & $0.88$ \\
        \texttt{strong slice}  & $73\%$ & $9\%$  & $0.18$ & $0.89$ \\
        \hline
    \end{tabular}
    \caption{Diversity of the attempts of the pure RL policy (\texttt{rl}) and of the mixed agent (\texttt{rl\_rw}) on the development set, at $800$ units per link.
    Most common first band: share of the episodes of a link that open with its most common first band, median over the links.
    Distinct first bands per episode: number of distinct first bands of a link divided by its number of episodes, averaged over the links.}
    \label{tab:diversity}
\end{table}

\subsection{Best agent among RW, RL, and RL+RW}
\label{subsec:best-agent-rl-vs-rw}

To determine the best agent for each framework, we evaluated all agents on the held-out test datasets, using a nominal evaluation budget of $800$ units per link.
The same analysis was performed at budgets of $200$, $400$, and $1{,}600$ units per link, without changing the ranking of the agents. 
We therefore report only the results of the analysis at $800$ units per link.

\paragraph{Benchmark setup.}
Alongside the three agents discussed above (which we call \texttt{bo\_rw}, \texttt{rl}, and \texttt{rl\_rw}), we also tested a fourth agent, \texttt{naive\_rw}: a non-optimised random walker with equal weights.\footnote{Its parameters are $\texttt{p\_twist} = \texttt{p\_end} = 1$, $\texttt{max\_twists} = 5$, $\texttt{max\_actions} = 30$, and, in the slice frameworks, $\texttt{p\_birth} = 0.1$.} 
This agent is useful as a baseline.
The metric is the accuracy on the test datasets, which were untouched by model selection and parameter tuning.
Each agent was run five times with different random seeds, and we report the mean accuracy over those runs, together with its standard deviation.

Because all four agents were evaluated on the same links, we computed their differences link-by-link. This paired comparison controls for variation in the intrinsic difficulty of individual links.
Links of the same dataset may share a prime summand, however, so their differences are not independent.
We therefore estimate the standard error of the mean difference allowing for correlation between any two links that share a prime. 
Writing $e_i$ for the deviation of the difference on link $i$ from its mean over the $N$ links of the dataset, we set
\[
    \mathrm{SE}^2 = \frac{1}{N^2} \sum_{i} \sum_{j \sim i} e_i e_j,
\]
where $j \sim i$ means that $j = i$ or that the links $i$ and $j$ share a prime.
This reduces to the usual standard error when no two links share a prime.
Clustering the links by the connected components of the relation $\sim$ is not possible here, as in the strong frameworks, connected sums join all the test links into a single component.
Accounting for the randomness of the agents, the finite size of the test sets, and the shared primes, the standard error of the differences discussed below lies between $0.8\%$ and $4.3\%$, and every one of these differences is at least $3.7$ standard errors away from zero.

\paragraph{Benchmark result.}
Table~\ref{tab:rl-vs-rw-results} summarises the benchmark results.
The ordering is the same in every framework: $\texttt{naive\_rw} \ll \texttt{bo\_rw} \ll \texttt{rl} < \texttt{rl\_rw}$.
Table~\ref{tab:rl-vs-rw-steps} reports the three successive accuracy differences.

\begin{table}[ht]
    \centering
    \begin{tabular}{|l|c|c|c|c|c|}
        \hline
        \multirow{2}{*}{Framework}
            & \multicolumn{4}{c|}{Accuracy (\%) of} & \multirow{2}{*}{Best agent} \\
        & \texttt{naive\_rw} & \texttt{bo\_rw} & \texttt{rl} & \texttt{rl\_rw} & \\
        \hline
        \texttt{unknotting}    & $8.4 \pm 0.8$  & $58.6 \pm 0.9$ & $82.4 \pm 0.8$ & $\mathbf{87.4 \pm 0.7}$ & \texttt{rl\_rw} \\
        \texttt{ribbon}        & $10.2 \pm 0.9$ & $54.6 \pm 1.1$ & $73.3 \pm 0.4$ & $\mathbf{82.1 \pm 0.5}$ & \texttt{rl\_rw} \\
        \texttt{slice}         & $9.8 \pm 0.5$  & $55.2 \pm 2.2$ & $68.9 \pm 0.5$ & $\mathbf{80.6 \pm 1.1}$ & \texttt{rl\_rw} \\
        \texttt{strong ribbon} & $1.2 \pm 0.3$  & $12.7 \pm 1.2$ & $37.9 \pm 1.7$ & $\mathbf{51.9 \pm 0.8}$ & \texttt{rl\_rw} \\
        \texttt{strong slice}  & $1.4 \pm 0.5$  & $12.1 \pm 0.7$ & $27.6 \pm 1.3$ & $\mathbf{36.7 \pm 1.2}$ & \texttt{rl\_rw} \\
        \hline
    \end{tabular}
    \caption{Accuracy of the four agents on the test set of each framework, expressed as percentages, with mean and standard deviation over five runs at $800$ units per link.
    The highest mean accuracy is in bold.}
    \label{tab:rl-vs-rw-results}
\end{table}

\begin{table}[ht]
    \centering
    \begin{tabular}{|l|c|c|c|}
        \hline
        \multirow{2}{*}{Framework} & Tuning the walker & RW $\rightarrow$ RL & RL $\rightarrow$ RL+RW \\
            & ($\texttt{bo\_rw} - \texttt{naive\_rw}$) & ($\texttt{rl} - \texttt{bo\_rw}$) & ($\texttt{rl\_rw} - \texttt{rl}$) \\
        \hline
        \texttt{unknotting}    & $+50.2$ & $+23.8$ & $+5.1$ \\
        \texttt{ribbon}        & $+44.4$ & $+18.7$ & $+8.8$ \\
        \texttt{slice}         & $+45.4$ & $+13.8$ & $+11.7$ \\
        \texttt{strong ribbon} & $+11.6$ & $+25.2$ & $+14.0$ \\
        \texttt{strong slice}  & $+10.6$ & $+15.5$ & $+9.2$ \\
        \hline
    \end{tabular}
    \caption{Accuracy differences between the various models, in percentage points.
    Recall that the margin of error on these differences lies between $0.8$ and $4.3$ points.}
    \label{tab:rl-vs-rw-steps}
\end{table}

Tuning the random walker improves accuracy by between $+10.6\%$ and $+50.2\%$ over the naive walker, while replacing the tuned walker with the RL policy adds a further $+13.8\%$ to $+25.2\%$. Both steps are significant in every framework.
On \texttt{unknotting}, \texttt{ribbon}, and \texttt{slice}, where the tuned random walker already solves more than half of the links, tuning it is the largest improvement.
In the two strong frameworks, the learned policy contributes more than the tuning.

Mixing the walker into the policy improves accuracy further on every framework, by between $+5.1\%$ on \texttt{unknotting} and $+14.0\%$ on \texttt{strong ribbon}.
As expected from the discussion in Section~\ref{subsubsec:mixed-explanation}, on every framework except \texttt{strong ribbon} the gain is large on the links with the lowest invariant value, where the pure policy is less accurate than the tuned random walker, and the mixed agent is more accurate than both.
The mixed agent also gains on the other classes of links and loses significantly on none.

A higher accuracy, however, does not mean that the mixed agent solves every link that the RL policy solves.
The two agents solve measurably different sets of links: in a single run on \texttt{ribbon}, $100$ test links are solved only by the mixed agent and $29$ only by the pure RL policy.

\subsubsection{Equal computing time}
\label{subsec:equal-time}

The benchmark above gives every agent the same number of steps, not the same computing time.
A step of the RL policy is more expensive than a step of the random walker, since it requires the computation of the input features and a forward pass of the policy network.
On the machines used for the benchmark, a policy step costs between $4.0$ and $6.8$ times a random walker step, depending on the framework.
The mixed agent costs no more than the pure RL policy.

To check whether the ranking survives when the computing time is matched, we repeated the comparison on the test datasets, giving the tuned random walker the number of steps per link that matches, on the same machine, the running time of the RL policy at $800$ steps per link.
The random walker thus received between about $3{,}500$ and $5{,}700$ steps per link on average, depending on the framework, while the RL policy and the mixed agent kept $800$.
The results are in Table~\ref{tab:equal-time}.

\begin{table}[H]
    \centering
    \small
    \setlength{\tabcolsep}{4pt}
    \begin{tabular}{|l|c|c|c|c|c|c|}
        \hline
        \multirow{2}{*}{Framework} & \texttt{bo\_rw} steps & \multicolumn{3}{c|}{Accuracy (\%) of} & \multirow{2}{*}{$\texttt{bo\_rw} - \texttt{rl}$} & \multirow{2}{*}{$\texttt{bo\_rw} - \texttt{rl\_rw}$} \\
            & per link & \texttt{bo\_rw} & \texttt{rl} & \texttt{rl\_rw} & & \\
        \hline
        \texttt{unknotting}    & $3452$ & $81.6 \pm 1.0$ & $82.3 \pm 0.8$ & $\mathbf{87.4 \pm 0.7}$ & $-0.7 \pm 1.6$  & $-5.9 \pm 1.4$ \\
        \texttt{ribbon}        & $5196$ & $\mathbf{84.3 \pm 0.9}$ & $73.3 \pm 0.4$ & $82.1 \pm 0.5$ & $+11.0 \pm 2.1$ & $+2.2 \pm 1.8$ \\
        \texttt{slice}         & $5725$ & $\mathbf{84.4 \pm 0.5}$ & $68.9 \pm 0.5$ & $80.6 \pm 1.1$ & $+15.5 \pm 1.9$ & $+3.8 \pm 1.5$ \\
        \texttt{strong ribbon} & $5088$ & $29.4 \pm 1.0$ & $37.9 \pm 1.7$ & $\mathbf{51.9 \pm 0.8}$ & $-8.5 \pm 4.8$  & $-22.5 \pm 4.7$ \\
        \texttt{strong slice}  & $3999$ & $25.7 \pm 1.5$ & $27.6 \pm 1.3$ & $\mathbf{36.7 \pm 1.2}$ & $-1.9 \pm 4.6$  & $-11.1 \pm 4.1$ \\
        \hline
    \end{tabular}
    \caption{The agents on the test set when the tuned random walker is given the computing time of the RL policy.
    Accuracies are mean and standard deviation over five runs. 
    Differences are in percentage points, with the standard error of the paired comparison link-by-link, allowing for links that share a prime.}
    \label{tab:equal-time}
\end{table}

On the two strong frameworks, the mixed agent remains the best, $22.5\%$ and $11.1\%$ ahead of the walker, and on \texttt{unknotting}, it leads by $5.9\%$.
On \texttt{ribbon} and \texttt{slice}, instead, the random walker with matched time is at least as accurate as the mixed agent, by $2.2\%$ and $3.8\%$, and much more accurate than the pure RL policy, by $11.0\%$ and $15.5\%$, respectively.
On \texttt{ribbon}, the random walker's lead over the mixed agent is only $1.2$ standard errors, so we consider it a tie.
The pure RL policy is not significantly more accurate than the tuned random walker on any framework.
With the matched computing time on \texttt{unknotting} and \texttt{strong slice}, the two agents differ by less than one standard error.
On \texttt{strong ribbon}, the pure policy is ahead of the tuned walker by $8.5\%$, which is only $1.8$ standard errors.

By looking at the link classes, we can explain this difference.
The random walker wins on links of genus zero and one, and on those of unknotting number at most three.
The policy wins where the answer requires a long and specific sequence of moves. 
On the strong frameworks, the walker solves none of the $20$ test links of genus five, even with more than four times as many steps, while the policy solves between $69\%$ and $89\%$ of them, depending on the framework and on the mixing.
On \texttt{ribbon} and \texttt{slice}, $44\%$ of the test links have genus zero or one.

In general, we run both the random walker and the mixed agent on the same dataset to maximise the chance of finding new upper bounds.

\section{New bounds}
\label{sec:results}

We report, per database tier, how much the known values were improved, splitting the contribution into \emph{theoretical} (the bounds of Appendix~\ref{sec:bounds}) and \emph{agentic} (an explicit band path found by the RW or RL$+$RW agent).
We denote a value to be \emph{exact} when its lower and upper bounds coincide, \emph{bounded} when it is confined to an interval, and \emph{unknown} when nothing is on record.
A value is agentic only when an agent's witness strictly improved the previously known upper bound.

Each tier below is therefore split into two parts, one per stage.
First, the \emph{theoretical bounds} part: beginning from the starting knowledge provided by the external databases, we apply a bounds-only pipeline to it.
This applies all the results reported in Appendix~\ref{sec:bounds} and verifies if better lower and upper bounds exist.
Second, the \emph{agent bounds} part: this starts where the theoretical pipeline ends and reports what the agent adds on top.

\subsection{Building the reference bounds dataset}
\label{sec:dataset}

Several public sources report values for slice genera and unknotting numbers, and they do not always agree. 
For some links, one source's bounds are sharper than another's.
We therefore assembled a single, up-to-date dataset with the best known bounds on the invariants we target by pulling together every public knot/link database we are aware of and merging them with explicit provenance.
This dataset plays two roles: it is the reference against which any improvement reported below is measured, and it is what the agents are trained and evaluated on.

\paragraph{Sources.}
We ingest the databases below, splitting them into three sections: knots up to $13$ crossings, $14$-crossing knots, and at least two-component links up to 11 crossings.

\begin{itemize}
    \item \cite{knotinfo} and \cite{linkinfo}, our main references: data on the unknotting number for prime knots up to $13$ crossings and at least two-component prime links up to $10$ crossings, and on the slice genus for prime knots up to $13$ crossings and at least two-component prime links up to $11$ crossings;
    \item \cite{knotatlas}: data on the unknotting number for prime knots up to 11 crossings and the slice genus for prime knots up to 10 crossings;
    \item \cite{brittenham2017unknottingsearch}: data on the unknotting number for prime knots up to 14 crossings;
    \item \cite{dunfield2025ribbonconslice}: data on the slice genus for prime knots up to 19 crossings, where they determine whether a knot is slice, not slice ($\gen{} \geq 1$), or undetermined ($\gen{} \geq 0$);
    \item \cite{jablonowski2026inequalities}: data on the unknotting number and the slice genus for prime knots up to 13 crossings;
    \item \cite{lee2026computationunknottingnumbersknot}: data on the unknotting number for prime knots up to 13 crossings.
\end{itemize}

KnotInfo for prime knots up to 13 crossings and LinkInfo for at least two-component prime links up to 11 crossings are taken as the main references. 
In particular, we use names, PD codes, and orientation conventions from there.
For $14$-crossing prime knots, the SnapPy census supplies the names and PD codes for all $46{,}972$ knots.
Crossing signs are part of the data: a \texttt{(PD code, crossing signs)} pair uniquely determines the oriented diagram, and signs are preserved across all computations.

\paragraph{Merging the datasets.}
After collecting all the data, we merge the various sources together.
This means that for an invariant whose exact value is not known and for which we have only an upper and a lower bound, we compare the different sources and keep the best bounds, recording for each invariant a \texttt{\_lower\_bound\_source} and \texttt{\_upper\_bound\_source} provenance column naming which database set the tightest bound on each side.

\paragraph{Discrepancies in the datasets.}
During the merge, two discrepancies emerged among the databases.
First, the unknotting number of $\mathtt{10_{6}}$ is recorded as having unknotting number 3 in various databases, while from \cite{brittenham2026unknottingnumberconnectedsums4_1_5_1} we know its value is either $2$ or $3$.

Second, the unknotting numbers of three further knots disagree between \cite{knotinfo} and \cite{brittenham2017unknottingsearch}: these are $\mathtt{10_{162}}, \mathtt{10_{164}}$, and $\mathtt{10_{165}}$.
We think this is due to the presence of the Perko pair in Brittenham's database, and we adopt the KnotInfo data in the merged release.

\paragraph{Narrowing by theoretical bounds.}
\label{subs:gt-bounds}
After merging the various references into a single dataset, we ran a bounds-only pipeline that narrows the recorded intervals where possible using the classical theoretical bounds discussed in Appendix~\ref{sec:bounds}.
Note that since $\gen \leq \unk$ and $\gen \leq \sgen$, all the lower bounds for the slice genus are also used for the strong slice genus and the unknotting number.
For a positive diagram, Theorem~\ref{thm:positive-links-genus} determines the slice genus exactly, which is a lower bound on the unknotting number. 
Whenever it matches the recorded unknotting upper bound, we can determine the unknotting number exactly.

\paragraph{Narrowing via agents.}
The narrowed dataset is the input that the agents start from, so any value the agent later pins down is genuinely beyond what these classical bounds already provide.
We present in the next sections the results we obtained from the theoretical bounds and from the agents' upper bounds, splitting the improvements by contribution.

\subsection{At least two-component prime links with up to 11 crossings}
\label{subs:updated-links}

In the dataset, there are $4{,}188$ oriented prime links with at least two components ($1{,}424$, disregarding orientation).
Note that we did not have any value for the strong slice genus of any link, nor the unknotting number for all the $11$-crossings links ($2{,}920$ oriented links).

\subsubsection{Theoretical bounds}

The bounds-only pipeline acts unevenly across the three invariants.
For the unknotting number, it works almost entirely on the $2{,}924$ links carrying no recorded value at all. 
It determines $1{,}132$ of them outright and leaves the other $1{,}792$ with a bounded interval, so that after this stage, no unknotting number is completely unknown.
For the slice genus, which starts out with only $407$ non-exact cases, it determines $19$ of them.
For the strong slice genus, which no database tabulates, it supplies the entire starting point. 
$3{,}456$ of the links have two components with non-zero pairwise linking number, hence $\sgen{} = \infty$ (counted as an exact value).
The remaining $732$ have vanishing pairwise linking numbers and therefore finite $\sgen{}$. 
These remain all non-exact after the bounds are computed.

\begin{table}[ht]
    \centering
    \small
    \setlength{\tabcolsep}{5pt}
    \renewcommand{\arraystretch}{1.15}
    \begin{tabular}{@{}lrrrrrrr@{}}
        \toprule
        & \multicolumn{3}{c}{Start}
        & \multicolumn{2}{c}{Theoretical contribution}
        & \multicolumn{2}{c}{After theory} \\
        \cmidrule(lr){2-4} \cmidrule(lr){5-6} \cmidrule(lr){7-8}
        Invariant & exact & bounded & unknown & made exact & better bounds & exact & bounded \\
        \midrule
        unknotting $\unk$
            & $1{,}264$ & $0$ & $2{,}924$
            & $1{,}132$ & $1{,}792$
            & $2{,}396$ & $1{,}792$ \\
        \addlinespace[2pt]
        slice genus $\gen$
            & $3{,}781$ & $407$ & $0$
            & $19$ & $0$
            & $3{,}800$ & $388$ \\
        \addlinespace[2pt]
        strong slice genus $\sgen$
            & $0$ & $0$ & $4{,}188$
            & $3{,}456$ & $732$
            & $3{,}456$ & $732$ \\
        \bottomrule
    \end{tabular}
    \caption{At least two-component oriented prime links with up to $11$ crossings: contribution of the theoretical bounds.}
    \label{tab:links-theory}
\end{table}

\subsubsection{Agent bounds}

Starting from the narrowed dataset above, the best-performing RW and RL$+$RW agents for each framework try to find explicit movies that realise a lower invariant value than the known upper bound.
Table~\ref{tab:links-rl} shows the improvements made by the agents. When studying the slice genera, allowing the birth of unlinked unknot components did not improve our upper bounds. Hence, we report only on the slice genus.

\begin{table}[ht]
    \centering
    \small
    \setlength{\tabcolsep}{6pt}
    \renewcommand{\arraystretch}{1.15}
    \begin{tabular}{@{}lrrrrrr@{}}
        \toprule
        & \multicolumn{2}{c}{After theory}
        & \multicolumn{2}{c}{Agent contribution}
        & \multicolumn{2}{c}{Final} \\
        \cmidrule(lr){2-3} \cmidrule(lr){4-5} \cmidrule(lr){6-7}
        Invariant & exact & bounded & made exact & better bounds & exact & bounded \\
        \midrule
        unknotting $\unk$
            & $2{,}396$ & $1{,}792$
            & $846$ & $934$
            & $3{,}242$ & $946$ \\
        \addlinespace[2pt]
        slice genus $\gen$
            & $3{,}800$ & $388$
            & $125$ & $0$
            & $3{,}925$ & $263$ \\
        \addlinespace[2pt]
        strong slice genus $\sgen$
            & $3{,}456$ & $732$
            & $676$ & $56$
            & $4{,}132$ & $56$ \\
        \bottomrule
    \end{tabular}
    \caption{At least two-component oriented prime links with up to $11$ crossings: contribution of the agent, on top of Table~\ref{tab:links-theory}.}
    \label{tab:links-rl}
\end{table}

\subsubsection{Comments}
We now briefly comment and provide some statistics on the results found for each invariant.

\paragraph{Unknotting number ($11$-crossing links).}
No table of unknotting numbers for $11$-crossing at least two-component prime links was previously available ($2{,}920$ oriented links).
Disregarding orientation, of the $1{,}007$ unique $11$-crossing at least two-component prime links, we determine the exact value for $690$ and confine each of the remaining $317$ to an interval.
The $2{,}920$ oriented $11$-crossing links split into $1{,}978$ exact values and $942$ intervals, of which $810$ have width one and $132$ width two.

\paragraph{Weak slice genus.}
The slice genus was undetermined for $407$ links.
The theoretical bounds settle $19$ of them, and the agent a further $125$, all of which turn out to be slice. 
Of the latter, $4$ links have $9$ crossings, $23$ have $10$, and $98$ have $11$.
The $263$ links still undetermined are all confined to $\gen{} \in \{0,1\}$; i.e., the only open question left is whether they are slice.

\paragraph{Strong slice genus.}
Of the $732$ links left with finite strong slice genus, the agent determines the exact value for $676$ and strictly narrows the interval of the other $56$.
Out of those $56$, $54$ links are confined to $\sgen \in \{1,2\}$ and $2$ links to $\sgen \in \{0,1\}$.
Table~\ref{tab:updated-strong-exact} collects the strong slice genus of the $676$ links for which we determined an exact value.

\begin{table}[H]
    \centering
    \small
    \setlength{\tabcolsep}{8pt}
    \renewcommand{\arraystretch}{1.15}
    \begin{tabular}{@{}crrrrr@{}}
        \toprule
        $\sgen{}$ & $0$ & $1$ & $2$ & $3$ & $4$ \\
        \midrule
        \# links & $4$ & $388$ & $238$ & $44$ & $2$ \\
        \bottomrule
    \end{tabular}
    \caption{The $676$ links of finite strong slice genus whose value was determined exactly.}
    \label{tab:updated-strong-exact}
\end{table}

\subsection{Prime knots with up to 13 crossings}
\label{subs:updated-knots13}

There are $12{,}965$ prime knots with up to 13 crossings.
Neither stage improved on the known data. 
The theoretical bounds reproduce the tabulated values without sharpening any of them, and no agent beat a known upper bound.
The agents did reproduce the tabulated upper bound on $12{,}865$ ($99.2 \%$ of the total) of the knots for the unknotting number and on $12{,}964$ ($99.9 \%$ of the total) for the slice genus.
On the remaining $100$ knots (unknotting number) and on $12a_{631}$ (ribbon genus), the best answer stayed one above the tabulated value.

\subsection{Prime knots with 14 crossings}
\label{subs:updated-knots14}

There are $46{,}972$ prime knots with 14 crossings.
Here, the two stages divide the work by invariant. 
The improvements on the slice genus side come almost entirely from the agent, while on the unknotting number side they come entirely from theory.

\subsubsection{Theoretical bounds}

While no slice genus dataset for $14$-crossing knots was available, Dunfield and Gong \cite{dunfield2025ribbonconslice} do report sliceness for $1{,}194$ of them, and slice obstruction for $408$, which we record as the bound $\gen{} \geq 1$.
That is the whole of the starting knowledge, and the only exact value it supplies is $\gen{} = 0$.
The theoretical bounds then determine a further $1{,}671$ values exactly and strictly narrow the interval of $44{,}107$ more, bringing the census to $2{,}865$ exact values.

For the unknotting number, the upper bounds of all $46{,}972$ knots come from \cite{brittenham2017unknottingsearch}, and the new contribution is purely on the \emph{lower} bound side. 
The obstructions give a non-trivial lower bound for $1{,}288$ knots, of which $1{,}097$ become exactly determined, the new bound meeting Brittenham's upper bound.
Table~\ref{tab:knots14-theory} shows the contribution of the theoretical pipeline.
Note that no value remains completely unbounded after the bounds are computed.

\begin{table}[H]
    \centering
    \small
    \setlength{\tabcolsep}{6pt}
    \renewcommand{\arraystretch}{1.15}
    \begin{tabular}{@{}lrrrrrrr@{}}
        \toprule
        & \multicolumn{3}{c}{Start}
        & \multicolumn{2}{c}{Theoretical contribution}
        & \multicolumn{2}{c}{After theory} \\
        \cmidrule(lr){2-4} \cmidrule(lr){5-6} \cmidrule(lr){7-8}
        Invariant & exact & bounded & unknown & made exact & better bounds & exact & bounded \\
        \midrule
        slice genus $\gen$
            & $1{,}194$ & $408$ & $45{,}370$
            & $1{,}671$ & $44{,}107$
            & $2{,}865$ & $44{,}107$ \\
        \addlinespace[2pt]
        unknotting $\unk$
            & $19{,}781$ & $27{,}191$ & $0$
            & $1{,}097$ & $191$
            & $20{,}878$ & $26{,}094$ \\
        \bottomrule
    \end{tabular}
    \caption{Prime knots with $14$ crossings: contribution from the theoretical bounds.}
    \label{tab:knots14-theory}
\end{table}

\subsubsection{Agent bounds}

The agent contributes only to the slice genus. 
It makes exact $43{,}167$ of the $44{,}107$ intervals left open above and strictly narrows $914$ more, bringing the census to an \emph{exact} slice genus for $46{,}032$ knots, with each of the remaining $940$ narrowed to an interval of at most two values.
Again, allowing the birth of unlinked unknot components did not improve our upper bounds.
For the unknotting number, it finds no unknotting sequence that improves on Brittenham's upper bounds, so that row remains unchanged.

\begin{table}[H]
    \centering
    \small
    \setlength{\tabcolsep}{6pt}
    \renewcommand{\arraystretch}{1.15}
    \begin{tabular}{@{}lrrrrrr@{}}
        \toprule
        & \multicolumn{2}{c}{After theory}
        & \multicolumn{2}{c}{Agent contribution}
        & \multicolumn{2}{c}{Final} \\
        \cmidrule(lr){2-3} \cmidrule(lr){4-5} \cmidrule(lr){6-7}
        Invariant & exact & bounded & made exact & better bounds & exact & bounded \\
        \midrule
        slice genus $\gen$
            & $2{,}865$ & $44{,}107$
            & $43{,}167$ & $914$
            & $46{,}032$ & $940$ \\
        \addlinespace[2pt]
        unknotting $\unk$
            & $20{,}878$ & $26{,}094$
            & --- & ---
            & $20{,}878$ & $26{,}094$ \\
        \bottomrule
    \end{tabular}
    \caption{Prime knots with $14$ crossings: what the agent contributes on top of Table~\ref{tab:knots14-theory}.}
    \label{tab:knots14-rl}
\end{table}

\subsubsection{Comments}
We now briefly comment on and provide some statistics regarding the results found for each invariant.

\paragraph{Unknotting number lower bounds source.}
The sources of the new theoretical bounds that improve the previously known values are summarised in Table~\ref{tab:updated-14-unknotting}.
\begin{table}[H]
    \centering
    \small
    \setlength{\tabcolsep}{8pt}
    \renewcommand{\arraystretch}{1.15}
    \begin{tabular}{@{}lr@{}}
        \toprule
        New lower bound source & \# knots \\
        \midrule
        $\tau$ and $\nu_s$ & $957$ \\
        positive and $g_4 =$ upper bound & $327$ \\
        Murasugi--Tristram & $4$ \\
        \midrule
        total & $1{,}288$ \\
        \bottomrule
    \end{tabular}
    \caption{New lower bounds on the unknotting number for $14$-crossing prime knots, with source.}
    \label{tab:updated-14-unknotting}
\end{table}

\paragraph{Slice genus distribution.}
In the following tables, we can observe the distribution of the values found for the slice genus.
Table~\ref{tab:updated-14-exact} summarises the $46{,}032$ exact values, where the $1{,}194$ knots with $\gen{} = 0$ are exactly the slice knots of Dunfield and Gong~\cite{dunfield2025ribbonconslice}; every value above $0$ is new.

\begin{table}[H]
    \centering
    \small
    \setlength{\tabcolsep}{6pt}
    \renewcommand{\arraystretch}{1.15}
    \begin{tabular}{@{}lrrrrrrr@{}}
        \toprule
        $\gen{}$ & $0$ & $1$ & $2$ & $3$ & $4$ & $5$ & $6$ \\
        \midrule
        \# knots & $1{,}194$ & $29{,}531$ & $10{,}178$ & $3{,}822$ & $1{,}088$ & $202$ & $17$ \\
        \bottomrule
    \end{tabular}
    \caption{The $46{,}032$ knots with $14$ crossings whose slice genus was determined exactly.}
    \label{tab:updated-14-exact}
\end{table}

Table~\ref{tab:updated-14-bounded} summarises the intervals of values left undetermined for the $940$ remaining knots by counting how many knots have each interval.

\begin{table}[H]
    \centering
    \small
    \setlength{\tabcolsep}{8pt}
    \renewcommand{\arraystretch}{1.15}
    \begin{tabular}{@{}lrrr@{}}
        \toprule
        interval & $[0,1]$ & $[1,2]$ & $[2,3]$ \\
        \midrule
        \# knots & $10$ & $919$ & $11$ \\
        \bottomrule
    \end{tabular}
    \caption{The $940$ knots with $14$ crossings whose slice genus was narrowed to a short interval, but not determined exactly.}
    \label{tab:updated-14-bounded}
\end{table}

\subsection{Non-additivity of the unknotting number}
\label{subs:unknotting-additivity}

Our original motivation for developing the unknotting framework was to find counterexamples to the additivity of the unknotting number.
Following the results of Brittenham and Hermiller~\cite{brittenham2026unknottingnumberconnectedsums4_1_5_1}, we ran the unknotting agents on some connected sums of \emph{symbiont knots}.
In their paper, they call two knots $K$ and $K^{\prime}$ symbionts if $\unk{K \sharp K^{\prime}} < \unk{K} + \unk{K^{\prime}}$.
They showed that $\mathtt{4_{1}}$ and $\mathtt{9_{10}}$ are symbionts; see \cite[Theorem~1.1]{brittenham2026unknottingnumberconnectedsums4_1_5_1}. 
Furthermore, $\mathtt{4_1}$ is Gordian adjacent to $\mathtt{8_{4}}$, $\mathtt{9_{8}}$, $\mathtt{9_{17}}$, $\mathtt{9_{40}}$, $\mathtt{10_{41}}$, $\mathtt{10_{89}}$, $\mathtt{10_{115}}$, $\mathtt{10_{121}}$, $\mathtt{10_{140}}$, and $\mathtt{10_{144}}$, as $\mathtt{4_{1}}$ is contained in a minimal unknotting sequence for each of them; see the set $S_2$ in \cite{brittenham2026unknottingnumberconnectedsums4_1_5_1}.
It now follows from \cite[Lemma~4.1]{brittenham2026unknottingnumberconnectedsums4_1_5_1} that $\mathtt{8_{4}}$, $\mathtt{9_{8}}$, $\mathtt{9_{17}}$, $\mathtt{9_{40}}$, $\mathtt{10_{41}}$, $\mathtt{10_{89}}$, $\mathtt{10_{115}}$, $\mathtt{10_{121}}$, $\mathtt{10_{140}}$, and $\mathtt{10_{144}}$ are symbionts of $\mathtt{9_{10}}$.
Since $\mathtt{4_1}$ is amphichiral, $\mir{\mathtt{9_{10}}}$ is a symbiont of $\mathtt{4_1}$ as well, and hence of the same knots.

Our unknotting agent found that
\begin{gather*}
\mathtt{8_{4}} \sharp \mathtt{9_{10}}, \mathtt{9_{8}} \sharp \mathtt{9_{10}}, \mathtt{9_{17}} \sharp \mathtt{9_{10}}, \mathtt{9_{40}} \sharp \mathtt{9_{10}}, \mathtt{10_{41}} \sharp \mathtt{9_{10}}, \mathtt{10_{89}} \sharp \mathtt{9_{10}}, \\
\mathtt{10_{115}} \sharp \mathtt{9_{10}}, \mathtt{10_{121}} \sharp \mathtt{9_{10}}, \mathtt{10_{140}} \sharp \mathtt{9_{10}}, \mathtt{10_{144}} \sharp \mathtt{9_{10}}, \mathtt{10_{41}} \sharp \mir{\mathtt{9_{10}}}
\end{gather*}
all have unknotting number at most four. As $\mathtt{8_{4}}$, $\mathtt{9_{8}}$, $\mathtt{9_{17}}$, $\mathtt{9_{40}}$, $\mathtt{10_{41}}$, $\mathtt{10_{89}}$, $\mathtt{10_{115}}$, $\mathtt{10_{121}}$, $\mathtt{10_{140}}$, and $\mathtt{10_{144}}$ have unknotting number two, and $\unk{\mathtt{9_{10}}} = \unk{\mir{\mathtt{9_{10}}}} = 3$, we recover the non-additivity of the unknotting number for the above connected sums.
Some of the prime summands were present in the datasets on which we trained or tested the models, while the exact connected sums were not.
The exact presence of the knots in the splits is reported in~Table~\ref{tab:symbionts-in-datasets}.

\begin{table}[ht]
\centering
\begin{tabular}{ll}
Split & Knot \\
\hline
training      & $\mathtt{8_{4}} \sharp \mir{\mathtt{9_{25}}}$   \\
           & $\mathtt{9_{8}} \sharp \mathtt{9_{38}}$       \\
           & $\mathtt{10_{115}} \sharp \mathtt{10_{124}}$  \\
\greyrule
validation        & $\mathtt{10_{41}}$                           \\
           & $\mathtt{10_{41}} \sharp \mathtt{10_{74}}$     \\
\greyrule
coarse validation & --- \\
\greyrule
development        & ---  \\
\greyrule
testing       & $\mathtt{10_{140}} \sharp \mathtt{10_{152}}$  \\
           & $\mathtt{9_{10}} \sharp \mathtt{10_{84}}$      \\
           & $\mathtt{9_{10}} \sharp \mir{\mathtt{10_{14}}}$   \\
           & $\mathtt{9_{10}} \sharp \mir{\mathtt{10_{152}}}$  \\
\greyrule
\end{tabular}
\caption{Rows of the unknotting splits containing $\mathtt{8_{4}}$, $\mathtt{9_{8}}$, $\mathtt{9_{10}}$, $\mathtt{9_{17}}$, $\mathtt{9_{40}}$, $\mathtt{10_{41}}$, $\mathtt{10_{89}}$, $\mathtt{10_{115}}$, $\mathtt{10_{121}}$, $\mathtt{10_{140}}$, or $\mathtt{10_{144}}$ (either chirality) as the knot itself or as a connected summand.}
\label{tab:symbionts-in-datasets}
\end{table}

\paragraph{Unknotting sequences.}
We comment briefly on the unknotting sequences we found.
Each band but one is of length one; i.e., the agent performs a crossing change by passing the band through an edge contained in the boundary of the same face where it started the band.
The exception is the last band of the sequence for $\mathtt{10_{41}} \sharp \mir{\mathtt{9_{10}}}$, which first passes under an edge and ends in the adjacent face.
After each band, the diagram is then simplified before starting a new one.
In half of the bands, the band corresponds to exactly one diagrammatic crossing change.

Recall the unknotting sequence proposed in the proof of~\cite[Theorem~1.1]{brittenham2026unknottingnumberconnectedsums4_1_5_1}, which is
\begin{equation}
\label{eq:unknotting-sequence-briher}
     \mathtt{4_1} \sharp \mathtt{9_{10}} \to \mir{\mathtt{K15n4866}} \to \mathtt{8_{14}} \to \text{unknot}.
\end{equation}
All the unknotting sequences we found reach $\mathtt{8_{14}}$, or its mirror for $\mathtt{10_{41}} \sharp \mir{\mathtt{9_{10}}}$, after three crossing changes and then the unknot.
The sequences for $\mathtt{8_{4}} \sharp \mathtt{9_{10}}$, $\mathtt{9_{17}} \sharp \mathtt{9_{10}}$, $\mathtt{10_{89}} \sharp \mathtt{9_{10}}$, and $\mathtt{10_{144}} \sharp \mathtt{9_{10}}$ find exactly $\mathtt{4_{1}} \sharp \mathtt{9_{10}}$ with the first crossing change and then follow the same sequence \eqref{eq:unknotting-sequence-briher}.
The sequence for $\mathtt{10_{41}} \sharp \mir{\mathtt{9_{10}}}$ is the mirror image of the one for $\mathtt{10_{89}} \sharp \mathtt{9_{10}}$: it passes through $\mathtt{4_{1}} \sharp \mir{\mathtt{9_{10}}}$, $\mathtt{K15n4866}$, and $\mir{\mathtt{8_{14}}}$.
The sequences for $\mathtt{9_{8}} \sharp \mathtt{9_{10}}$ and $\mathtt{10_{115}} \sharp \mathtt{9_{10}}$ also find $\mathtt{4_{1}} \sharp \mathtt{9_{10}}$ with the first crossing change, but then pass through a different knot, $\mathtt{K15n4005}$.
This gives a second sequence of three crossing changes from $\mathtt{4_{1}} \sharp \mathtt{9_{10}}$ to the unknot and shows that $\unk{\mathtt{K15n4005}} = 2$, since $\abs{\tau(\mathtt{K15n4005})} = 2$.
The sequences for $\mathtt{9_{40}} \sharp \mathtt{9_{10}}$ and $\mathtt{10_{41}} \sharp \mathtt{9_{10}}$ reach $\mathtt{K15n4005}$ after two crossing changes, passing through a knot that is not in the SnapPy tables of knots up to 15 crossings.
On $\mathtt{10_{140}} \sharp \mathtt{9_{10}}$, the mixed agent reaches $\mir{\mathtt{K15n4866}}$ after two crossing changes without passing through $\mathtt{4_{1}} \sharp \mathtt{9_{10}}$, while the random walker paths on $\mathtt{10_{140}} \sharp \mathtt{9_{10}}$ and $\mathtt{10_{121}} \sharp \mathtt{9_{10}}$ pass through two knots not in the tables and have only the last step ($\mathtt{8_{14}}$ to the unknot) in common with sequence~\eqref{eq:unknotting-sequence-briher}.
All the unknotting sequences are collected in Table~\ref{tab:symbionts-unknotting-values}.
The sequences that differ from~\eqref{eq:unknotting-sequence-briher} are drawn step by step in Appendix~\ref{app:sequences}.

\begin{table}[ht]
\centering
\newcommand{\summands}[1]{{\footnotesize\color{black!60}$(#1)$}}
\begin{tabular}{l@{\qquad}c@{\enspace}c@{\enspace}c@{\enspace}c@{\enspace}c@{\enspace}c@{\enspace}c@{\enspace}c@{\enspace}c}
\hline
Agent & \multicolumn{9}{c}{Unknotting sequence} \\
\hline
RW    & $\mathtt{8_{4}} \sharp \mathtt{9_{10}}$ & $\to$ & $\mathtt{4_{1}} \sharp \mathtt{9_{10}}$ & $\to$ & $\mir{\mathtt{K15n4866}}$ & $\to$ & $\mathtt{8_{14}}$ & $\to$ & unknot \\
      & $[1,4]$ \summands{2+3} & & $[2,3]$ \summands{1+3} & & $2$ & & $1$ & & $0$ \\
\greyrule
RW    & $\mathtt{9_{8}} \sharp \mathtt{9_{10}}$ & $\to$ & $\mathtt{4_{1}} \sharp \mathtt{9_{10}}$ & $\to$ & $\mathtt{K15n4005}$ & $\to$ & $\mathtt{8_{14}}$ & $\to$ & unknot \\
      & $[3,4]$ \summands{2+3} & & $[2,3]$ \summands{1+3} & & $2$ & & $1$ & & $0$ \\
\greyrule
RW    & $\mathtt{9_{17}} \sharp \mathtt{9_{10}}$ & $\to$ & $\mathtt{4_{1}} \sharp \mathtt{9_{10}}$ & $\to$ & $\mir{\mathtt{K15n4866}}$ & $\to$ & $\mathtt{8_{14}}$ & $\to$ & unknot \\
      & $[3,4]$ \summands{2+3} & & $[2,3]$ \summands{1+3} & & $2$ & & $1$ & & $0$ \\
\greyrule
RW    & $\mathtt{9_{40}} \sharp \mathtt{9_{10}}$ & $\to$ & $\mathtt{K20}$ & $\to$ & $\mathtt{K15n4005}$ & $\to$ & $\mathtt{8_{14}}$ & $\to$ & unknot \\
      & $[3,4]$ \summands{2+3} & & $[2,3]$ & & $2$ & & $1$ & & $0$ \\
\greyrule
RW    & $\mathtt{10_{41}} \sharp \mathtt{9_{10}}$ & $\to$ & $\mathtt{K21}$ & $\to$ & $\mathtt{K15n4005}$ & $\to$ & $\mathtt{8_{14}}$ & $\to$ & unknot \\
      & $[3,4]$ \summands{2+3} & & $[2,3]$ & & $2$ & & $1$ & & $0$ \\
\greyrule
RW    & $\mathtt{10_{89}} \sharp \mathtt{9_{10}}$ & $\to$ & $\mathtt{4_{1}} \sharp \mathtt{9_{10}}$ & $\to$ & $\mir{\mathtt{K15n4866}}$ & $\to$ & $\mathtt{8_{14}}$ & $\to$ & unknot \\
      & $[1,4]$ \summands{2+3} & & $[2,3]$ \summands{1+3} & & $2$ & & $1$ & & $0$ \\
\greyrule
RW    & $\mathtt{10_{115}} \sharp \mathtt{9_{10}}$ & $\to$ & $\mathtt{4_{1}} \sharp \mathtt{9_{10}}$ & $\to$ & $\mathtt{K15n4005}$ & $\to$ & $\mathtt{8_{14}}$ & $\to$ & unknot \\
      & $[2,4]$ \summands{2+3} & & $[2,3]$ \summands{1+3} & & $2$ & & $1$ & & $0$ \\
\greyrule
RW    & $\mathtt{10_{121}} \sharp \mathtt{9_{10}}$ & $\to$ & $\mathtt{K21}$ & $\to$ & $\mathtt{K23}$ & $\to$ & $\mathtt{8_{14}}$ & $\to$ & unknot \\
      & $[3,4]$ \summands{2+3} & & $3$ & & $2$ & & $1$ & & $0$ \\
\greyrule
RW    & $\mathtt{10_{140}} \sharp \mathtt{9_{10}}$ & $\to$ & $\mathtt{K21}$ & $\to$ & $\mathtt{K23}$ & $\to$ & $\mathtt{8_{14}}$ & $\to$ & unknot \\
      & $[2,4]$ \summands{2+3} & & $[2,3]$ & & $2$ & & $1$ & & $0$ \\
\greyrule
RL$+$RW & $\mathtt{10_{140}} \sharp \mathtt{9_{10}}$ & $\to$ & $\mathtt{K21}$ & $\to$ & $\mir{\mathtt{K15n4866}}$ & $\to$ & $\mathtt{8_{14}}$ & $\to$ & unknot \\
      & $[2,4]$ \summands{2+3} & & $[2,3]$ & & $2$ & & $1$ & & $0$ \\
\greyrule
RW    & $\mathtt{10_{144}} \sharp \mathtt{9_{10}}$ & $\to$ & $\mathtt{4_{1}} \sharp \mathtt{9_{10}}$ & $\to$ & $\mir{\mathtt{K15n4866}}$ & $\to$ & $\mathtt{8_{14}}$ & $\to$ & unknot \\
      & $[3,4]$ \summands{2+3} & & $[2,3]$ \summands{1+3} & & $2$ & & $1$ & & $0$ \\
\greyrule
RW    & $\mathtt{10_{41}} \sharp \mir{\mathtt{9_{10}}}$ & $\to$ & $\mathtt{4_{1}} \sharp \mir{\mathtt{9_{10}}}$ & $\to$ & $\mathtt{K15n4866}$ & $\to$ & $\mir{\mathtt{8_{14}}}$ & $\to$ & unknot \\
      & $[1,4]$ \summands{2+3} & & $[2,3]$ \summands{1+3} & & $2$ & & $1$ & & $0$ \\
\hline
\end{tabular}
\caption{The unknotting sequences found by the tuned random walker (RW) and by the mixed agent (RL$+$RW); each arrow is one crossing change.
The knot $\mathtt{Kn}$ is not in SnapPy's census up to 15 crossings, found as a diagram with $n$ crossings.
Knots with the same label $\mathtt{Kn}$ in different rows are different as they are hyperbolic and their volumes differ.
The knots $\mathtt{K15n4005}$ and $\mathtt{K15n4866}$ have the chirality of the diagrams SnapPy stores under these names. 
Below each knot is its unknotting number, or the interval known for it, whose the upper end is the number of crossing changes in the sequence, and whose lower end is the best lower bound of Appendix~\ref{sec:bounds}.
Next to a connected sum, in grey, are the unknotting numbers of its two summands, whose sum is the value that additivity would give.}
\label{tab:symbionts-unknotting-values}
\end{table}

\appendix

\section{One episode in the band environment}
\label{app:episode}

The pipeline of a single episode in the environment of Section~\ref{sec:env} is as follows.

\begin{lstlisting}[mathescape=true, caption={The pipeline of an episode.}, label={lst:episode}]
A new link is extracted from the dataset: the episode starts
The observation for the link is created
We then loop until one of the following holds:
  - the diagram is terminal                            --> episode won
  - an illegal state is reached                        --> episode lost
  - the step budget is reached                         --> episode truncated
Loop:
  - all legal actions are enumerated (the illegal ones are masked out)
  - the agent samples a legal action from the policy $\pi$
    -- if the action extends the band (start / over / under / twist):
      $\ast$ only the band path data inside the observation space is updated, 
        while the diagram is untouched
    -- if the action ends the band (end0 / end1 / end2):
      $\ast$ the diagram is rebuilt with the attached band
      $\ast$ the new diagram is simplified via spherogram-nim 
      $\ast$ (for genus computations) the surface bookkeeping is updated
    -- if the action is unknot_birth:
      $\ast$ a new unlinked trivial component is added to the diagram
  - the reward is computed and the new observation is created

A won episode returns a certificate: an explicit sequence of moves that 
provides a verifiable upper bound on the invariant
\end{lstlisting}

\subsection{A worked example}
\label{subs:worked-example-episode}
We close this appendix by displaying one full episode, with all the observations printed.
The episode we have chosen is exactly the one presented in Example~\ref{ex:trefoil-slice}. 
It builds a genus-1 surface for the trefoil (Figures~\ref{fig:slice_example1} and~\ref{fig:slice_example2}). 
In the following discussion, recall that we denote a strand of a link by \texttt{C\{x\}S\{y\}}, where \texttt{x} is the label of the crossing and \texttt{y} identifies which of the four strands at that crossing it is: the zeroth is the incoming under-strand, and the others are numbered in increasing order, going counter-clockwise around the crossing.

\medskip

The episode consists of two bands.
The first one is of length 2.
It starts, slides under a strand, and ends.
The second one is of length 1; i.e., it starts and ends without passing over or under anything. Such bands are the shortest the environment allows.
As in Example~\ref{ex:trefoil-slice}, attaching the first band to the trefoil yields the Hopf link, and attaching the second to the Hopf link yields the final unknot.

The framework is \texttt{slice}, and the environment is one of the \texttt{slice} models selected in Table~\ref{tab:feature-selection} (the model called \texttt{Slice\_Seif\_840k} in the header of the transcript). 
The observation is the one described in Section~\ref{subs:RL-policy-features}, with the Seifert genus bound (\texttt{Seif}) appended, and the rewards are those of Section~\ref{subs:env-reward}.
Note that, in the following transcript, we omit the padding. A line such as \texttt{[+97 pad]} records the $97$ suppressed rows of $-1{,}000$s, and a dash marks a tensor with no live entries.
The value $-500$ inside \texttt{gauss\_code} closes the Gauss code of a link component, and the value $-10$ inside \texttt{band\_matrix} marks a crossing strand that exists but has not been visited by the band.
The observation is printed as the environment builds it. 
Before it reaches the policy network, these sentinels are re-encoded and every entry is scaled, as described in Sections~\ref{subs:env-obs} and~\ref{subs:RL-policy-features}.
The last seven entries of each block are the Seifert genus bound (\texttt{seifert\_genus\_bound}) and the entries that complete the description of the state: the twist of each band segment (\texttt{band\_twists}), the number of crossings of the current diagram after simplification and of the initial diagram, the CCL of each link component (\texttt{component\_ccl}), one row of \texttt{genera\_info} per CCL, recording the CCL, the number of bands attached to that surface component, the initial number of link components (including the unknots born), the current number of link components, and the number of original link components it contains, and the genus of the surface built so far (\texttt{current\_genus}).
Finally, an \texttt{end} action is displayed as two blocks. The \texttt{STEP} block shows the diagram just after the band has been attached and the PD code rebuilt, and the \texttt{SIMPLIFY} block that follows shows the same diagram once \texttt{spherogram-nim} has simplified it; this is the diagram on which the next step acts.
Since the rebuilt diagram before the simplification is never a state of the environment, its \texttt{STEP} block shows only the entries relative to the diagram and the band path.
In Listing~\ref{lst:worked-example}, we present the episode.

\begin{lstlisting}[basicstyle=\ttfamily\scriptsize, escapechar=@, caption={A complete episode on the trefoil, in the \texttt{slice} framework.}, label={lst:worked-example}]
================================================================================
framework = slice      link = 3_1      shape = (100, 16)
policy = Slice_Seif_840k
max_twists = 5   max_actions = 80   observation = base + SeifertGenusBound
================================================================================

RESET      link 0 = 3_1: 3 crossings, 1 component        genus so far = 0
  PD_code                    : (6,3,1,4) (4,1,5,2) (2,5,3,6)  [+97 pad]
  crossing_signs             : -1 -1 -1
  writhe                     : -3
  gauss_code                 : -1 2 -3 1 -2 3 -500
  num_link_components        : 1
  unlinked_unknot_components : 0
  num_connected_components   : 1
  band_path                  : -  (no band in progress)
  total_twists               : +0
  band_matrix                : (-10,-10,-10,-10) (-10,-10,-10,-10) (-10,-10,-10,-10)
  seifert_genus_bound        : 1
  band_twists                : -
  simplified_num_crossings   : 3
  init_num_crossings         : 3
  component_ccl              : 0
  genera_info                : (0,0,1,1,1)
  current_genus              : 0

--------------------------------------------------------------------------------
STEP 1   action = C2S3_start     (13 of 2403 legal)   reward = -0.0100
  -> 3 crossings, 1 component        genus so far = 0
  PD_code                    : (6,3,1,4) (4,1,5,2) (2,5,3,6)  [+97 pad]
  crossing_signs             : -1 -1 -1
  writhe                     : -3
  gauss_code                 : -1 2 -3 1 -2 3 -500
  num_link_components        : 1
  unlinked_unknot_components : 0
  num_connected_components   : 1
  band_path                  : (2,3,0)
  total_twists               : +0
  band_matrix                : (-10,-10,-10,-10) (-10,-10,-10,0) (-10,-10,-10,-10)
  seifert_genus_bound        : 1
  band_twists                : 0
  simplified_num_crossings   : 3
  init_num_crossings         : 3
  component_ccl              : 0
  genera_info                : (0,0,1,1,1)
  current_genus              : 0

--------------------------------------------------------------------------------
STEP 2   action = C2S2_under     (6 of 2403 legal)   reward = -0.0100
  -> 3 crossings, 1 component        genus so far = 0
  PD_code                    : (6,3,1,4) (4,1,5,2) (2,5,3,6)  [+97 pad]
  crossing_signs             : -1 -1 -1
  writhe                     : -3
  gauss_code                 : -1 2 -3 1 -2 3 -500
  num_link_components        : 1
  unlinked_unknot_components : 0
  num_connected_components   : 1
  band_path                  : (2,3,0) (2,2,-1)
  total_twists               : +0
  band_matrix                : (-10,-10,-10,-10) (-10,-10,-1,0) (-10,-10,-10,-10)
  seifert_genus_bound        : 1
  band_twists                : 0 0
  simplified_num_crossings   : 3
  init_num_crossings         : 3
  component_ccl              : 0
  genera_info                : (0,0,1,1,1)
  current_genus              : 0

--------------------------------------------------------------------------------
STEP 3   action = C2S1_end0      (8 of 2403 legal)   reward = +0.1879
  -> band attached, diagram rebuilt   (not yet simplified)
  PD_code                    : (10,5,3,6) (6,2,7,1) (4,9,5,10) (1,7,2,8) (3,9,4,8)  [+95 pad]
  crossing_signs             : -1 -1 -1 -1 +1
  writhe                     : -3
  gauss_code                 : 2 -4 -500 -1 -5 -3 1 -2 4 5 3 -500
  num_link_components        : 2
  unlinked_unknot_components : 0
  num_connected_components   : 1
  band_path                  : (2,3,0) (2,2,-1) (2,1,0)
  total_twists               : +0
  band_matrix                : (-10,-10,-10,-10) (-10,0,-1,0) (-10,-10,-10,-10)

--------------------------------------------------------------------------------
SIMPLIFY   the diagram rebuilt by step 3   (spherogram-nim, global mode)
  -> 2 crossings, 2 components        genus so far = 0
  PD_code                    : (1,4,2,3) (3,2,4,1)  [+98 pad]
  crossing_signs             : -1 -1
  writhe                     : -2
  gauss_code                 : 2 -1 -500 1 -2 -500
  num_link_components        : 2
  unlinked_unknot_components : 0
  num_connected_components   : 1
  band_path                  : -  (no band in progress)
  total_twists               : +0
  band_matrix                : (-10,-10,-10,-10) (-10,-10,-10,-10)
  seifert_genus_bound        : 0
  band_twists                : -
  simplified_num_crossings   : 2
  init_num_crossings         : 3
  component_ccl              : 0 0
  genera_info                : (0,1,1,2,1)
  current_genus              : 0

--------------------------------------------------------------------------------
STEP 4   action = C2S0_start     (9 of 2403 legal)   reward = -0.0100
  -> 2 crossings, 2 components        genus so far = 0
  PD_code                    : (1,4,2,3) (3,2,4,1)  [+98 pad]
  crossing_signs             : -1 -1
  writhe                     : -2
  gauss_code                 : 2 -1 -500 1 -2 -500
  num_link_components        : 2
  unlinked_unknot_components : 0
  num_connected_components   : 1
  band_path                  : (2,0,0)
  total_twists               : +0
  band_matrix                : (-10,-10,-10,-10) (0,-10,-10,-10)
  seifert_genus_bound        : 0
  band_twists                : 0
  simplified_num_crossings   : 2
  init_num_crossings         : 3
  component_ccl              : 0 0
  genera_info                : (0,1,1,2,1)
  current_genus              : 0

--------------------------------------------------------------------------------
STEP 5   action = C2S3_end0      (5 of 2403 legal)   reward = +32.4758
  -> band attached, diagram rebuilt   (not yet simplified)
  PD_code                    : (1,4,2,1) (3,2,4,3)  [+98 pad]
  crossing_signs             : -1 -1
  writhe                     : -2
  gauss_code                 : 1 -1 2 -2 -500
  num_link_components        : 1
  unlinked_unknot_components : 0
  num_connected_components   : 1
  band_path                  : (2,0,0) (2,3,0)
  total_twists               : +0
  band_matrix                : (-10,-10,-10,-10) (0,-10,-10,0)

--------------------------------------------------------------------------------
SIMPLIFY   the diagram rebuilt by step 5   (spherogram-nim, global mode)
  -> 0 crossings, 1 unlinked unknot        genus so far = 1
  PD_code                    : (no crossings)  [+100 pad]
  crossing_signs             : -
  writhe                     : +0
  gauss_code                 : -
  num_link_components        : 0
  unlinked_unknot_components : 1
  num_connected_components   : 0
  band_path                  : -  (no band in progress)
  total_twists               : +0
  band_matrix                : -
  seifert_genus_bound        : 0
  band_twists                : -
  simplified_num_crossings   : 0
  init_num_crossings         : 3
  component_ccl              : -
  genera_info                : (0,2,1,1,1)
  current_genus              : 1

================================================================================
EPISODE WON  reached the unlink with 2 band moves,
             resulting in a slice surface of genus 1 with 1 unlinked unknot component
return = sum of the rewards = +32.6337   (gamma = 1)
================================================================================
\end{lstlisting}

\paragraph{Reading the transcript.}
The first band starts at the crossing strand \texttt{C2S3}, slides under the strand \texttt{C2S2} (a Reidemeister~II move, recorded as the entry $(2,2,-1)$ of \texttt{band\_path} and as the $-1$ in \texttt{band\_matrix}), and is attached at \texttt{C2S1} with no final twist.
Attaching the band rebuilds the diagram with $5$ crossings (step~3). 
Simplification then cuts this down to a $2$-crossing, $2$-component diagram of the Hopf link $L_{t_g}$ of Figure~\ref{subfig:slice2-g}.
The second band is the shortest one possible. It begins at \texttt{C2S0} and ends at \texttt{C2S3}, two strands of the same crossing that share a face.
The band rebuilds the $2$-crossing diagram \texttt{(1,4,2,1) (3,2,4,3)} of step~5, in which each crossing repeats one of its labels and is therefore a Reidemeister~I loop.
Undoing both loops is precisely what the last simplification does, and what is left is the unknot, which the environment reports as a diagram with $0$ crossings and \texttt{unlinked\_unknot\_components} equal to $1$.

The line \texttt{genus so far} shows the genus of the surface built up to that point. It is still $0$ after the first band (the surface is a pair of pants, with $\chi = -1$ and $3$ boundary components) and becomes $1$ after the second one since $2g = 2 + \#\{\text{bands}\} - \#\{\text{boundary components}\} = 2 + 2 - (1+1)$.
The \texttt{genera\_info} row records the same computation for the single surface component: $(0,1,1,2,1)$ after the first band, which is attached to it and raises its boundary components from $1$ to $2$, and $(0,2,1,1,1)$ at the end.
The episode therefore returns the certificate $\gen{\mathtt{3_1}} \leq 1$, which we know to be sharp.

Note that no \texttt{unknot\_birth} was played, even though the \texttt{slice} framework offers it (which is why the action space has size $2403 = 24 \cdot 100 + 2 + 1$ rather than $2402$).
This is expected since, for the trefoil, the ribbon genus already equals the slice genus, and adding a maximum to the surface would only make the certificate longer.

In Step~5, the reward is computed as
\begin{align*}
    R(s,a,s') & = \sigma\bigl(R_{\text{win}} - W (i(s') - i(s)) + \Phi(s') - \Phi(s)\bigr)\\
              & = 0.01\bigl(3600 - 400\cdot(1-0) + 2\kappa\bigr) = 32.4758,
\end{align*}
using $\kappa = \frac{65}{1 + \sqrt{3}}$: the reward for winning, minus the charge for the unit of genus that the band adds, plus the shaping term for the two crossings that it removes, all scaled by $0.01$; see Section~\ref{subs:env-reward}.
Since $\gamma = 1$, the return of the episode is the plain sum of its rewards, $32.6337$.

A further feature of the environment is visible in the transcript: the effect of the action masking of Section~\ref{subs:env-mdp}.
Given the small size of the trefoil, only $13$, $6$, $8$, $9$, and $5$ of the $2403$ actions are legal at the five steps, respectively, so the agent chooses from a set smaller by more than two orders of magnitude.

\section{Bounds}
\label{sec:bounds}

The environment we developed, presented in Section~\ref{sec:env}, produces \emph{upper} bounds on the invariants we are interested in, in the form of an explicit sequence of moves.
What turns an upper bound into an exact value is a matching lower bound, computed independently.
This appendix collects the bounds we evaluate before computing explicit band paths using some agent.
We assume the invariants themselves ($\sigma$, $\tau$, $s$, $\genSeif$) are known to the reader, and we provide references for each bound rather than restating the definitions.

\subsection{Relations between the three invariants}

The following proposition is well known and connects the unknotting number and the slice genus; see e.g.\ \cite[Section 1.1]{applebaum2025unknottingrl}.

\begin{proposition}
\label{prop:g4-lower-bound-unknotting}
For any link $L$, we have $\gen{L} \leq \unk{L}$.
\end{proposition}

Therefore, every lower bound on the smooth slice genus automatically gives a lower bound on the unknotting number, and we will not recompute those bounds for both invariants.
Note that, on the other hand, the strong slice genus does \emph{not} satisfy $\sgen{L} \leq \unk{L}$.

\begin{remark}[Orientation dependence]
\label{rmk:orientation-dependence-g4}
The weak slice genus depends on the choice of orientation of $L$, since the bounding surface $F$ must be oriented compatibly with $L$, while the strong slice genus is independent of the orientation.
Since $\unk{L}$ is also orientation-independent (Remark~\ref{rmk:u-orientation-independent}), the strongest lower bounds from the slice genus are
\[
\sgen{L} \geq \max_{\vec{o}} \gen{L_{\vec{o}}}, \qquad \unk{L} \geq \max_{\vec{o}} \gen{L_{\vec{o}}},
\]
where $\vec{o} \in \{ +1, -1 \}^{\ell}$ is a choice of orientation.
\end{remark}

Let us present a lower bound for the unknotting number of links that does not factor through the slice genus.
The following inequality is not stated explicitly by Collari~\cite{collari2021slicetoruspositive}, but it appears in the proof of \cite[Lemma~2.8]{collari2021slicetoruspositive}.
The statement of \cite[Lemma~2.8]{collari2021slicetoruspositive} comes with the additional hypothesis that the strong splitting number equals the absolute linking number; however, that is needed to produce an equality. 

\begin{proposition}[From the proof of {\cite[Lemma 2.8]{collari2021slicetoruspositive}}]
    \label{prop:lower-bound-unknotting-link}
    Let $L = K_1 \cup \dots \cup K_{\ell}$ be a link.
    Then, we have
    \[
        \abslk{L} + \sum_{i=1}^{\ell} \unk{K_i} \leq \unk{L},
    \]
    where $\abslk{L}$ denotes the absolute linking number of $L$, which is the sum of the absolute values of the linking numbers between the components of $L$.
\end{proposition}

\begin{remark}
The inequality of Proposition \ref{prop:lower-bound-unknotting-link} is an equality whenever $L$ is \emph{simply-linked}.
See Theorem \ref{thm:unknotting-simply-linked-links} for details.
\end{remark}

Lastly, let us state the following trivial lower bound on the strong slice genus.

\begin{lemma}[Lower bound on the strong slice genus]
    \label{lem:lower-bound-strong-genus-link-components}
    Given a link $L = K_1 \cup \dots \cup K_{\ell}$, we have
    \[
    \sum_{i = 1}^{\ell} \gen{K_i} \leq \sgen{L}.
    \]
\end{lemma}
Indeed, given a strong slice surface $F$ of minimal genus for $L$, we get that $F = \sqcup_{i=1}^{\ell} F_i$ with $\partial F_i = K_i$ and $g_4(K_i) \le g(F_i)$.
Using existing databases and possibly new computations of the slice genera of $K_1,\dots,K_{\ell}$, we might derive a non-trivial lower bound from this inequality.

\subsection{Upper bounds}
\label{subs:upper-bounds}

We have the following elementary bounds on the smooth slice genus.

\begin{proposition}
\label{prop:genera-relations}
For any link $L$,
\[
\gentop{L} \leq \gen{L} \leq \genSeif{L}.
\]
\end{proposition}

Indeed, any smoothly embedded surface is also locally flat; hence, any smooth slice surface for $L$ is also a topological slice surface for $L$. Furthermore, we can push the interior of any Seifert surface for $L$ in $\S^3$ into $B^4$, resulting in a properly embedded surface bounding $L$ in the $4$-ball.

We can compute $\genSeif{K}$ via knot Floer homology, which, for knots, was implemented by Ozsv\'ath and Szab\'o~\cite{culler2024snappy}. This gives us a practical slice genus upper bound for knots.
While we do not have an implementation of knot Floer homology for links, we can compute an upper bound on the Seifert genus of links, which results in an upper bound on the weak slice genus.

\begin{theorem}[Upper bound from Seifert surfaces, {\cite[Proposition 4.1]{cavallo2015links}}]
\label{thm:diagramUpperBoundLinks}
Given a link $L$ with $\ell$ components and a diagram $D$, Seifert's algorithm produces a surface $\Sigma$ with $\chi(\Sigma) = O(D) - c(D)$, where $O(D)$ is the number of Seifert circles and $c(D)$ is the number of crossings.
If the Seifert surface has $r(D)$ components, then
\begin{equation}
\label{eq:seifert-upper}
\gen{L} \leq \genSeif{L} \leq g(\Sigma) = \frac{2r(D) - \ell - O(D) + c(D)}{2}.
\end{equation}
\end{theorem}

We have the following simple upper bound on the unknotting number given by the crossing number.

\begin{proposition}[Trivial upper bound]
\label{prop:trivial-upper-bound}
For any link $L$,
\[
\unk{L} \leq \left\lfloor \frac{\mathrm{c}(L)}{2} \right\rfloor.
\]
\end{proposition}

\begin{proof}[Proof from {\cite[Discussion after Theorem 1.1]{akimoto2020unknottinggraph}}]
It is a known fact that any diagram $D$ of $L$ with $c$ crossings can be turned into a diagram of the unlink by changing some of its crossings.
Let $\mathcal{C}(D)$ be the set of crossings of $D$ and $A \subset \mathcal{C}(D)$ be the subset of crossings that yields a diagram $U$ of the unlink if changed.
Let now $S$ be the diagram we obtain from $D$ if we change all the crossings in $B = \mathcal{C}(D) \smallsetminus A$. This is the mirror of $U$, which is again the unlink.
Therefore, we get the unlink with at most $\min \{ \lvert A \rvert, \lvert B \rvert \} \leq \frac{c(D)}{2}$ crossing changes.
\end{proof}

No upper bound on the strong slice genus is available to us. A Seifert surface is not a strong slice surface, so every finite upper bound on $\sgen{L}$ comes from an explicitly exhibited surface, which is precisely what the agents of Section~\ref{sec:agents} produce.

\subsection{Lower bounds, by technique}
\label{subs:lower-bounds}

\subsubsection{Signature-based}
\label{subsubs:signature}

The signature is a classical and computationally cheap link invariant. Given any Seifert surface $S$ for a link $L$, obtained for instance by Seifert's algorithm applied to a diagram of $L$, it is the signature of $A + A^{T}$, where $A$ is a Seifert matrix associated with $S$.
It was introduced by Trotter~\cite{trotter1962homology} for knots and extended to links by Murasugi~\cite[Theorem 3.1]{murasugi1965numerical}, who also showed that it bounds the (weak) slice genus of links.

\begin{theorem}[Signature bound {\cite[Theorem 9.1]{murasugi1965numerical}}]
\label{thm:murasugi-signature-bound}
For any link $L$ with $\ell$ components,
\[
\frac{ \lvert \sigma(L) \rvert - \ell + 1}{2} \leq \gen{L},
\]
where $\sigma(L)$ is the signature of $L$.
For a knot $K$, this reads $\tfrac{1}{2}\abs{\sigma(K)} \leq \gen{K}$.
\end{theorem}

Sharper bounds come from ``signature-like'' quantities, the main one being the Levine--Tristram signature $\sigma_L(\omega)$ and nullity $\eta_L(\omega)$ at $\omega \in \S^1$, introduced independently by Levine~\cite{levine1969cobordism} and Tristram~\cite{tristram1969cobordism}. We refer to \cite[Definition 1]{conway2019signature} for the definition and for a complete survey.
That they give slice genus bounds was originally proved by Tristram for $\omega = -1$ or $\omega = \exp(\nicefrac{2q \pi i}{2q+1})$ with $2q+1$ an odd prime, and extended by Kauffman~\cite{kauffman1978signature} to prime power order roots of unity at the cost of requiring the surface to be connected.
Note that Tristram and Kauffman's nullity $\eta_\omega(L)$ equals our $\eta_L(\omega) + 1$, as defined in \cite[Definition~2$\cdot$14]{tristram1969cobordism} and \cite[above Theorem~4.1]{kauffman1978signature}.
The case $\omega = -1$ recovers the signature.

\begin{theorem}[Murasugi--Tristram inequality {\cite[Theorem~2$\cdot$27]{tristram1969cobordism}\cite[Theorem~4.1]{kauffman1978signature}}]
\label{thm:murasugi-tristram}
Let $L$ be an $\ell$-component link and $\omega \in \S^1 \subset \C$ be a primitive $a$-th root of unity, where $a = p^s$ with $p$ prime.
Then, we have
\[
\frac{ \lvert \sigma_L(\omega) \rvert + \eta_L(\omega) - (\ell-1)}{2} \leq \gen{L}.
\footnote{%
  \label{foot:kauf-misprint}
  Kauffman's Theorem 4.1 is misprinted as $ \lvert \sigma_\omega(L) \rvert \le \mathbf{g(F)} + \mu(L) - \eta_\omega(L)$, but the final display of his own proof gives $ \lvert \sigma_\omega(L) \rvert \le \mathbf{2g(F)} + \mu(L) - \eta_\omega(L)$; hence, the factor $2$ is missing from the statement and we use the corrected form.}
\]
\end{theorem}

\begin{remark}[Disconnected surfaces and the strong slice genus]
\label{rmk:disconnected-mt}
Tristram's Theorem~2$\cdot$27 holds for $\omega = -1$ and for $\omega = \exp(\nicefrac{2q \pi i}{2q+1})$ with $2q+1$ an odd prime \cite[Definitions~2$\cdot$3, 2$\cdot$14]{tristram1969cobordism} and is stated for an arbitrary, possibly disconnected, spanning surface $F \subset B^{4}$. If $F$ has $m$ connected components, it reads
\[
  \lvert \sigma_L(\omega) \rvert + \lvert (\eta_L(\omega)+ 1) - m \rvert \leq 2g(F) + \ell - m.
\]
Theorem~\ref{thm:murasugi-tristram} is the case $m = 1$, which Kauffman extended to all roots of unity of prime power order.
In the strong slice case, the surface $F = F_1 \sqcup \dots \sqcup F_{\ell}$ is a disjoint union of $\ell$ surfaces, one per component of $L$, so $m = \ell$ and the inequality becomes a bound on the strong slice genus:
\[
    \frac{\lvert \sigma_L(\omega) \rvert + \lvert (\eta_L(\omega)+ 1) - \ell \rvert}{2} \leq \sgen{L}.
\]
In particular, if $L$ is strongly slice, then $\sigma_L(\omega) = 0$ and $\eta_L(\omega) = \ell - 1$.
See \cite[Theorem~3]{conway2019signature} for a version valid for locally flat surfaces and for all $\omega$ outside the set of \emph{Knotennullstellen}, and \cite[Theorem~2.1]{florensgilmer2003links} for the connected case.
\end{remark}

\paragraph{Lower bound specific to \texorpdfstring{$\unk{L}$}{u(L)}.}
The Levine--Tristram signature for links provides a lower bound directly on $\unk{L}$ without factoring through $\gen{L}$.
It rests on a lemma of Nagel and Owens \cite[Lemma 2.1]{nagelowens2015unlinking}, which describes how the signature and nullity behave when a single crossing is changed, extending the results of Cochran and Lickorish \cite[Proposition 2.1]{cochranlickorish1986unknotting} and Stoimenow \cite[Theorem 5.1]{stoimenow2004polynomial} for knots to the multi-component case.
Conway \cite[Proof of Proposition 3]{conway2019signature} noted that this lemma and its proof hold not only for $\omega = -1$, as Nagel and Owens stated, but for any $\omega \in \S^1_{\ast} = \S^1 \smallsetminus \{1\}$, and deduced the following bound.

\begin{theorem}[{\cite[Theorem 1]{conway2019signature}}]
\label{thm:levine-tristram-unknotting-links}
Let $L$ be an oriented $\ell$-component link and let $\omega \in \S^1_{\ast}$. Then
\[
  \frac{\lvert \sigma_L(\omega) \rvert + \lvert \eta_L(\omega) - \ell + 1 \rvert}{2} \leq \unk{L}.
  \footnote{%
  The statement as printed in \cite[Theorem 1]{conway2019signature} reads $\abs{\sigma_L(\omega)} + \abs{\eta_L(\omega) + \ell - 1} \le 2\unk{L}$; this is a misprint.
  Indeed, for the $\ell$-component unlink $U_\ell$ one has $\unk{U_\ell} = 0$, $\sigma_{U_\ell}(\omega) = 0$, and $\eta_{U_\ell}(\omega) = \ell - 1$, so the printed inequality would give $|2\ell - 2| \le 0$, which is false for $\ell \ge 2$.
  The term $-\ell+1$ is moreover the one consistent with the Murasugi--Tristram inequality \cite[Theorem~3]{conway2019signature}, whose nullity term is $\abs{\eta_L(\omega) - m + 1}$ with $m$ the number of components of the bounding surface.}
\]
\end{theorem}

\subsubsection{Knot Floer homology}
\label{subsubs:floer-homology}

Knot Floer homology was first introduced independently by Ozsv\'ath and Szab\'o~\cite{ozsvath2004knotfloer} and Rasmussen~\cite{rasmussen2003floerhomologyknotcomplements}.
They defined a concordance invariant $\tau$; see~\cite{ozsvath2003fourball}.

\begin{theorem}[{\cite[Corollary 1.3]{ozsvath2003fourball}}]
\label{thm:tau-bound}
The $\tau$-invariant satisfies
\[
\abs{\tau(K)} \leq \gen{K}.
\]
\end{theorem}

Ozsv{\'a}th and Szab{\'o}~\cite[Definition~9.1]{ozsvath2011rationalsurgery} introduced a refinement $\nu$ of $\tau$, which satisfies $\nu(K) \in \{\tau(K), \tau(K)+1\}$ and also bounds the slice genus from below.

\begin{proposition}[{\cite[Proposition 2.3 and Proposition 2.4]{hom2016nuplus}}]
\label{prop:nu-bound}
For a knot $K$,
\[
  \nu(K) \leq \gen{K}.
\]
\end{proposition}

Both $\tau$ and $\nu$ can be computed for knots via the \texttt{knot\_floer\_homology} module of \cite{culler2024snappy}.

\paragraph{An extension of \texorpdfstring{$\tau$}{tau} to links.}
Cavallo~\cite{cavallo2018concordancetau} introduced a version of $\tau$ for links via link grid homology, and is a strong concordance invariant that yields a lower bound on the weak slice genus of a link.
However, as of now, there is no publicly available implementation that computes $\tau$ for links.

\begin{theorem}[{\cite[Proposition 1.4]{cavallo2018concordancetau}}]
\label{thm:tau-links}
For an $\ell$-component link $L$,
\[
\abs{\tau(L)}  +1 - \ell \leq \gen{L}.
\]
\end{theorem}

\subsubsection{Khovanov homology}
\label{subsubs:khovanov-hom}

Khovanov~\cite{khovanov2000categorification} categorified the Jones polynomial and introduced Khovanov homology.
Building on the work of Lee~\cite{lee2005endomorphism}, Rasmussen~\cite[Definition 3.4]{rasmussen2010slice} introduced a new knot invariant $s$, later called the \emph{Rasmussen $s$-invariant}, that provides a lower bound on the slice genus.
This invariant was subsequently extended to multi-component links by Beliakova and Wehrli~\cite[Section 7.1]{beliakova2008categorification}.

\begin{theorem}[Rasmussen's $s$-invariant lower bound {\cite[Theorem 1]{rasmussen2010slice}\cite[Equation (4.1)]{cavallo2015links}}]
\label{thm:rasmussen-lower-bound}
For an $\ell$-component link $L$, we have
\[
\frac{\abs{s(L)} + 1 - \ell}{2} \leq \gen{L}.
\]
\end{theorem}

One of the properties proved for the link version is of particular interest to us.\footnote{Beliakova and Wehrli~\cite[equation~(7.1)]{beliakova2008categorification} only assume that each component of $F$ intersects $L_1$, but the cobordism from a 2-component to a 1-component unlink consisting of a disc and a product cylinder provides a counterexample. However, their proof works if we assume that each component of $F$ also intersects $L_2$.}

\begin{lemma}[{\cite[equation~(7.1)]{beliakova2008categorification}}]
    \label{lem:cobordims_bound_s}
    Let $F$ be a smooth oriented cobordism from $L_1$ to $L_2$ such that every connected component of $F$ intersects both $L_1$ and $L_2$.
    Then
    \[
        \abs{s(L_1) - s(L_2)} \leq -\chi(F).
    \]
\end{lemma}

Applying this cobordism estimate to the case of a strong cobordism, we can produce a lower bound for the strong slice genus of a link.

\begin{corollary}[Strong slice genus bound]
\label{cor:strong-slice-rasmussen}
Let $L$ be an $\ell$-component, algebraically split link.
Then
\[
    \frac{\abs{s(L) + \ell - 1}}{2} \leq \sgen{L}.
\]
\end{corollary}

\begin{proof}
Let $F = \bigsqcup_{i=1}^{\ell} F_i$ be a minimal genus strong cobordism between $L$ and $U$, the $\ell$-component unlink.
Then $\sum_i g(F_i) = \sgen{L}$.
Assume that $\partial F_i = -K_i \sqcup U_i$, where $U_i$ is the $i$-th component of $U$.
As $F$ has $\ell$ connected components and $2\ell$ boundary circles, by equation~\eqref{eq:genus-bound-surface}, we get that its Euler characteristic is
\[
\chi(F) = \sum_{i=1}^{\ell} \chi(F_i) = \sum_{i=1}^{\ell} (2 -2g(F_i) - \abs{\partial F_i}) =  2\ell - 2\sgen{L} - 2\ell = -2\sgen{L}.
\]
Since every component of $F$ intersects $L$ and $U$, Lemma~\ref{lem:cobordims_bound_s} applies.
As $s(U) = 1 - \ell$, we get
\[
    \abs{s(L) + \ell - 1} \leq - \chi(F) = 2 \sgen{L},
\]
which concludes the proof.
\end{proof}

\paragraph{On the computability of the Rasmussen $s$-invariant.}
Several pieces of software were created to compute the $s$-invariant for knots.
In our codebase, we integrated the implementation by Sch{\"u}tz~\cite{schuetz2025knotjob}, which, to our knowledge, is the only one able to compute the $s$-invariant for links.
Details on the algorithm used for knots can be found in Sch{\"u}tz~\cite{schuetz2021algorithm, schuetz2021errataalgorithm}, and for links in \cite{schuetz2025links}.

\subsubsection{Slice-torus link invariants}
\label{subsubs:slice-torus-invariants}

Livingston~\cite{livingston2004computations} defined a class of knot invariants that Lewark~\cite{lewark2014spectral} later called \emph{slice-torus invariants}, and he observed that they could be generalised from taking values in even integers to reals.
They were extended to links by Cavallo and Collari~\cite{cavallocollari2020links}, and they proved several results concerning weak and strong sliceness.
What follows is a short survey of these results and the bounds they yield on the weak and strong slice genera of a link.

\begin{example}[Known slice-torus link invariants {\cite[Examples 2.3--2.4]{cavallocollari2020links}}]
\label{ex:slice-torus-invariants}
The following are slice-torus link invariants for an $\ell$-component link $L$:
\begin{enumerate}
    \item the Ozsv\'ath--Szab\'o $\tau$-invariant for links (Section~\ref{subsubs:floer-homology}): $\nu_{\tau}(L) := \tau(L)$;
    \item the Rasmussen $s$-invariant for links (Section~\ref{subsubs:khovanov-hom}): $\nu_s(L) := \tfrac{s(L) + \ell - 1}{2}$.
\end{enumerate}
\end{example}

The following result generalises what we already stated for $s$ and $\tau$ in this framework.

\begin{theorem}[Slice-torus bound on the weak slice genus {\cite[Proposition~1.2]{cavallocollari2020links}}]
\label{thm:slice-torus-bound}
If $\nu$ is a slice-torus link invariant and $L$ is an $\ell$-component link, then:
\[
-\gen{L} \leq \nu(L) \leq \gen{L} + \ell - 1.
\]
Furthermore, if $L$ is \emph{strongly slice}, then $\nu(L) = 0$.
\end{theorem}

We state the bound on the slice genus that we obtain from the above theorem.
Since the slice genus is invariant under mirroring and orientation reversal, we can also consider the values we obtain from the slice-torus invariant applied to $-\mir{L}$.

\begin{corollary}[Lower bound on the weak slice genus]
\label{cor:lower-bound-nu}
For any slice-torus link invariant $\nu$ and $\ell$-component link $L$,
\[
\gen{L} \geq \max\left\{ -\nu(L), \,\, \nu(L) - \ell + 1, \,\, -\nu(-\mir{L}), \,\, \nu(-\mir{L}) - \ell + 1 \right\}.
\]
\end{corollary}

Theorem~\ref{thm:slice-torus-bound} follows from the following cobordism bound.

\begin{theorem}[Cobordism bound {\cite[Proposition~2.11]{cavallocollari2020links}}]
\label{thm:cobordism-bound}
Let $\nu$ be a slice-torus link invariant. If $\Sigma \subset \S^3 \times [0,1]$ is a cobordism from $L_0$ to $L_1$ with $n$ connected components, each intersecting both $L_0$ and $L_1$, then
\[
\nu(L_1) - g(\Sigma) - \ell_1 + n \leq \nu(L_0) \leq \nu(L_1) + g(\Sigma) + \ell_0 - n,
\]
where $\ell_i$ denotes the number of components of $L_i$.
\end{theorem}

Indeed, by taking a slice surface for a link $L$ (i.e., a cobordism between $L$ and the unknot), one obtains the inequalities of Theorem~\ref{thm:slice-torus-bound}.
Moreover, we can get a better bound on the strong slice genus by taking a minimal genus strong slice surface for $L$; i.e., a cobordism between $L$ and the $\ell$-component unlink $U_{\ell}$, and recalling that $g(F) = \sgen{L}$. Various terms cancel, and we obtain the following corollary.

\begin{corollary}
\label{cor:strong-cobordism-bound}
Let $L$ be an $\ell$-component algebraically split link.
Then, for any slice-torus link invariant $\nu$,
\[
\abs{\nu(L)} \leq \sgen{L}.
\]
Since the strong slice genus is invariant under mirroring and orientation reversal, we obtain
\[
\max\{\abs{\nu(L)}, \abs{\nu(-\mir{L})}\} \leq \sgen{L}.
\]
\end{corollary}

By Lemma~\ref{lem:lower-bound-strong-genus-link-components}, if we denote by $K_1, \dots, K_{\ell}$ the components of the link $L$, then
\begin{equation}\label{eqn:strong-genus-bound}
\max\left\{\abs{\nu(L)}, \abs{\nu(-\mir{L})}, \sum_{i=1}^{\ell} \abs{\nu(K_i)}\right\} \leq \sgen{L}.
\end{equation}
Indeed, $\gen{K} = \sgen{K}$ for a knot $K$. By Corollary~\ref{cor:strong-cobordism-bound}, we have $\abs{\nu(K)} \leq \gen{K}$. Together with $\sum_{i = 1}^{\ell} \gen{K_i} \leq \sgen{L}$, we obtain inequality~\eqref{eqn:strong-genus-bound}.

\paragraph{Bounds on the unknotting number.}
By studying the behaviour of slice-torus invariants under crossing changes, we can also obtain bounds on the unknotting number that do not factor through the slice genus.

The first result gives a lower bound on the unknotting number of a link by studying a splitting sequence.
Recall that the splitting number of a link, denoted $\widetilde{sp}$ by Cavallo and Collari\footnote{
    Note that in subsequent work of Collari~\cite{collari2021slicetoruspositive}, the notation is swapped and $\widetilde{sp}$ denotes the \emph{strong splitting number} of a link; that is, only mixed-crossing changes are allowed.
    We use the notation from \cite{cavallocollari2020links}.
    },
is the minimum number of crossing changes needed to change an $\ell$-component link into a disjoint union of $\ell$ knots.
Clearly, $\widetilde{sp}(L) \leq \unk{L}$ since an unknotting sequence is also a splitting sequence.

\begin{theorem}[{\cite[Theorem 1.5]{cavallocollari2020links}}]
\label{thm:splitting-number-bound}
    Suppose that $\nu$ is a slice-torus link invariant and $L$ is a link with components $K_1, \dots , K_{\ell}$.
    Then
    \[
    \abs{\nu(L) - \sum_{i=1}^{\ell} \nu(K_i)} \leq \widetilde{sp}(L) \leq \unk{L}.
    \]
\end{theorem}

The above theorem is based on the following result, which describes the effect of a crossing change on a slice-torus link invariant.

\begin{proposition}[Crossing changes and slice-torus link invariants {\cite[Proposition 2.10]{cavallocollari2020links}}]
\label{prop:crossing-change-slice-torus}
Let $D_{-}$ be obtained from $D_{+}$ by replacing a positive crossing with a negative one, and let $L_{\pm}$ be the corresponding links.
Then, for every slice-torus link invariant $\nu$,
\[
     \quad 0 \leq \nu\left(L_{+}\right) - \nu\left(L_{-}\right) \leq 1.
\]
\end{proposition}

Using Proposition \ref{prop:crossing-change-slice-torus}, we can obtain another bound on the unknotting number of a link.

\begin{corollary}[Bound on the unknotting number]
\label{cor:unlinking-signed-slice-torus}
Let $\nu$ be a slice-torus link invariant, and let $L$ be a link.
If $L$ admits an unknotting sequence consisting of $u_+$ changes of a positive crossing into a negative one and $u_-$ changes of a negative crossing into a positive one, then
\[
    \nu(L) \leq u_{+} \qquad\text{and}\qquad -\nu(L) \leq u_{-} .
\]
Combining the two inequalities, we get $\abs{\nu(L)} \leq \unk{L}$.
\end{corollary}

\begin{proof}
    Consider an unknotting sequence from $L$ with $u_{+}$ positive crossing changes, $u_{-}$ negative ones, and $u_{+} + u_{-} = N = \unk{L}$.
    We can describe the sequence from $L$ to the unlink $U_{\ell}$ as
    \begin{center}
        \begin{tikzcd}
        L = L^{(0)} \arrow[r, "u_1"] & L^{(1)} \arrow[r, "u_2"] & \cdots \arrow[r, "u_N"] & L^{(N)} = U_{\ell}
        \end{tikzcd}    
    \end{center}
    Let $\delta_j \coloneqq \nu(L^{(j)}) - \nu(L^{(j-1)})$.
    \begin{itemize}
        \item If $u_j$ turns a positive crossing into a negative one (i.e., is one of the $u_{+}$), then in the notation of Proposition~\ref{prop:crossing-change-slice-torus}, we have $L^{(j-1)} = L_{+}$ and $L^{(j)} = L_{-}$, so $0 \leq \nu(L^{(j-1)}) - \nu(L^{(j)}) \leq 1$, which gives $-1 \leq \delta_j \leq 0$.
        \item If $u_j$ turns a negative crossing into a positive one (i.e., is one of the $u_{-}$), then with $L^{(j-1)} = L_{-}$ and $L^{(j)} = L_{+}$, we obtain $0 \leq \nu(L^{(j)}) - \nu(L^{(j-1)}) \leq 1$; that is, $0 \leq \delta_j \leq 1$.
    \end{itemize}
    In particular, by summing all $N$ inequalities,
    \[
        -u_{+} \leq \sum_{j=1}^{N} \delta_j \leq u_{-}.
    \]
    Recall that $\nu(U_{\ell}) = 0$, so 
    \begin{align*}
        -\nu(L) & = \nu(U_{\ell}) - \nu(L) = \nu(L^{(N)}) - \nu(L^{(0)}) \\
        & = \nu(L^{(N)}) - \nu(L^{(N-1)}) + \nu(L^{(N-1)}) - \, \ldots \,  -\nu(L^{(1)})  + \nu(L^{(1)}) - \nu(L^{(0)}) \\
        & = \sum_{j=1}^{N} \delta_j.
    \end{align*}
    Hence,
    \[
        -u_{+} \leq -\nu(L) \leq u_{-},
    \]
    which yields the two stated inequalities.
    Finally, since $u_{+} \geq \max\{\nu(L), 0\}$ and $u_{-} \geq \max\{-\nu(L), 0\}$, we get $\unk{L} = u_+ + u_- \geq \abs{\nu(L)}$.
\end{proof}

\subsubsection{The Fox--Milnor condition}
\label{subsubs:polynomial}

The Alexander polynomial $\Delta_K(t)$ of a knot $K$ is a classical invariant, introduced by Alexander~\cite{alexander1928polynomial}.
It is a Laurent polynomial that is well-defined up to a factor $\pm t^m$ for some integer $m$.
Fox and Milnor proved~\cite{foxmilnor1966cobordism} that this polynomial yields an obstruction to sliceness.

\begin{theorem}[{\cite[Theorem 2]{foxmilnor1966cobordism}}]
    \label{thm:fox-milnor-condition}
    If $K$ is a slice knot, then its Alexander polynomial $\Delta_K(t)$ is of the form
    \[
        \Delta_K(t) = t^m \cdot p \left( t \right) p \left( \nicefrac{1}{t} \right),
    \]
    where $p(t)$ is a polynomial with integral coefficients and $m$ is an integer.
\end{theorem}

The determinant of a knot agrees with $\abs{ \Delta_K(-1) }$. Hence, we obtain the following corollary, which provides a simpler obstruction to sliceness.

\begin{corollary}
\label{cor:determinant-square}
    If a knot $K$ is slice, then its determinant is a perfect square.
\end{corollary}

\subsection{Special families}
\label{subs:special-families}

For certain special families of links, we can give stronger bounds.
What follows is a short survey on some of these, namely: positive and negative links, torus links, alternating links, and pseudo-thin links.

\subsubsection{Positive and negative links}
\label{sec:pos_links_slice}

In this section, we study positive and negative links, whose four-genera can be computed exactly and directly from a diagram.
Recall that an oriented link is \emph{positive} if it admits a diagram where all crossings are positive, and \emph{negative} if it admits a diagram where all crossings are negative.

The positive case was settled in stages.
Nakamura~\cite[Theorem~1.1]{nakamura2000positive} proved that for a non-split positive link, the slice genus, the Seifert genus, and the genus of the surface produced by Seifert's algorithm all agree, so $\gen{L}$ can be read off from any positive diagram.
For knots, the same result was later reproved by Rasmussen~\cite[Theorem~4]{rasmussen2010slice}, using a purely combinatorial argument in Khovanov homology, which ties both genera to the $s$-invariant. The resulting diagrammatic formula for $s$ was then extended to positive links by Cavallo~\cite[Proposition~3.3]{cavallo2015links}.
All of these are instances of a single statement about slice-torus link invariants, due to Cavallo and Collari, which is the form we use.

\begin{theorem}[Positive links and slice-torus invariants {\cite[Theorem~1.3]{cavallocollari2020links}}]
\label{thm:positive-links-genus}
Let $L$ be an $\ell$-component positive link with $\ell_s$ split components and positive diagram $D$ that has $c(D)$ crossings and $O(D)$ Seifert circles. Then, for any slice-torus link invariant $\nu$,
\begin{equation*}
    \label{eq:slice_torus_positive}
    \nu(L) = \genSeif{L} + \ell - \ell_s = \frac{c(D) - O(D) + \ell}{2}.
\end{equation*}
If $L$ is also non-split, then
\begin{equation}
    \label{eq:slice_torus_positive_genus}
    \nu(L) = \gen{L} + \ell - 1,
\end{equation}
and consequently $\gen{L} = \genSeif{L}$.
\end{theorem}

The negative case follows by mirroring.
If $D$ is a negative diagram of $L$, then the mirror image $\mir{D}$ is a positive diagram of $\mir{L}$ with the same numbers of crossings and Seifert circles.
Since the four-genus is invariant under mirroring, one obtains for a non-split negative diagram $D$ of an $\ell$-component link $L$ that
\[
\gen{L} = \genSeif{L} = \frac{c(D) - O(D) + 2 - \ell}{2}.
\]

\paragraph{Unknotting number of positive links.}
While no formula is known for arbitrary positive links, the following results help us to compute some unknotting numbers.
We say that a link is \emph{simply-linked} if it admits a diagram where the crossings between any two fixed components have the same sign.
Clearly, positive links are simply-linked.

\begin{theorem}[Unknotting number of simply-linked links {\cite[Proposition~4.1 and Theorem~1.9]{collari2021slicetoruspositive}}]
    \label{thm:unknotting-simply-linked-links}
    If $L = K_1 \cup \dots \cup K_{\ell}$ is a simply-linked link, then
    \[
    \abslk{L} = \unk{L} - \sum_{i=1}^{\ell} \unk{K_i},
    \]
where $\abslk{L}$ is the \emph{absolute linking number} of $L$; i.e., the sum of the absolute values of $\lk{K_i, K_j}$ for all $i < j$.
For positive links, we have $\abslk{L} = \lk{L}$.
\end{theorem}

Therefore, we can quickly compute the unknotting numbers of several positive links with up to 11 crossings by knowing the unknotting numbers of positive knots with up to 9 crossings and identifying the link components using topological invariants.

\subsubsection{Torus links}
\label{subsubs:torus}

Recall that the torus link $\torus{a,b}$ with $\gcd(a,b) = \ell$ is an $\ell$-component link lying on a standard torus in $\S^3$, wrapping $a$ times around the meridian and $b$ times around the longitude. When $\gcd(a,b) = 1$, $\torus{a,b}$ is a knot.

\begin{theorem}[Milnor conjecture]
\label{thm:milnor}
For $a$, $b \in \mathbb{Z}_+$ and $\gcd(a,b) = 1$,
\[
\gen{\torus{a,b}} = \genSeif{\torus{a,b}} = \unk{\torus{a,b}} = \frac{(a-1)(b-1)}{2}.
\]
\end{theorem}

This was originally proved using gauge theory by Kronheimer and Mrowka~\cite{kronheimermrowka1993gauge,kronheimermrowka1995embedded}.
Subsequently, Ozsv\'ath and Szab\'o~\cite[Corollary 1.7]{ozsvath2003fourball} reproved it using Heegaard Floer homology by showing that $\tau(\torus{a,b}) = (a-1)(b-1)/2$. Rasmussen~\cite[Corollary 1]{rasmussen2010slice} gave a purely combinatorial proof using the $s$-invariant in Khovanov homology.
The unknotting number, the four-ball genus, and the Seifert genus of torus links are also known.

\begin{theorem}[Unknotting number of positive torus links {\cite[Example 4.1]{kawamura2002unknotting}}]
\label{thm:torus-link-unknotting}
Let $\torus{a,b}$ be a torus link with $a$, $b \in \mathbb{Z}_+$ and all components oriented coherently.
Let $\ell = \gcd(a,b)$.
Then
    \[
    \unk{\torus{a,b}} = \frac{(a-1)(b-1) + \ell - 1}{2}.
    \]
\end{theorem}

\subsubsection{Alternating links}
\label{subsubs:alternating}

For alternating links, the Rasmussen $s$-invariant agrees with the signature up to an overall sign (depending on the sign convention for the signature), which is much faster to compute. On an alternating diagram, it can be read off from the chequerboard surface; see Ozsv\'ath and Szab\'o~\cite[Theorem~3.1]{ozsvath2003alternating}.

\begin{proposition}[{{\cite[Theorem~3]{rasmussen2010slice}} for knots, \cite[Corollary~3.2]{cavallo2015links}} for links]
\label{prop:alternating_link_s_sign}
If $L$ is a non-split alternating link, then $\abs{s(L)} = \abs{\sigma(L)}$.
\end{proposition}

Since $\tau(K) = -\sigma(K)/2$ for alternating knots \cite[Theorem~1.4]{ozsvath2003fourball}, the signature, $s$, and $\tau$ all give the same lower bound. Combining Proposition~\ref{prop:alternating_link_s_sign} with Theorem~\ref{thm:rasmussen-lower-bound}, we recover the classical lower bound
\begin{equation}
\label{eq:genus-alternating}
\gen{L} \geq \frac{\abs{\sigma(L)} + 1 - \ell}{2}
\end{equation}
for an alternating link $L$.
This also holds in the broader class of \emph{quasi-alternating} links \cite[Definition~3.1]{ozsvath2005branched}, by using
\[
\tau(L) = \frac{\ell - 1 - \sigma(L)}{2}, \quad \tau(\mir{L}) = \ell - 1 - \tau(L)
\]
due to Cavallo~\cite[Corollary~3.7]{cavallo2018concordancetau}.

On the unknotting side, McCoy~\cite[Theorem~1]{mccoy2017alternatingunknottingnumber} proved Kohn's conjecture, which states that an alternating knot satisfies $\unk{K} = 1$ if and only if it has an unknotting crossing in every alternating diagram.
More recently, the signature was shown to determine $\unk{L}$ for \emph{special} alternating links, which admit an alternating diagram where one of the planar spanning surfaces is a Seifert surface for the link.

\begin{theorem}[{\cite[Theorem~1 and Corollary~2]{mccoy2026specialalternatinglinks}}]
\label{thm:special-alternating}
Let $L$ be an oriented non-split special alternating link with $\ell$ components.
Then $\unk{L} = \frac{\abs{\sigma(L)}+\ell-1}{2}$ if and only if $L$ can be unlinked by that many crossing changes in any alternating or minimal diagram. In particular, if $K$ is a special alternating knot, then $\unk{K} = \abs{\sigma(K)}/2$ if and only if $K$ can be unknotted by $\abs{\sigma(K)}/2$ crossing changes in any alternating or minimal diagram.
\end{theorem}

\subsubsection{Pseudo-thin links}
\label{subsubs:pseudo-thin-links}

Pseudo-thin links form a class for which the Rasmussen $s$-invariant and related bounds can be computed more explicitly.
This class includes all alternating links (see Ozsv\'ath and Szab\'o~\cite[Lemma~3.2]{ozsvath2005branched}), quasi-alternating links (see Cavallo~\cite[Proposition~3.7]{cavallo2015links}), and positive and negative links (see Cavallo~\cite[after Proposition~3.7]{cavallo2015links}), allowing us to apply the following result to all these classes of links.

\begin{theorem}[Strong slice genus of pseudo-thin links {\cite[Theorem~5.4]{cavallo2015links}}]
\label{thm:strong-genus-pseudo-thin}
Let $L = K_1 \cup \cdots \cup K_{\ell}$ be an $\ell$-component non-split pseudo-thin link which is also algebraically split (i.e., $\lk{K_i, K_j} = 0$ for all $i \neq j$). Then
\[
\frac{|s(L)| + \ell - 1}{2} \leq \sgen{L}.
\]
\end{theorem}

Moreover, $|s(L)| = |\sigma(L)|$ for an alternating link $L$ by Proposition~\ref{prop:alternating_link_s_sign}. Hence, we can compute the signature instead of the Rasmussen $s$-invariant for them, which is computationally cheaper.

\subsection{Computational considerations}
\label{subs:bounds-computational-sublinks}

Some of the results presented above (Proposition~\ref{prop:lower-bound-unknotting-link}, Lemma~\ref{lem:lower-bound-strong-genus-link-components}, and Theorem~\ref{thm:unknotting-simply-linked-links}) require knowing the unknotting numbers or the slice genera of the individual link components.
This makes them somewhat impractical in a general setting.
Our implementation addresses this issue using two distinct strategies: one general and one specific to the dataset we investigated.

For the general case, we estimate the required invariants of the components using the other bounds presented in this section. 
Given a link $L = K_1 \cup \dots \cup K_{\ell}$, we first compute the bounds for each component $K_i$ and then use these estimates in the inequalities for links obtained from the results mentioned above.

For our specific dataset, which considers links with up to 11 crossings, we can often identify the exact link components by querying a comprehensive table of knots with up to ten crossings.
The table collects all prime knots with up to ten crossings, their mirror images, and all the possible connected sums that stay within this crossing limit.
Completeness follows from the classification of knot diagrams through ten crossings by Cantarella, Chapman, and Mastin~\cite{Cantarella}. 
In this range, every composite knot admitting an $n$-crossing diagram has prime summands whose crossing numbers total at most $n$. Hence, the table includes every possible knot type.
For each entry of the table, we store its PD code, its crossing signs, and a vector of pre-computed invariants: Seifert genus, $\tau$, total rank of the hat version of knot Floer homology, determinant, Alexander polynomial, $\sigma$, and Jones polynomial.
Given a diagram for a component $K_i$ of our link $L$, the identification proceeds as follows:
\begin{itemize}
    \item We simplify the diagram of $K_i$. 
    If it simplifies to ten crossings or fewer, we proceed. Otherwise, we attempt identification through visible connected summands, each within the table’s crossing limit.
    \item We compute the invariant vector for $K_i$ and check how many table entries match it.
    \item If the invariant vector matches a single entry, the knot is uniquely identified. If the remaining candidates differ only by mirroring, it is identified up to mirroring. Otherwise, we retain all matching candidates. For each invariant, we use the minimum of their lower bounds and the maximum of their upper bounds, provided the corresponding bounds are available for every candidate.
    \item For prime candidates, we retrieve the available lower and upper bounds on the unknotting number and slice genus from our database. For composite candidates, we obtain upper bounds by summing the corresponding upper bounds of their prime summands, using subadditivity. After combining candidates as above, we combine these bounds with the independently computed bounds for $K_i$ and apply the relevant link inequalities.
\end{itemize}

\section{The unknotting sequences}
\label{app:sequences}

\raggedbottom
\captionsetup[figure]{justification=raggedright, singlelinecheck=false}

In this appendix, we collect the drawings of the unknotting sequences of Table~\ref{tab:symbionts-unknotting-values} that differ from the sequence~\eqref{eq:unknotting-sequence-briher} of Brittenham and Hermiller.
In each subsection, we present one sequence $D_0 \to D_1 \to D_2 \to D_3 \to D_4$, where $D_4$ is a diagram of the unknot.
We draw one figure per crossing change: the diagram $D_k$ with the band of the crossing change in dashed red, before the simplification that gives $D_{k+1}$.
The circled numbers are the crossings of the PD code of $D_k$ and the grey numbers its edge labels; the arrows give the orientation.
For each drawing we report
\begin{itemize}
    \item its PD code before the crossing change;
    \item its crossing signs;
    \item the band path, which is drawn in dashed red, as a list of actions (Section~\ref{subs:env-mdp}), as printed by the replay scripts of the supporting data.
\end{itemize}
An action \texttt{C\{x\}S\{y\}\_\{type\}} refers to the crossing strand $(x,y)$: the edge meeting crossing $x$ at position $y$ of its PD tuple, the positions being numbered counter-clockwise from the incoming under-strand.
A band with \texttt{start} action \texttt{C\{x\}S\{y\}\_start} starts on that edge, in the face whose boundary contains the crossing strands $(x,y)$ and $(x, y-1 \bmod 4)$; that is, in the corner of crossing $x$ between them.
An action \texttt{C\{x\}S\{y\}\_over} (resp.\ \texttt{C\{x\}S\{y\}\_under}) takes the band over (resp.\ under) the edge of $(x,y)$ into the next face.
The last action, \texttt{C\{x\}S\{y\}\_end1} or \texttt{C\{x\}S\{y\}\_end2}, takes the band across the edge of $(x,y)$, which lies on the boundary of the face the band is in, so that the band clasps that edge.
The two endings tell us how the crossing change is performed.
The clasp gives two crossings between the band and the final edge.
Following the knot from crossing $x$ (the $x$ in \texttt{C\{x\}S\{y\}\_start}) along the band, we meet one of them before the tip of the band, the closer crossing, and the other one after it, the farther crossing.
With \texttt{end1}, the band passes over the final edge at the closer crossing and under it at the farther one.
With \texttt{end2}, the opposite.
Diagrams that are the same up to relabelling are drawn in the same way.

\subsection{Unknotting sequence for \texorpdfstring{$\mathtt{9_{8}} \sharp \mathtt{9_{10}}$}{9\_8 \# 9\_10}}
\label{app:seq-rweval_9_8_9_10}

The sequence found by the tuned random walker is
\begin{equation}
\label{eq:seq-rweval_9_8_9_10}
     \mathtt{9_{8}} \sharp \mathtt{9_{10}} \to \mathtt{4_{1}} \sharp \mathtt{9_{10}} \to \mathtt{K15n4005} \to \mathtt{8_{14}} \to \text{unknot}.
\end{equation}

\bigskip
\noindent
\begin{minipage}[c]{0.45\linewidth}
In Figure~\ref{fig:seq-rweval_9_8_9_10-0} we see the diagram $D_0$ of $\mathtt{9_{8}} \sharp \mathtt{9_{10}}$ to which we apply the first band.
Before the unknotting band, we have \\
PD code:
\begin{flushleft}\footnotesize\ttfamily
[(11, 26, 12, 27), (13, 24, 14, 25),\\
\hphantom{[}(15, 21, 16, 20), (17, 29, 18, 28),\\
\hphantom{[}(19, 23, 20, 22), (21, 17, 22, 16),\\
\hphantom{[}(23, 14, 24, 15), (25, 12, 26, 13),\\
\hphantom{[}(27, 19, 28, 18), (29, 5, 30, 4),\\
\hphantom{[}(31, 3, 32, 2), (33, 9, 34, 8),\\
\hphantom{[}(35, 7, 0, 6), (1, 11, 2, 10),\\
\hphantom{[}(3, 31, 4, 30), (5, 33, 6, 32),\\
\hphantom{[}(7, 35, 8, 34), (9, 1, 10, 0)].
\end{flushleft}
Crossing signs:
\begin{flushleft}\footnotesize\ttfamily
[-1, -1, 1, 1, 1, 1, -1, -1, 1,\\
\hphantom{[}1, 1, 1, 1, 1, 1, 1, 1, 1].
\end{flushleft}

The band applied is
\begin{flushleft}\footnotesize\ttfamily
[\textquotesingle{}C4S2\_start\textquotesingle{}, \textquotesingle{}C4S1\_end2\textquotesingle{}].
\end{flushleft}
\end{minipage}\hfill
\begin{minipage}[c]{0.52\linewidth}
    \captionsetup{type=figure}
    \centering
    \includegraphics[max width=\linewidth, max height=6.5cm]{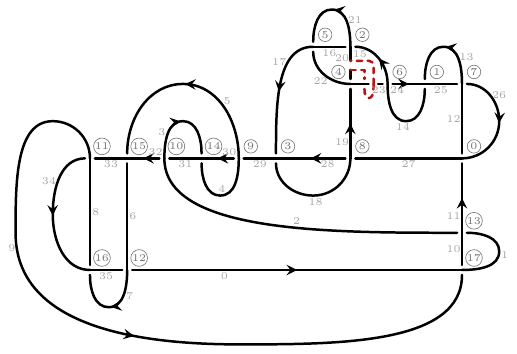}
    \caption{Sequence~\eqref{eq:seq-rweval_9_8_9_10}, band from diagram $D_{0}$ of $\mathtt{9_{8}} \sharp \mathtt{9_{10}}$ to diagram $D_{1}$ of $\mathtt{4_{1}} \sharp \mathtt{9_{10}}$.}
    \label{fig:seq-rweval_9_8_9_10-0}
\end{minipage}

\bigskip
\noindent
\begin{minipage}[c]{0.52\linewidth}
    \captionsetup{type=figure}
    \centering
    \includegraphics[max width=\linewidth, max height=6.5cm]{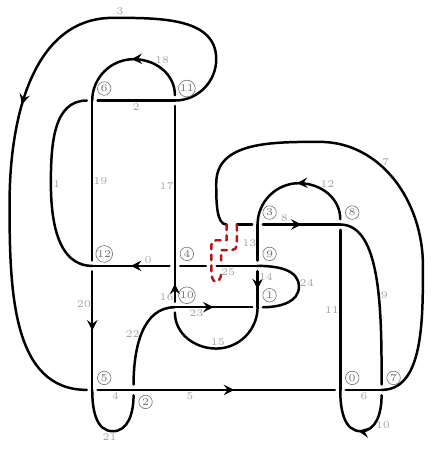}
    \caption{Sequence~\eqref{eq:seq-rweval_9_8_9_10}, band from diagram $D_{1}$ of $\mathtt{4_{1}} \sharp \mathtt{9_{10}}$ to diagram $D_{2}$ of $\mathtt{K15n4005}$.}
    \label{fig:seq-rweval_9_8_9_10-1}
\end{minipage}\hfill
\begin{minipage}[c]{0.45\linewidth}
In Figure~\ref{fig:seq-rweval_9_8_9_10-1} we see the diagram $D_1$ of $\mathtt{4_{1}} \sharp \mathtt{9_{10}}$ to which we apply the second band.
Before the unknotting band, we have \\
PD code:
\begin{flushleft}\footnotesize\ttfamily
[(5, 10, 6, 11), (23, 15, 24, 14),\\
\hphantom{[}(21, 5, 22, 4), (7, 13, 8, 12),\\
\hphantom{[}(25, 17, 0, 16), (3, 21, 4, 20),\\
\hphantom{[}(1, 19, 2, 18), (9, 6, 10, 7),\\
\hphantom{[}(11, 9, 12, 8), (13, 25, 14, 24),\\
\hphantom{[}(15, 23, 16, 22), (17, 3, 18, 2),\\
\hphantom{[}(19, 1, 20, 0)].
\end{flushleft}
Crossing signs:
\begin{flushleft}\footnotesize\ttfamily
[-1, 1, 1, 1, 1, 1, 1, -1, 1,\\
\hphantom{[}1, 1, 1, 1].
\end{flushleft}

The band applied is
\begin{flushleft}\footnotesize\ttfamily
[\textquotesingle{}C7S3\_start\textquotesingle{}, \textquotesingle{}C4S0\_end1\textquotesingle{}].
\end{flushleft}
\end{minipage}

\bigskip
\noindent
\begin{minipage}[c]{0.45\linewidth}
In Figure~\ref{fig:seq-rweval_9_8_9_10-2} we see the diagram $D_2$ of $\mathtt{K15n4005}$ to which we apply the third band.
Before the unknotting band, we have \\
PD code:
\begin{flushleft}\footnotesize\ttfamily
[(10, 17, 11, 18), (1, 21, 2, 20),\\
\hphantom{[}(28, 10, 29, 9), (14, 22, 15, 21),\\
\hphantom{[}(4, 24, 5, 23), (8, 28, 9, 27),\\
\hphantom{[}(6, 26, 7, 25), (16, 11, 17, 12),\\
\hphantom{[}(18, 16, 19, 15), (19, 3, 20, 2),\\
\hphantom{[}(22, 0, 23, 29), (24, 8, 25, 7),\\
\hphantom{[}(26, 6, 27, 5), (3, 12, 4, 13),\\
\hphantom{[}(13, 0, 14, 1)].
\end{flushleft}
Crossing signs:
\begin{flushleft}\footnotesize\ttfamily
[-1, 1, 1, 1, 1, 1, 1, -1, 1,\\
\hphantom{[}1, 1, 1, 1, -1, -1].
\end{flushleft}

The band applied is
\begin{flushleft}\footnotesize\ttfamily
[\textquotesingle{}C8S0\_start\textquotesingle{}, \textquotesingle{}C10S3\_end2\textquotesingle{}].
\end{flushleft}
\end{minipage}\hfill
\begin{minipage}[c]{0.52\linewidth}
    \captionsetup{type=figure}
    \centering
    \includegraphics[max width=\linewidth, max height=6.5cm]{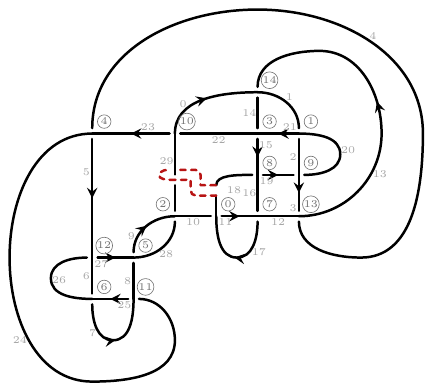}
    \caption{Sequence~\eqref{eq:seq-rweval_9_8_9_10}, band from diagram $D_{2}$ of $\mathtt{K15n4005}$ to diagram $D_{3}$ of $\mathtt{8_{14}}$.}
    \label{fig:seq-rweval_9_8_9_10-2}
\end{minipage}

\bigskip
\noindent
\begin{minipage}[c]{0.52\linewidth}
    \captionsetup{type=figure}
    \centering
    \includegraphics[max width=\linewidth, max height=5cm]{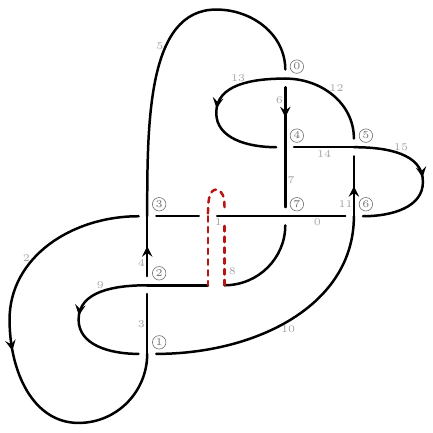}
    \caption{Sequence~\eqref{eq:seq-rweval_9_8_9_10}, band from diagram $D_{3}$ of $\mathtt{8_{14}}$ to diagram $D_{4}$ of the unknot.}
    \label{fig:seq-rweval_9_8_9_10-3}
\end{minipage}\hfill
\begin{minipage}[c]{0.45\linewidth}
In Figure~\ref{fig:seq-rweval_9_8_9_10-3} we see the diagram $D_3$ of $\mathtt{8_{14}}$ to which we apply the fourth band.
Before the unknotting band, we have \\
PD code:
\begin{flushleft}\footnotesize\ttfamily
[(5, 13, 6, 12), (9, 2, 10, 3),\\
\hphantom{[}(3, 8, 4, 9), (1, 5, 2, 4),\\
\hphantom{[}(13, 7, 14, 6), (11, 15, 12, 14),\\
\hphantom{[}(15, 11, 0, 10), (7, 1, 8, 0)].
\end{flushleft}
Crossing signs:
\begin{flushleft}\footnotesize\ttfamily
[1, -1, -1, 1, 1, 1, 1, 1].
\end{flushleft}

The band applied is
\begin{flushleft}\footnotesize\ttfamily
[\textquotesingle{}C7S2\_start\textquotesingle{}, \textquotesingle{}C7S1\_end2\textquotesingle{}].
\end{flushleft}
\end{minipage}

\medskip

After the last crossing change, we obtain the unknot.

\subsection{Unknotting sequence for \texorpdfstring{$\mathtt{9_{40}} \sharp \mathtt{9_{10}}$}{9\_40 \# 9\_10}}
\label{app:seq-rweval_9_40_9_10}

The sequence found by the tuned random walker is
\begin{equation}
\label{eq:seq-rweval_9_40_9_10}
     \mathtt{9_{40}} \sharp \mathtt{9_{10}} \to \mathtt{K20} \to \mathtt{K15n4005} \to \mathtt{8_{14}} \to \text{unknot}.
\end{equation}

\bigskip
\noindent
\begin{minipage}[c]{0.45\linewidth}
In Figure~\ref{fig:seq-rweval_9_40_9_10-0} we see the diagram $D_0$ of $\mathtt{9_{40}} \sharp \mathtt{9_{10}}$ to which we apply the first band.
Before the unknotting band, we have \\
PD code:
\begin{flushleft}\footnotesize\ttfamily
[(11, 25, 12, 24), (13, 22, 14, 23),\\
\hphantom{[}(15, 21, 16, 20), (17, 13, 18, 12),\\
\hphantom{[}(19, 28, 20, 29), (21, 27, 22, 26),\\
\hphantom{[}(23, 19, 24, 18), (25, 16, 26, 17),\\
\hphantom{[}(27, 15, 28, 14), (29, 5, 30, 4),\\
\hphantom{[}(31, 3, 32, 2), (33, 9, 34, 8),\\
\hphantom{[}(35, 7, 0, 6), (1, 11, 2, 10),\\
\hphantom{[}(3, 31, 4, 30), (5, 33, 6, 32),\\
\hphantom{[}(7, 35, 8, 34), (9, 1, 10, 0)].
\end{flushleft}
Crossing signs:
\begin{flushleft}\footnotesize\ttfamily
[1, -1, 1, 1, -1, 1, 1, -1, 1,\\
\hphantom{[}1, 1, 1, 1, 1, 1, 1, 1, 1].
\end{flushleft}

The band applied is
\begin{flushleft}\footnotesize\ttfamily
[\textquotesingle{}C2S3\_start\textquotesingle{}, \textquotesingle{}C15S0\_end2\textquotesingle{}].
\end{flushleft}
\end{minipage}\hfill
\begin{minipage}[c]{0.52\linewidth}
    \captionsetup{type=figure}
    \centering
    \includegraphics[max width=\linewidth, max height=6.5cm]{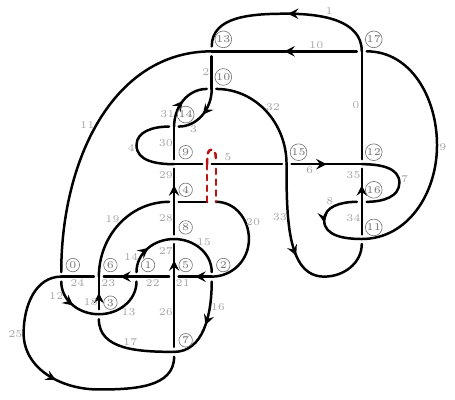}
    \caption{Sequence~\eqref{eq:seq-rweval_9_40_9_10}, band from diagram $D_{0}$ of $\mathtt{9_{40}} \sharp \mathtt{9_{10}}$ to diagram $D_{1}$ of $\mathtt{K20}$.}
    \label{fig:seq-rweval_9_40_9_10-0}
\end{minipage}

\bigskip
\noindent
\begin{minipage}[c]{0.52\linewidth}
    \captionsetup{type=figure}
    \centering
    \includegraphics[max width=\linewidth, max height=6.5cm]{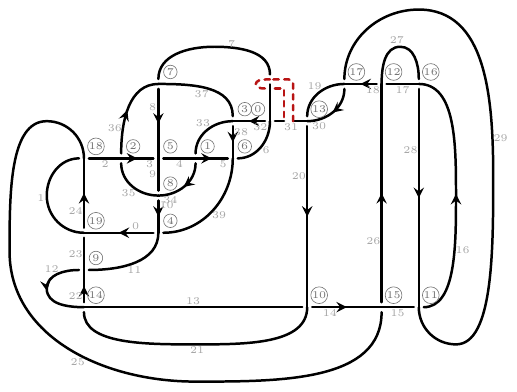}
    \caption{Sequence~\eqref{eq:seq-rweval_9_40_9_10}, band from diagram $D_{1}$ of $\mathtt{K20}$ to diagram $D_{2}$ of $\mathtt{K15n4005}$.}
    \label{fig:seq-rweval_9_40_9_10-1}
\end{minipage}\hfill
\begin{minipage}[c]{0.45\linewidth}
In Figure~\ref{fig:seq-rweval_9_40_9_10-1} we see the diagram $D_1$ of $\mathtt{K20}$ to which we apply the second band.
Before the unknotting band, we have \\
PD code:
\begin{flushleft}\footnotesize\ttfamily
[(31, 7, 32, 6), (33, 4, 34, 5),\\
\hphantom{[}(35, 3, 36, 2), (37, 33, 38, 32),\\
\hphantom{[}(39, 10, 0, 11), (3, 9, 4, 8),\\
\hphantom{[}(5, 39, 6, 38), (7, 36, 8, 37),\\
\hphantom{[}(9, 35, 10, 34), (11, 23, 12, 22),\\
\hphantom{[}(13, 21, 14, 20), (15, 29, 16, 28),\\
\hphantom{[}(17, 27, 18, 26), (19, 31, 20, 30),\\
\hphantom{[}(21, 13, 22, 12), (25, 15, 26, 14),\\
\hphantom{[}(27, 17, 28, 16), (29, 19, 30, 18),\\
\hphantom{[}(1, 24, 2, 25), (23, 0, 24, 1)].
\end{flushleft}
Crossing signs:
\begin{flushleft}\footnotesize\ttfamily
[1, -1, 1, 1, -1, 1, 1, -1, 1,\\
\hphantom{[}1, 1, 1, 1, 1, 1, 1, 1, 1,\\
\hphantom{[}-1, -1].
\end{flushleft}

The band applied is
\begin{flushleft}\footnotesize\ttfamily
[\textquotesingle{}C13S1\_start\textquotesingle{}, \textquotesingle{}C7S0\_end1\textquotesingle{}].
\end{flushleft}
\end{minipage}

\bigskip
\noindent
\begin{minipage}[c]{0.45\linewidth}
In Figure~\ref{fig:seq-rweval_9_40_9_10-2} we see the diagram $D_2$ of $\mathtt{K15n4005}$ to which we apply the third band.
Before the unknotting band, we have \\
PD code:
\begin{flushleft}\footnotesize\ttfamily
[(27, 20, 28, 21), (21, 26, 22, 27),\\
\hphantom{[}(18, 8, 19, 7), (16, 6, 17, 5),\\
\hphantom{[}(14, 4, 15, 3), (24, 0, 25, 29),\\
\hphantom{[}(9, 22, 10, 23), (25, 12, 26, 13),\\
\hphantom{[}(2, 24, 3, 23), (28, 14, 29, 13),\\
\hphantom{[}(1, 11, 2, 10), (4, 18, 5, 17),\\
\hphantom{[}(6, 16, 7, 15), (8, 20, 9, 19),\\
\hphantom{[}(11, 1, 12, 0)].
\end{flushleft}
Crossing signs:
\begin{flushleft}\footnotesize\ttfamily
[-1, -1, 1, 1, 1, 1, -1, -1, 1,\\
\hphantom{[}1, 1, 1, 1, 1, 1].
\end{flushleft}

The band applied is
\begin{flushleft}\footnotesize\ttfamily
[\textquotesingle{}C1S3\_start\textquotesingle{}, \textquotesingle{}C6S0\_end1\textquotesingle{}].
\end{flushleft}
\end{minipage}\hfill
\begin{minipage}[c]{0.52\linewidth}
    \captionsetup{type=figure}
    \centering
    \includegraphics[max width=\linewidth, max height=6.5cm]{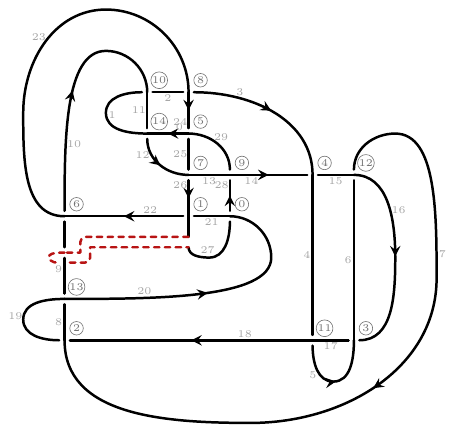}
    \caption{Sequence~\eqref{eq:seq-rweval_9_40_9_10}, band from diagram $D_{2}$ of $\mathtt{K15n4005}$ to diagram $D_{3}$ of $\mathtt{8_{14}}$.}
    \label{fig:seq-rweval_9_40_9_10-2}
\end{minipage}

\bigskip
\noindent
\begin{minipage}[c]{0.52\linewidth}
    \captionsetup{type=figure}
    \centering
    \includegraphics[max width=\linewidth, max height=5cm]{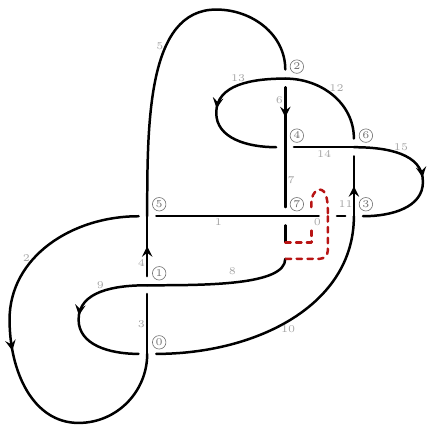}
    \caption{Sequence~\eqref{eq:seq-rweval_9_40_9_10}, band from diagram $D_{3}$ of $\mathtt{8_{14}}$ to diagram $D_{4}$ of the unknot.}
    \label{fig:seq-rweval_9_40_9_10-3}
\end{minipage}\hfill
\begin{minipage}[c]{0.45\linewidth}
In Figure~\ref{fig:seq-rweval_9_40_9_10-3} we see the diagram $D_3$ of $\mathtt{8_{14}}$ to which we apply the fourth band.
Before the unknotting band, we have \\
PD code:
\begin{flushleft}\footnotesize\ttfamily
[(9, 2, 10, 3), (3, 8, 4, 9),\\
\hphantom{[}(5, 13, 6, 12), (15, 11, 0, 10),\\
\hphantom{[}(13, 7, 14, 6), (1, 5, 2, 4),\\
\hphantom{[}(11, 15, 12, 14), (7, 1, 8, 0)].
\end{flushleft}
Crossing signs:
\begin{flushleft}\footnotesize\ttfamily
[-1, -1, 1, 1, 1, 1, 1, 1].
\end{flushleft}

The band applied is
\begin{flushleft}\footnotesize\ttfamily
[\textquotesingle{}C1S1\_start\textquotesingle{}, \textquotesingle{}C3S2\_end1\textquotesingle{}].
\end{flushleft}
\end{minipage}

\medskip

After the last crossing change, we obtain the unknot.

\subsection{Unknotting sequence for \texorpdfstring{$\mathtt{10_{41}} \sharp \mathtt{9_{10}}$}{10\_41 \# 9\_10}}
\label{app:seq-rweval_10_41_9_10}

The sequence found by the tuned random walker is
\begin{equation}
\label{eq:seq-rweval_10_41_9_10}
     \mathtt{10_{41}} \sharp \mathtt{9_{10}} \to \mathtt{K21} \to \mathtt{K15n4005} \to \mathtt{8_{14}} \to \text{unknot}.
\end{equation}

\bigskip
\noindent
\begin{minipage}[c]{0.45\linewidth}
In Figure~\ref{fig:seq-rweval_10_41_9_10-0} we see the diagram $D_0$ of $\mathtt{10_{41}} \sharp \mathtt{9_{10}}$ to which we apply the first band.
Before the unknotting band, we have \\
PD code:
\begin{flushleft}\footnotesize\ttfamily
[(11, 24, 12, 25), (13, 27, 14, 26),\\
\hphantom{[}(15, 20, 16, 21), (17, 29, 18, 28),\\
\hphantom{[}(19, 16, 20, 17), (21, 31, 22, 30),\\
\hphantom{[}(23, 12, 24, 13), (25, 23, 26, 22),\\
\hphantom{[}(27, 19, 28, 18), (29, 15, 30, 14),\\
\hphantom{[}(31, 5, 32, 4), (33, 3, 34, 2),\\
\hphantom{[}(35, 9, 36, 8), (37, 7, 0, 6),\\
\hphantom{[}(1, 11, 2, 10), (3, 33, 4, 32),\\
\hphantom{[}(5, 35, 6, 34), (7, 37, 8, 36),\\
\hphantom{[}(9, 1, 10, 0)].
\end{flushleft}
Crossing signs:
\begin{flushleft}\footnotesize\ttfamily
[-1, 1, -1, 1, -1, 1, -1, 1, 1,\\
\hphantom{[}1, 1, 1, 1, 1, 1, 1, 1, 1,\\
\hphantom{[}1].
\end{flushleft}

The band applied is
\begin{flushleft}\footnotesize\ttfamily
[\textquotesingle{}C10S1\_start\textquotesingle{}, \textquotesingle{}C2S2\_end2\textquotesingle{}].
\end{flushleft}
\end{minipage}\hfill
\begin{minipage}[c]{0.52\linewidth}
    \captionsetup{type=figure}
    \centering
    \includegraphics[max width=\linewidth, max height=6.5cm]{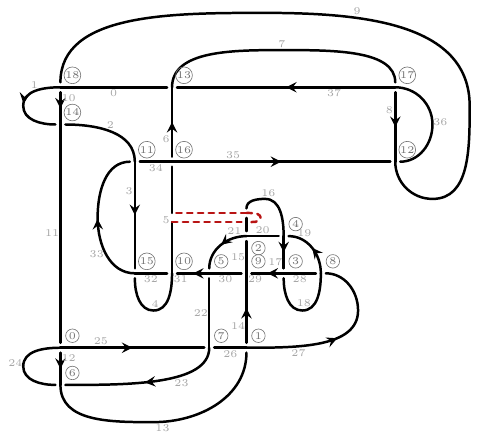}
    \caption{Sequence~\eqref{eq:seq-rweval_10_41_9_10}, band from diagram $D_{0}$ of $\mathtt{10_{41}} \sharp \mathtt{9_{10}}$ to diagram $D_{1}$ of $\mathtt{K21}$.}
    \label{fig:seq-rweval_10_41_9_10-0}
\end{minipage}

\bigskip
\noindent
\begin{minipage}[c]{0.52\linewidth}
    \captionsetup{type=figure}
    \centering
    \includegraphics[max width=\linewidth, max height=6.5cm]{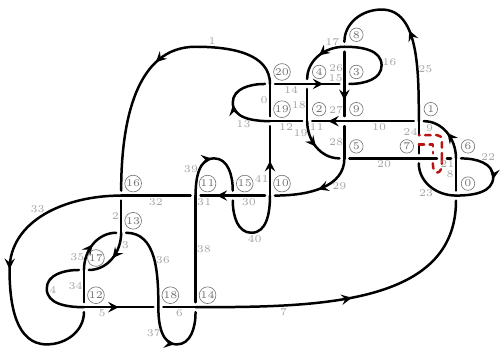}
    \caption{Sequence~\eqref{eq:seq-rweval_10_41_9_10}, band from diagram $D_{1}$ of $\mathtt{K21}$ to diagram $D_{2}$ of $\mathtt{K15n4005}$.}
    \label{fig:seq-rweval_10_41_9_10-1}
\end{minipage}\hfill
\begin{minipage}[c]{0.45\linewidth}
In Figure~\ref{fig:seq-rweval_10_41_9_10-1} we see the diagram $D_1$ of $\mathtt{K21}$ to which we apply the second band.
Before the unknotting band, we have \\
PD code:
\begin{flushleft}\footnotesize\ttfamily
[(7, 22, 8, 23), (9, 25, 10, 24),\\
\hphantom{[}(11, 18, 12, 19), (15, 27, 16, 26),\\
\hphantom{[}(17, 14, 18, 15), (19, 29, 20, 28),\\
\hphantom{[}(21, 8, 22, 9), (23, 21, 24, 20),\\
\hphantom{[}(25, 17, 26, 16), (27, 11, 28, 10),\\
\hphantom{[}(29, 41, 30, 40), (31, 39, 32, 38),\\
\hphantom{[}(33, 5, 34, 4), (35, 3, 36, 2),\\
\hphantom{[}(37, 7, 38, 6), (39, 31, 40, 30),\\
\hphantom{[}(1, 33, 2, 32), (3, 35, 4, 34),\\
\hphantom{[}(5, 37, 6, 36), (41, 12, 0, 13),\\
\hphantom{[}(13, 0, 14, 1)].
\end{flushleft}
Crossing signs:
\begin{flushleft}\footnotesize\ttfamily
[-1, 1, -1, 1, -1, 1, -1, 1, 1,\\
\hphantom{[}1, 1, 1, 1, 1, 1, 1, 1, 1,\\
\hphantom{[}1, -1, -1].
\end{flushleft}

The band applied is
\begin{flushleft}\footnotesize\ttfamily
[\textquotesingle{}C7S2\_start\textquotesingle{}, \textquotesingle{}C7S1\_end2\textquotesingle{}].
\end{flushleft}
\end{minipage}

\bigskip
\noindent
\begin{minipage}[c]{0.45\linewidth}
In Figure~\ref{fig:seq-rweval_10_41_9_10-2} we see the diagram $D_2$ of $\mathtt{K15n4005}$ to which we apply the third band.
Before the unknotting band, we have \\
PD code:
\begin{flushleft}\footnotesize\ttfamily
[(3, 20, 4, 21), (18, 1, 19, 2),\\
\hphantom{[}(29, 4, 0, 5), (2, 8, 3, 7),\\
\hphantom{[}(22, 12, 23, 11), (24, 14, 25, 13),\\
\hphantom{[}(5, 28, 6, 29), (26, 16, 27, 15),\\
\hphantom{[}(19, 9, 20, 8), (16, 28, 17, 27),\\
\hphantom{[}(14, 24, 15, 23), (10, 18, 11, 17),\\
\hphantom{[}(12, 26, 13, 25), (6, 22, 7, 21),\\
\hphantom{[}(9, 1, 10, 0)].
\end{flushleft}
Crossing signs:
\begin{flushleft}\footnotesize\ttfamily
[-1, -1, -1, 1, 1, 1, -1, 1, 1,\\
\hphantom{[}1, 1, 1, 1, 1, 1].
\end{flushleft}

The band applied is
\begin{flushleft}\footnotesize\ttfamily
[\textquotesingle{}C6S2\_start\textquotesingle{}, \textquotesingle{}C9S0\_end1\textquotesingle{}].
\end{flushleft}
\end{minipage}\hfill
\begin{minipage}[c]{0.52\linewidth}
    \captionsetup{type=figure}
    \centering
    \includegraphics[max width=\linewidth, max height=6.5cm]{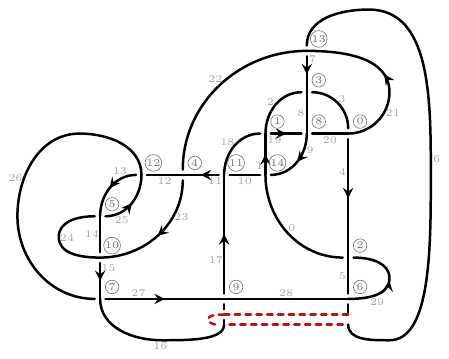}
    \caption{Sequence~\eqref{eq:seq-rweval_10_41_9_10}, band from diagram $D_{2}$ of $\mathtt{K15n4005}$ to diagram $D_{3}$ of $\mathtt{8_{14}}$.}
    \label{fig:seq-rweval_10_41_9_10-2}
\end{minipage}

\bigskip
\noindent
\begin{minipage}[c]{0.52\linewidth}
    \captionsetup{type=figure}
    \centering
    \includegraphics[max width=\linewidth, max height=5cm]{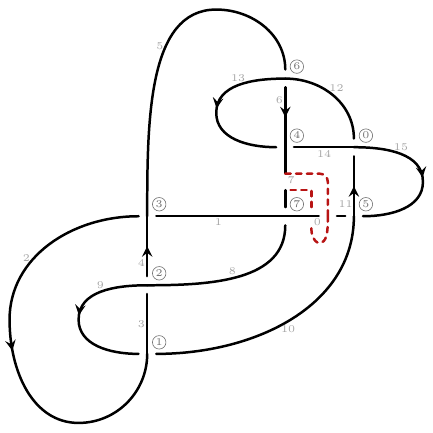}
    \caption{Sequence~\eqref{eq:seq-rweval_10_41_9_10}, band from $D_{3}$ of $\mathtt{8_{14}}$ to $D_{4}$ of the unknot.}
    \label{fig:seq-rweval_10_41_9_10-3}
\end{minipage}\hfill
\begin{minipage}[c]{0.45\linewidth}
In Figure~\ref{fig:seq-rweval_10_41_9_10-3} we see the diagram $D_3$ of $\mathtt{8_{14}}$ to which we apply the fourth band.
Before the unknotting band, we have \\
PD code:
\begin{flushleft}\footnotesize\ttfamily
[(11, 15, 12, 14), (9, 2, 10, 3),\\
\hphantom{[}(3, 8, 4, 9), (1, 5, 2, 4),\\
\hphantom{[}(13, 7, 14, 6), (15, 11, 0, 10),\\
\hphantom{[}(5, 13, 6, 12), (7, 1, 8, 0)].
\end{flushleft}
Crossing signs:
\begin{flushleft}\footnotesize\ttfamily
[1, -1, -1, 1, 1, 1, 1, 1].
\end{flushleft}

The band applied is
\begin{flushleft}\footnotesize\ttfamily
[\textquotesingle{}C7S0\_start\textquotesingle{}, \textquotesingle{}C7S3\_end2\textquotesingle{}].
\end{flushleft}
\end{minipage}

\medskip

After the last crossing change, we obtain the unknot.

\subsection{Unknotting sequence for \texorpdfstring{$\mathtt{10_{115}} \sharp \mathtt{9_{10}}$}{10\_115 \# 9\_10}}
\label{app:seq-rweval_10_115_9_10}

The sequence found by the tuned random walker is
\begin{equation}
\label{eq:seq-rweval_10_115_9_10}
     \mathtt{10_{115}} \sharp \mathtt{9_{10}} \to \mathtt{4_{1}} \sharp \mathtt{9_{10}} \to \mathtt{K15n4005} \to \mathtt{8_{14}} \to \text{unknot}.
\end{equation}
The diagrams $D_2$ and $D_3$ and their bands are those of the sequence of $\mathtt{9_{40}} \sharp \mathtt{9_{10}}$, up to relabelling.

\bigskip
\noindent
\begin{minipage}[c]{0.45\linewidth}
In Figure~\ref{fig:seq-rweval_10_115_9_10-0} we see the diagram $D_0$ of $\mathtt{10_{115}} \sharp \mathtt{9_{10}}$ to which we apply the first band.
Before the unknotting band, we have \\
PD code:
\begin{flushleft}\footnotesize\ttfamily
[(11, 27, 12, 26), (19, 31, 20, 30),\\
\hphantom{[}(25, 20, 26, 21), (21, 12, 22, 13),\\
\hphantom{[}(13, 24, 14, 25), (23, 16, 24, 17),\\
\hphantom{[}(15, 29, 16, 28), (29, 15, 30, 14),\\
\hphantom{[}(17, 22, 18, 23), (27, 19, 28, 18),\\
\hphantom{[}(31, 5, 32, 4), (33, 3, 34, 2),\\
\hphantom{[}(35, 9, 36, 8), (37, 7, 0, 6),\\
\hphantom{[}(1, 11, 2, 10), (3, 33, 4, 32),\\
\hphantom{[}(5, 35, 6, 34), (7, 37, 8, 36),\\
\hphantom{[}(9, 1, 10, 0)].
\end{flushleft}
Crossing signs:
\begin{flushleft}\footnotesize\ttfamily
[1, 1, -1, -1, -1, -1, 1, 1, -1,\\
\hphantom{[}1, 1, 1, 1, 1, 1, 1, 1, 1,\\
\hphantom{[}1].
\end{flushleft}

The band applied is
\begin{flushleft}\footnotesize\ttfamily
[\textquotesingle{}C2S1\_start\textquotesingle{}, \textquotesingle{}C2S0\_end1\textquotesingle{}].
\end{flushleft}
\end{minipage}\hfill
\begin{minipage}[c]{0.52\linewidth}
    \captionsetup{type=figure}
    \centering
    \includegraphics[max width=\linewidth, max height=6.5cm]{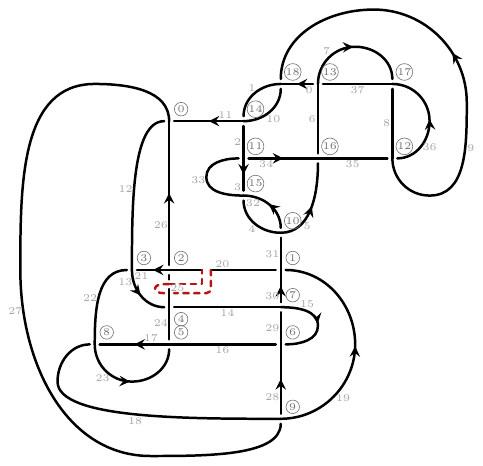}
    \caption{Sequence~\eqref{eq:seq-rweval_10_115_9_10}, band from diagram $D_{0}$ of $\mathtt{10_{115}} \sharp \mathtt{9_{10}}$ to diagram $D_{1}$ of $\mathtt{4_{1}} \sharp \mathtt{9_{10}}$.}
    \label{fig:seq-rweval_10_115_9_10-0}
\end{minipage}

\bigskip
\noindent
\begin{minipage}[c]{0.52\linewidth}
    \captionsetup{type=figure}
    \centering
    \includegraphics[max width=\linewidth, max height=6.5cm]{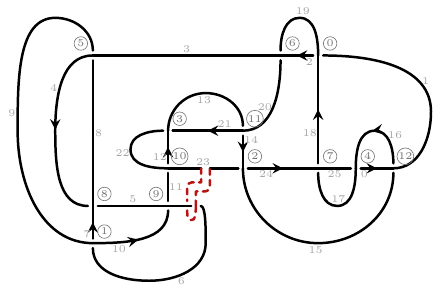}
    \caption{Sequence~\eqref{eq:seq-rweval_10_115_9_10}, band from diagram $D_{1}$ of $\mathtt{4_{1}} \sharp \mathtt{9_{10}}$ to diagram $D_{2}$ of $\mathtt{K15n4005}$.}
    \label{fig:seq-rweval_10_115_9_10-1}
\end{minipage}\hfill
\begin{minipage}[c]{0.45\linewidth}
In Figure~\ref{fig:seq-rweval_10_115_9_10-1} we see the diagram $D_1$ of $\mathtt{4_{1}} \sharp \mathtt{9_{10}}$ to which we apply the second band.
Before the unknotting band, we have \\
PD code:
\begin{flushleft}\footnotesize\ttfamily
[(1, 19, 2, 18), (6, 10, 7, 9),\\
\hphantom{[}(23, 15, 24, 14), (21, 13, 22, 12),\\
\hphantom{[}(25, 17, 0, 16), (8, 3, 9, 4),\\
\hphantom{[}(19, 3, 20, 2), (17, 25, 18, 24),\\
\hphantom{[}(4, 7, 5, 8), (10, 6, 11, 5),\\
\hphantom{[}(11, 23, 12, 22), (13, 21, 14, 20),\\
\hphantom{[}(15, 1, 16, 0)].
\end{flushleft}
Crossing signs:
\begin{flushleft}\footnotesize\ttfamily
[1, 1, 1, 1, 1, -1, 1, 1, -1,\\
\hphantom{[}1, 1, 1, 1].
\end{flushleft}

The band applied is
\begin{flushleft}\footnotesize\ttfamily
[\textquotesingle{}C10S1\_start\textquotesingle{}, \textquotesingle{}C9S1\_end2\textquotesingle{}].
\end{flushleft}
\end{minipage}

\bigskip
\noindent
\begin{minipage}[c]{0.45\linewidth}
In Figure~\ref{fig:seq-rweval_10_115_9_10-2} we see the diagram $D_2$ of $\mathtt{K15n4005}$ to which we apply the third band.
Before the unknotting band, we have \\
PD code:
\begin{flushleft}\footnotesize\ttfamily
[(6, 26, 7, 25), (12, 18, 13, 17),\\
\hphantom{[}(2, 22, 3, 21), (29, 19, 0, 18),\\
\hphantom{[}(4, 24, 5, 23), (15, 8, 16, 9),\\
\hphantom{[}(26, 8, 27, 7), (24, 4, 25, 3),\\
\hphantom{[}(9, 14, 10, 15), (20, 12, 21, 11),\\
\hphantom{[}(16, 2, 17, 1), (19, 29, 20, 28),\\
\hphantom{[}(22, 6, 23, 5), (27, 10, 28, 11),\\
\hphantom{[}(13, 0, 14, 1)].
\end{flushleft}
Crossing signs:
\begin{flushleft}\footnotesize\ttfamily
[1, 1, 1, 1, 1, -1, 1, 1, -1,\\
\hphantom{[}1, 1, 1, 1, -1, -1].
\end{flushleft}

The band applied is
\begin{flushleft}\footnotesize\ttfamily
[\textquotesingle{}C8S3\_start\textquotesingle{}, \textquotesingle{}C13S0\_end1\textquotesingle{}].
\end{flushleft}
\end{minipage}\hfill
\begin{minipage}[c]{0.52\linewidth}
    \captionsetup{type=figure}
    \centering
    \includegraphics[max width=\linewidth, max height=6.5cm]{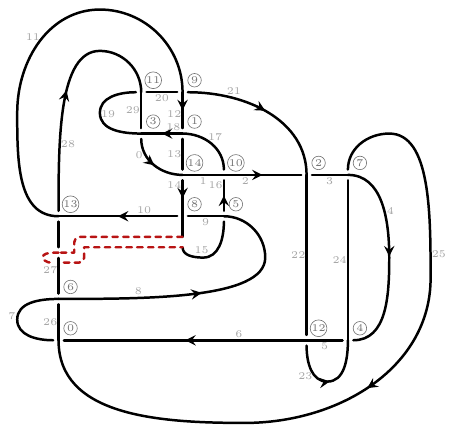}
    \caption{Sequence~\eqref{eq:seq-rweval_10_115_9_10}, band from diagram $D_{2}$ of $\mathtt{K15n4005}$ to diagram $D_{3}$ of $\mathtt{8_{14}}$.}
    \label{fig:seq-rweval_10_115_9_10-2}
\end{minipage}

\bigskip
\noindent
\begin{minipage}[c]{0.52\linewidth}
    \captionsetup{type=figure}
    \centering
    \includegraphics[max width=\linewidth, max height=5cm]{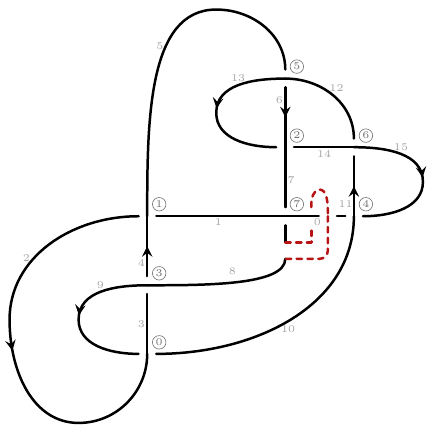}
    \caption{Sequence~\eqref{eq:seq-rweval_10_115_9_10}, band from diagram $D_{3}$ of $\mathtt{8_{14}}$ to diagram $D_{4}$ of the unknot.}
    \label{fig:seq-rweval_10_115_9_10-3}
\end{minipage}\hfill
\begin{minipage}[c]{0.45\linewidth}
In Figure~\ref{fig:seq-rweval_10_115_9_10-3} we see the diagram $D_3$ of $\mathtt{8_{14}}$ to which we apply the fourth band.
Before the unknotting band, we have \\
PD code:
\begin{flushleft}\footnotesize\ttfamily
[(9, 2, 10, 3), (1, 5, 2, 4),\\
\hphantom{[}(13, 7, 14, 6), (3, 8, 4, 9),\\
\hphantom{[}(15, 11, 0, 10), (5, 13, 6, 12),\\
\hphantom{[}(11, 15, 12, 14), (7, 1, 8, 0)].
\end{flushleft}
Crossing signs:
\begin{flushleft}\footnotesize\ttfamily
[-1, 1, 1, -1, 1, 1, 1, 1].
\end{flushleft}

The band applied is
\begin{flushleft}\footnotesize\ttfamily
[\textquotesingle{}C3S1\_start\textquotesingle{}, \textquotesingle{}C4S2\_end1\textquotesingle{}].
\end{flushleft}
\end{minipage}

\medskip

After the last crossing change, we obtain the unknot.

\subsection{Unknotting sequence for \texorpdfstring{$\mathtt{10_{121}} \sharp \mathtt{9_{10}}$}{10\_121 \# 9\_10}}
\label{app:seq-rweval_10_121_9_10}

The sequence found by the tuned random walker is
\begin{equation}
\label{eq:seq-rweval_10_121_9_10}
     \mathtt{10_{121}} \sharp \mathtt{9_{10}} \to \mathtt{K21} \to \mathtt{K23} \to \mathtt{8_{14}} \to \text{unknot}.
\end{equation}

\bigskip
\noindent
\begin{minipage}[c]{0.45\linewidth}
In Figure~\ref{fig:seq-rweval_10_121_9_10-0} we see the diagram $D_0$ of $\mathtt{10_{121}} \sharp \mathtt{9_{10}}$ to which we apply the first band.
Before the unknotting band, we have \\
PD code:
\begin{flushleft}\footnotesize\ttfamily
[(11, 20, 12, 21), (13, 19, 14, 18),\\
\hphantom{[}(15, 31, 16, 30), (17, 25, 18, 24),\\
\hphantom{[}(19, 26, 20, 27), (21, 17, 22, 16),\\
\hphantom{[}(23, 29, 24, 28), (25, 12, 26, 13),\\
\hphantom{[}(27, 15, 28, 14), (29, 23, 30, 22),\\
\hphantom{[}(31, 5, 32, 4), (33, 3, 34, 2),\\
\hphantom{[}(35, 9, 36, 8), (37, 7, 0, 6),\\
\hphantom{[}(1, 11, 2, 10), (3, 33, 4, 32),\\
\hphantom{[}(5, 35, 6, 34), (7, 37, 8, 36),\\
\hphantom{[}(9, 1, 10, 0)].
\end{flushleft}
Crossing signs:
\begin{flushleft}\footnotesize\ttfamily
[-1, 1, 1, 1, -1, 1, 1, -1, 1,\\
\hphantom{[}1, 1, 1, 1, 1, 1, 1, 1, 1,\\
\hphantom{[}1].
\end{flushleft}

The band applied is
\begin{flushleft}\footnotesize\ttfamily
[\textquotesingle{}C8S1\_start\textquotesingle{}, \textquotesingle{}C16S0\_end1\textquotesingle{}].
\end{flushleft}
\end{minipage}\hfill
\begin{minipage}[c]{0.52\linewidth}
    \captionsetup{type=figure}
    \centering
    \includegraphics[max width=\linewidth, max height=6.5cm]{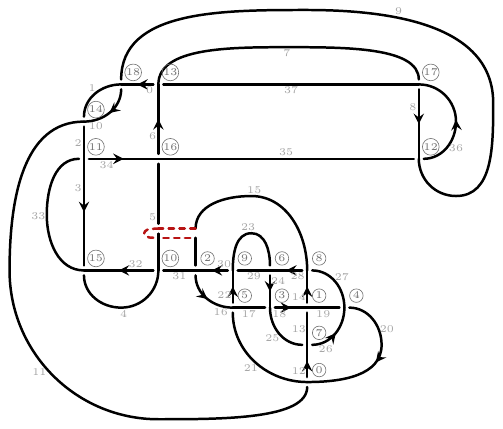}
    \caption{Sequence~\eqref{eq:seq-rweval_10_121_9_10}, band from diagram $D_{0}$ of $\mathtt{10_{121}} \sharp \mathtt{9_{10}}$ to diagram $D_{1}$ of $\mathtt{K21}$.}
    \label{fig:seq-rweval_10_121_9_10-0}
\end{minipage}

\bigskip
\noindent
\begin{minipage}[c]{0.52\linewidth}
    \captionsetup{type=figure}
    \centering
    \includegraphics[max width=\linewidth, max height=6.5cm]{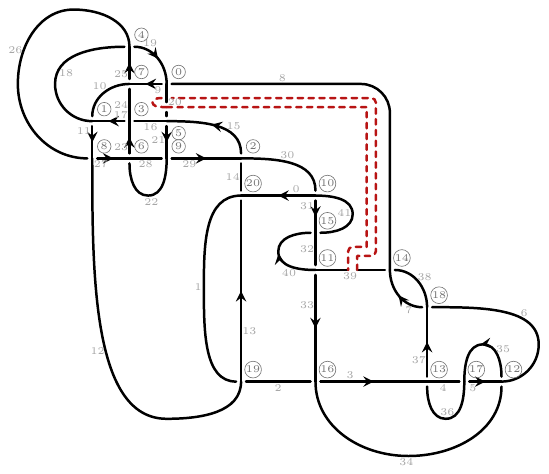}
    \caption{Sequence~\eqref{eq:seq-rweval_10_121_9_10}, band from diagram $D_{1}$ of $\mathtt{K21}$ to diagram $D_{2}$ of $\mathtt{K23}$.}
    \label{fig:seq-rweval_10_121_9_10-1}
\end{minipage}\hfill
\begin{minipage}[c]{0.45\linewidth}
In Figure~\ref{fig:seq-rweval_10_121_9_10-1} we see the diagram $D_1$ of $\mathtt{K21}$ to which we apply the second band.
Before the unknotting band, we have \\
PD code:
\begin{flushleft}\footnotesize\ttfamily
[(8, 19, 9, 20), (10, 18, 11, 17),\\
\hphantom{[}(14, 30, 15, 29), (16, 24, 17, 23),\\
\hphantom{[}(18, 25, 19, 26), (20, 16, 21, 15),\\
\hphantom{[}(22, 28, 23, 27), (24, 9, 25, 10),\\
\hphantom{[}(26, 12, 27, 11), (28, 22, 29, 21),\\
\hphantom{[}(30, 0, 31, 41), (32, 40, 33, 39),\\
\hphantom{[}(34, 6, 35, 5), (36, 4, 37, 3),\\
\hphantom{[}(38, 8, 39, 7), (40, 32, 41, 31),\\
\hphantom{[}(2, 34, 3, 33), (4, 36, 5, 35),\\
\hphantom{[}(6, 38, 7, 37), (1, 12, 2, 13),\\
\hphantom{[}(13, 0, 14, 1)].
\end{flushleft}
Crossing signs:
\begin{flushleft}\footnotesize\ttfamily
[-1, 1, 1, 1, -1, 1, 1, -1, 1,\\
\hphantom{[}1, 1, 1, 1, 1, 1, 1, 1, 1,\\
\hphantom{[}1, -1, -1].
\end{flushleft}

The band applied is
\begin{flushleft}\footnotesize\ttfamily
[\textquotesingle{}C14S2\_start\textquotesingle{}, \textquotesingle{}C0S3\_end2\textquotesingle{}].
\end{flushleft}
\end{minipage}

\bigskip
\noindent
\begin{minipage}[c]{0.45\linewidth}
In Figure~\ref{fig:seq-rweval_10_121_9_10-2} we see the diagram $D_2$ of $\mathtt{K23}$ to which we apply the third band.
Before the unknotting band, we have \\
PD code:
\begin{flushleft}\footnotesize\ttfamily
[(10, 23, 11, 24), (14, 22, 15, 21),\\
\hphantom{[}(18, 38, 19, 37), (20, 30, 21, 29),\\
\hphantom{[}(22, 31, 23, 32), (26, 20, 27, 19),\\
\hphantom{[}(28, 34, 29, 33), (30, 13, 31, 14),\\
\hphantom{[}(32, 16, 33, 15), (34, 28, 35, 27),\\
\hphantom{[}(35, 5, 36, 4), (38, 2, 39, 1),\\
\hphantom{[}(40, 10, 41, 9), (42, 8, 43, 7),\\
\hphantom{[}(45, 13, 0, 12), (3, 37, 4, 36),\\
\hphantom{[}(6, 40, 7, 39), (8, 42, 9, 41),\\
\hphantom{[}(11, 45, 12, 44), (5, 16, 6, 17),\\
\hphantom{[}(17, 2, 18, 3), (43, 24, 44, 25),\\
\hphantom{[}(25, 0, 26, 1)].
\end{flushleft}
Crossing signs:
\begin{flushleft}\footnotesize\ttfamily
[-1, 1, 1, 1, -1, 1, 1, -1, 1,\\
\hphantom{[}1, 1, 1, 1, 1, 1, 1, 1, 1,\\
\hphantom{[}1, -1, -1, -1, -1].
\end{flushleft}

The band applied is
\begin{flushleft}\footnotesize\ttfamily
[\textquotesingle{}C7S3\_start\textquotesingle{}, \textquotesingle{}C4S0\_end1\textquotesingle{}].
\end{flushleft}
\end{minipage}\hfill
\begin{minipage}[c]{0.52\linewidth}
    \captionsetup{type=figure}
    \centering
    \includegraphics[max width=\linewidth, max height=6.5cm]{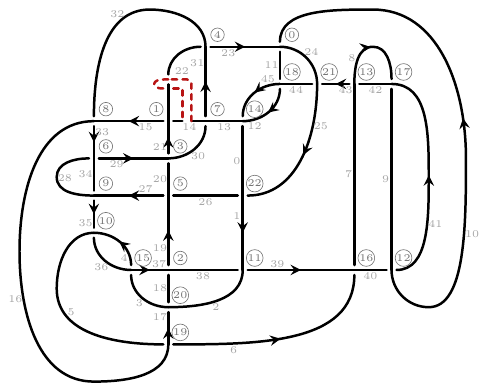}
    \caption{Sequence~\eqref{eq:seq-rweval_10_121_9_10}, band from diagram $D_{2}$ of $\mathtt{K23}$ to diagram $D_{3}$ of $\mathtt{8_{14}}$.}
    \label{fig:seq-rweval_10_121_9_10-2}
\end{minipage}

\bigskip
\noindent
\begin{minipage}[c]{0.52\linewidth}
    \captionsetup{type=figure}
    \centering
    \includegraphics[max width=\linewidth, max height=5cm]{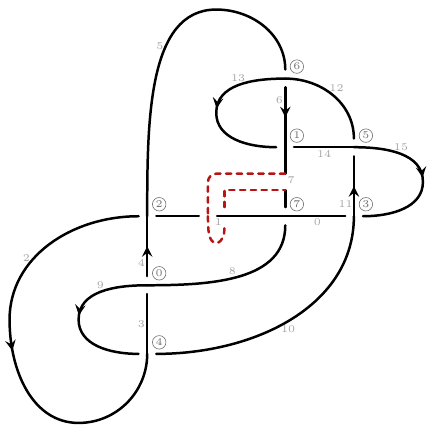}
    \caption{Sequence~\eqref{eq:seq-rweval_10_121_9_10}, band from diagram $D_{3}$ of $\mathtt{8_{14}}$ to diagram $D_{4}$ of the unknot.}
    \label{fig:seq-rweval_10_121_9_10-3}
\end{minipage}\hfill
\begin{minipage}[c]{0.45\linewidth}
In Figure~\ref{fig:seq-rweval_10_121_9_10-3} we see the diagram $D_3$ of $\mathtt{8_{14}}$ to which we apply the fourth band.
Before the unknotting band, we have \\
PD code:
\begin{flushleft}\footnotesize\ttfamily
[(3, 8, 4, 9), (13, 7, 14, 6),\\
\hphantom{[}(1, 5, 2, 4), (15, 11, 0, 10),\\
\hphantom{[}(9, 2, 10, 3), (11, 15, 12, 14),\\
\hphantom{[}(5, 13, 6, 12), (7, 1, 8, 0)].
\end{flushleft}
Crossing signs:
\begin{flushleft}\footnotesize\ttfamily
[-1, 1, 1, 1, -1, 1, 1, 1].
\end{flushleft}

The band applied is
\begin{flushleft}\footnotesize\ttfamily
[\textquotesingle{}C1S1\_start\textquotesingle{}, \textquotesingle{}C2S0\_end1\textquotesingle{}].
\end{flushleft}
\end{minipage}

\medskip

After the last crossing change, we obtain the unknot.

\subsection{Unknotting sequence for \texorpdfstring{$\mathtt{10_{140}} \sharp \mathtt{9_{10}}$}{10\_140 \# 9\_10} (random walker)}
\label{app:seq-rweval_10_140_9_10}

The sequence found by the tuned random walker is
\begin{equation}
\label{eq:seq-rweval_10_140_9_10}
     \mathtt{10_{140}} \sharp \mathtt{9_{10}} \to \mathtt{K21} \to \mathtt{K23} \to \mathtt{8_{14}} \to \text{unknot}.
\end{equation}
The diagram $D_3$ and its band are those of the sequence of $\mathtt{10_{41}} \sharp \mathtt{9_{10}}$, up to relabelling.

\bigskip
\noindent
\begin{minipage}[c]{0.45\linewidth}
In Figure~\ref{fig:seq-rweval_10_140_9_10-0} we see the diagram $D_0$ of $\mathtt{10_{140}} \sharp \mathtt{9_{10}}$ to which we apply the first band.
Before the unknotting band, we have \\
PD code:
\begin{flushleft}\footnotesize\ttfamily
[(11, 25, 12, 24), (13, 20, 14, 21),\\
\hphantom{[}(15, 22, 16, 23), (16, 28, 17, 27),\\
\hphantom{[}(18, 26, 19, 25), (21, 14, 22, 15),\\
\hphantom{[}(23, 31, 24, 30), (26, 18, 27, 17),\\
\hphantom{[}(28, 20, 29, 19), (29, 13, 30, 12),\\
\hphantom{[}(31, 5, 32, 4), (33, 3, 34, 2),\\
\hphantom{[}(35, 9, 36, 8), (37, 7, 0, 6),\\
\hphantom{[}(1, 11, 2, 10), (3, 33, 4, 32),\\
\hphantom{[}(5, 35, 6, 34), (7, 37, 8, 36),\\
\hphantom{[}(9, 1, 10, 0)].
\end{flushleft}
Crossing signs:
\begin{flushleft}\footnotesize\ttfamily
[1, -1, -1, 1, 1, -1, 1, 1, 1,\\
\hphantom{[}1, 1, 1, 1, 1, 1, 1, 1, 1,\\
\hphantom{[}1].
\end{flushleft}

The band applied is
\begin{flushleft}\footnotesize\ttfamily
[\textquotesingle{}C3S0\_start\textquotesingle{}, \textquotesingle{}C16S0\_end2\textquotesingle{}].
\end{flushleft}
\end{minipage}\hfill
\begin{minipage}[c]{0.52\linewidth}
    \captionsetup{type=figure}
    \centering
    \includegraphics[max width=\linewidth, max height=6.5cm]{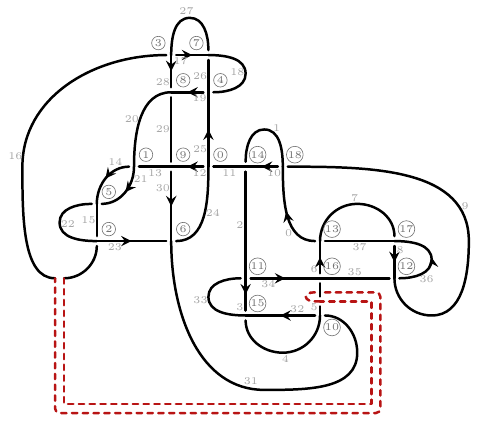}
    \caption{Sequence~\eqref{eq:seq-rweval_10_140_9_10}, band from diagram $D_{0}$ of $\mathtt{10_{140}} \sharp \mathtt{9_{10}}$ to diagram $D_{1}$ of $\mathtt{K21}$.}
    \label{fig:seq-rweval_10_140_9_10-0}
\end{minipage}

\bigskip
\noindent
\begin{minipage}[c]{0.52\linewidth}
    \captionsetup{type=figure}
    \centering
    \includegraphics[max width=\linewidth, max height=6.5cm]{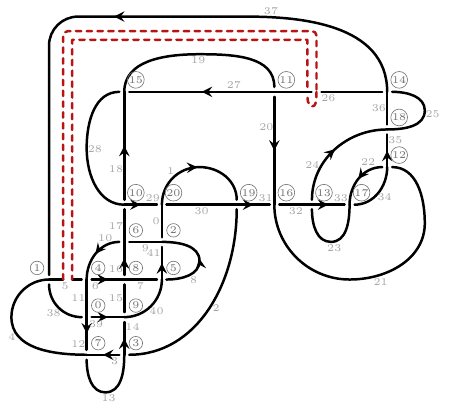}
    \caption{Sequence~\eqref{eq:seq-rweval_10_140_9_10}, band from diagram $D_{1}$ of $\mathtt{K21}$ to diagram $D_{2}$ of $\mathtt{K23}$.}
    \label{fig:seq-rweval_10_140_9_10-1}
\end{minipage}\hfill
\begin{minipage}[c]{0.45\linewidth}
In Figure~\ref{fig:seq-rweval_10_140_9_10-1} we see the diagram $D_1$ of $\mathtt{K21}$ to which we apply the second band.
Before the unknotting band, we have \\
PD code:
\begin{flushleft}\footnotesize\ttfamily
[(38, 12, 39, 11), (37, 4, 38, 5),\\
\hphantom{[}(41, 8, 0, 9), (2, 14, 3, 13),\\
\hphantom{[}(5, 11, 6, 10), (7, 40, 8, 41),\\
\hphantom{[}(9, 17, 10, 16), (12, 4, 13, 3),\\
\hphantom{[}(15, 7, 16, 6), (14, 40, 15, 39),\\
\hphantom{[}(17, 29, 18, 28), (19, 27, 20, 26),\\
\hphantom{[}(21, 35, 22, 34), (23, 33, 24, 32),\\
\hphantom{[}(25, 37, 26, 36), (27, 19, 28, 18),\\
\hphantom{[}(31, 21, 32, 20), (33, 23, 34, 22),\\
\hphantom{[}(35, 25, 36, 24), (1, 30, 2, 31),\\
\hphantom{[}(29, 0, 30, 1)].
\end{flushleft}
Crossing signs:
\begin{flushleft}\footnotesize\ttfamily
[1, -1, -1, 1, 1, -1, 1, 1, 1,\\
\hphantom{[}1, 1, 1, 1, 1, 1, 1, 1, 1,\\
\hphantom{[}1, -1, -1].
\end{flushleft}

The band applied is
\begin{flushleft}\footnotesize\ttfamily
[\textquotesingle{}C4S0\_start\textquotesingle{}, \textquotesingle{}C11S3\_end2\textquotesingle{}].
\end{flushleft}
\end{minipage}

\bigskip
\noindent
\begin{minipage}[c]{0.45\linewidth}
In Figure~\ref{fig:seq-rweval_10_140_9_10-2} we see the diagram $D_2$ of $\mathtt{K23}$ to which we apply the third band.
Before the unknotting band, we have \\
PD code:
\begin{flushleft}\footnotesize\ttfamily
[(39, 9, 40, 8), (36, 1, 37, 2),\\
\hphantom{[}(40, 5, 41, 6), (44, 12, 45, 11),\\
\hphantom{[}(4, 10, 5, 9), (3, 38, 4, 39),\\
\hphantom{[}(7, 15, 8, 14), (10, 0, 11, 45),\\
\hphantom{[}(13, 7, 14, 6), (12, 42, 13, 41),\\
\hphantom{[}(15, 29, 16, 28), (17, 27, 18, 26),\\
\hphantom{[}(19, 35, 20, 34), (21, 33, 22, 32),\\
\hphantom{[}(24, 38, 25, 37), (27, 17, 28, 16),\\
\hphantom{[}(31, 19, 32, 18), (33, 21, 34, 20),\\
\hphantom{[}(35, 23, 36, 22), (43, 30, 44, 31),\\
\hphantom{[}(29, 42, 30, 43), (2, 25, 3, 26),\\
\hphantom{[}(23, 0, 24, 1)].
\end{flushleft}
Crossing signs:
\begin{flushleft}\footnotesize\ttfamily
[1, -1, -1, 1, 1, -1, 1, 1, 1,\\
\hphantom{[}1, 1, 1, 1, 1, 1, 1, 1, 1,\\
\hphantom{[}1, -1, -1, -1, -1].
\end{flushleft}

The band applied is
\begin{flushleft}\footnotesize\ttfamily
[\textquotesingle{}C6S1\_start\textquotesingle{}, \textquotesingle{}C6S0\_end1\textquotesingle{}].
\end{flushleft}
\end{minipage}\hfill
\begin{minipage}[c]{0.52\linewidth}
    \captionsetup{type=figure}
    \centering
    \includegraphics[max width=\linewidth, max height=6.5cm]{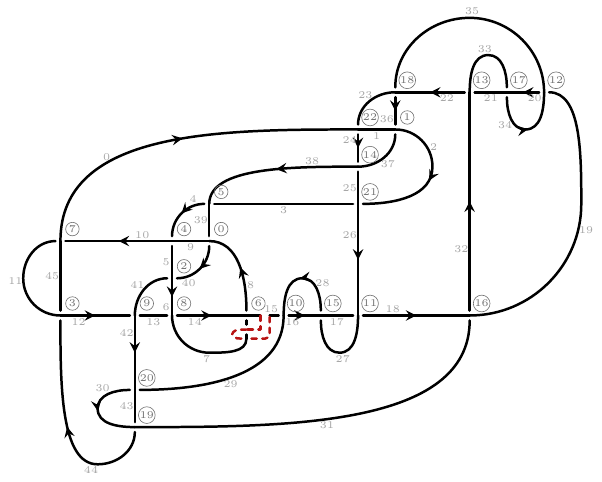}
    \caption{Sequence~\eqref{eq:seq-rweval_10_140_9_10}, band from diagram $D_{2}$ of $\mathtt{K23}$ to diagram $D_{3}$ of $\mathtt{8_{14}}$.}
    \label{fig:seq-rweval_10_140_9_10-2}
\end{minipage}

\bigskip
\noindent
\begin{minipage}[c]{0.52\linewidth}
    \captionsetup{type=figure}
    \centering
    \includegraphics[max width=\linewidth, max height=5cm]{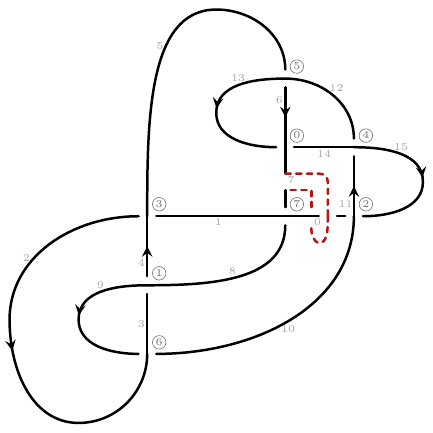}
    \caption{Sequence~\eqref{eq:seq-rweval_10_140_9_10}, band from diagram $D_{3}$ of $\mathtt{8_{14}}$ to diagram $D_{4}$ of the unknot.}
    \label{fig:seq-rweval_10_140_9_10-3}
\end{minipage}\hfill
\begin{minipage}[c]{0.45\linewidth}
In Figure~\ref{fig:seq-rweval_10_140_9_10-3} we see the diagram $D_3$ of $\mathtt{8_{14}}$ to which we apply the fourth band.
Before the unknotting band, we have \\
PD code:
\begin{flushleft}\footnotesize\ttfamily
[(13, 7, 14, 6), (3, 8, 4, 9),\\
\hphantom{[}(15, 11, 0, 10), (1, 5, 2, 4),\\
\hphantom{[}(11, 15, 12, 14), (5, 13, 6, 12),\\
\hphantom{[}(9, 2, 10, 3), (7, 1, 8, 0)].
\end{flushleft}
Crossing signs:
\begin{flushleft}\footnotesize\ttfamily
[1, -1, 1, 1, 1, 1, -1, 1].
\end{flushleft}

The band applied is
\begin{flushleft}\footnotesize\ttfamily
[\textquotesingle{}C7S0\_start\textquotesingle{}, \textquotesingle{}C7S3\_end2\textquotesingle{}].
\end{flushleft}
\end{minipage}

\medskip

After the last crossing change, we obtain the unknot.

\subsection{Unknotting sequence for \texorpdfstring{$\mathtt{10_{140}} \sharp \mathtt{9_{10}}$}{10\_140 \# 9\_10} (mixed agent)}
\label{app:seq-rleval_10_140_9_10}

The sequence found by the mixed agent is
\begin{equation}
\label{eq:seq-rleval_10_140_9_10}
     \mathtt{10_{140}} \sharp \mathtt{9_{10}} \to \mathtt{K21} \to \mir{\mathtt{K15n4866}} \to \mathtt{8_{14}} \to \text{unknot}.
\end{equation}
The diagram $D_0$ is that of the random walker's sequence above (the same PD code); the band differs.

\bigskip
\noindent
\begin{minipage}[c]{0.45\linewidth}
In Figure~\ref{fig:seq-rleval_10_140_9_10-0} we see the diagram $D_0$ of $\mathtt{10_{140}} \sharp \mathtt{9_{10}}$ to which we apply the first band.
Before the unknotting band, we have \\
PD code:
\begin{flushleft}\footnotesize\ttfamily
[(11, 25, 12, 24), (13, 20, 14, 21),\\
\hphantom{[}(15, 22, 16, 23), (16, 28, 17, 27),\\
\hphantom{[}(18, 26, 19, 25), (21, 14, 22, 15),\\
\hphantom{[}(23, 31, 24, 30), (26, 18, 27, 17),\\
\hphantom{[}(28, 20, 29, 19), (29, 13, 30, 12),\\
\hphantom{[}(31, 5, 32, 4), (33, 3, 34, 2),\\
\hphantom{[}(35, 9, 36, 8), (37, 7, 0, 6),\\
\hphantom{[}(1, 11, 2, 10), (3, 33, 4, 32),\\
\hphantom{[}(5, 35, 6, 34), (7, 37, 8, 36),\\
\hphantom{[}(9, 1, 10, 0)].
\end{flushleft}
Crossing signs:
\begin{flushleft}\footnotesize\ttfamily
[1, -1, -1, 1, 1, -1, 1, 1, 1,\\
\hphantom{[}1, 1, 1, 1, 1, 1, 1, 1, 1,\\
\hphantom{[}1].
\end{flushleft}

The band applied is
\begin{flushleft}\footnotesize\ttfamily
[\textquotesingle{}C7S2\_start\textquotesingle{}, \textquotesingle{}C14S0\_end1\textquotesingle{}].
\end{flushleft}
\end{minipage}\hfill
\begin{minipage}[c]{0.52\linewidth}
    \captionsetup{type=figure}
    \centering
    \includegraphics[max width=\linewidth, max height=6.5cm]{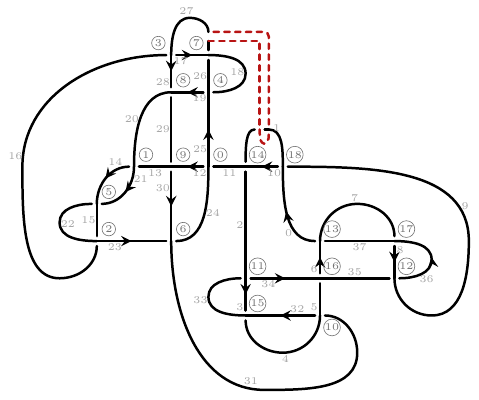}
    \caption{Sequence~\eqref{eq:seq-rleval_10_140_9_10}, band from diagram $D_{0}$ of $\mathtt{10_{140}} \sharp \mathtt{9_{10}}$ to diagram $D_{1}$ of $\mathtt{K21}$.}
    \label{fig:seq-rleval_10_140_9_10-0}
\end{minipage}

\bigskip
\noindent
\begin{minipage}[c]{0.52\linewidth}
    \captionsetup{type=figure}
    \centering
    \includegraphics[max width=\linewidth, max height=6.5cm]{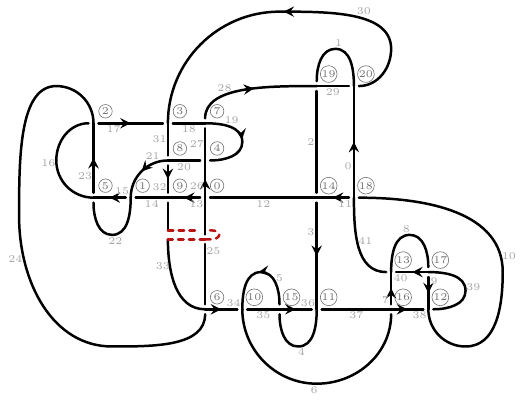}
    \caption{Sequence~\eqref{eq:seq-rleval_10_140_9_10}, band from diagram $D_{1}$ of $\mathtt{K21}$ to diagram $D_{2}$ of $\mir{\mathtt{K15n4866}}$.}
    \label{fig:seq-rleval_10_140_9_10-1}
\end{minipage}\hfill
\begin{minipage}[c]{0.45\linewidth}
In Figure~\ref{fig:seq-rleval_10_140_9_10-1} we see the diagram $D_1$ of $\mathtt{K21}$ to which we apply the second band.
Before the unknotting band, we have \\
PD code:
\begin{flushleft}\footnotesize\ttfamily
[(12, 26, 13, 25), (14, 21, 15, 22),\\
\hphantom{[}(16, 23, 17, 24), (17, 31, 18, 30),\\
\hphantom{[}(19, 27, 20, 26), (22, 15, 23, 16),\\
\hphantom{[}(24, 34, 25, 33), (27, 19, 28, 18),\\
\hphantom{[}(31, 21, 32, 20), (32, 14, 33, 13),\\
\hphantom{[}(34, 6, 35, 5), (36, 4, 37, 3),\\
\hphantom{[}(38, 10, 39, 9), (40, 8, 41, 7),\\
\hphantom{[}(2, 12, 3, 11), (4, 36, 5, 35),\\
\hphantom{[}(6, 38, 7, 37), (8, 40, 9, 39),\\
\hphantom{[}(10, 0, 11, 41), (1, 28, 2, 29),\\
\hphantom{[}(29, 0, 30, 1)].
\end{flushleft}
Crossing signs:
\begin{flushleft}\footnotesize\ttfamily
[1, -1, -1, 1, 1, -1, 1, 1, 1,\\
\hphantom{[}1, 1, 1, 1, 1, 1, 1, 1, 1,\\
\hphantom{[}1, -1, -1].
\end{flushleft}

The band applied is
\begin{flushleft}\footnotesize\ttfamily
[\textquotesingle{}C6S3\_start\textquotesingle{}, \textquotesingle{}C6S2\_end1\textquotesingle{}].
\end{flushleft}
\end{minipage}

\bigskip
\noindent
\begin{minipage}[c]{0.45\linewidth}
In Figure~\ref{fig:seq-rleval_10_140_9_10-2} we see the diagram $D_2$ of $\mir{\mathtt{K15n4866}}$ to which we apply the third band.
Before the unknotting band, we have \\
PD code:
\begin{flushleft}\footnotesize\ttfamily
[(7, 20, 8, 21), (2, 5, 3, 6),\\
\hphantom{[}(21, 6, 22, 7), (4, 10, 5, 9),\\
\hphantom{[}(29, 19, 0, 18), (8, 1, 9, 2),\\
\hphantom{[}(27, 17, 28, 16), (23, 13, 24, 12),\\
\hphantom{[}(25, 15, 26, 14), (10, 4, 11, 3),\\
\hphantom{[}(11, 25, 12, 24), (13, 23, 14, 22),\\
\hphantom{[}(15, 29, 16, 28), (17, 27, 18, 26),\\
\hphantom{[}(19, 1, 20, 0)].
\end{flushleft}
Crossing signs:
\begin{flushleft}\footnotesize\ttfamily
[-1, -1, -1, 1, 1, -1, 1, 1, 1,\\
\hphantom{[}1, 1, 1, 1, 1, 1].
\end{flushleft}

The band applied is
\begin{flushleft}\footnotesize\ttfamily
[\textquotesingle{}C9S3\_start\textquotesingle{}, \textquotesingle{}C10S3\_end2\textquotesingle{}].
\end{flushleft}
\end{minipage}\hfill
\begin{minipage}[c]{0.52\linewidth}
    \captionsetup{type=figure}
    \centering
    \includegraphics[max width=\linewidth, max height=6.5cm]{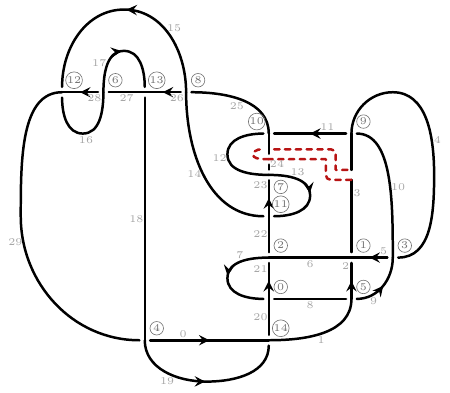}
    \caption{Sequence~\eqref{eq:seq-rleval_10_140_9_10}, band from diagram $D_{2}$ of $\mir{\mathtt{K15n4866}}$ to diagram $D_{3}$ of $\mathtt{8_{14}}$.}
    \label{fig:seq-rleval_10_140_9_10-2}
\end{minipage}

\bigskip
\noindent
\begin{minipage}[c]{0.52\linewidth}
    \captionsetup{type=figure}
    \centering
    \includegraphics[max width=\linewidth, max height=5cm]{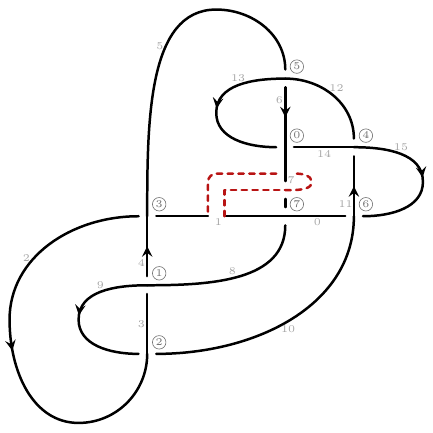}
    \caption{Sequence~\eqref{eq:seq-rleval_10_140_9_10}, band from diagram $D_{3}$ of $\mathtt{8_{14}}$ to diagram $D_{4}$ of the unknot.}
    \label{fig:seq-rleval_10_140_9_10-3}
\end{minipage}\hfill
\begin{minipage}[c]{0.45\linewidth}
In Figure~\ref{fig:seq-rleval_10_140_9_10-3} we see the diagram $D_3$ of $\mathtt{8_{14}}$ to which we apply the fourth band.
Before the unknotting band, we have \\
PD code:
\begin{flushleft}\footnotesize\ttfamily
[(13, 7, 14, 6), (3, 8, 4, 9),\\
\hphantom{[}(9, 2, 10, 3), (1, 5, 2, 4),\\
\hphantom{[}(11, 15, 12, 14), (5, 13, 6, 12),\\
\hphantom{[}(15, 11, 0, 10), (7, 1, 8, 0)].
\end{flushleft}
Crossing signs:
\begin{flushleft}\footnotesize\ttfamily
[1, -1, -1, 1, 1, 1, 1, 1].
\end{flushleft}

The band applied is
\begin{flushleft}\footnotesize\ttfamily
[\textquotesingle{}C7S1\_start\textquotesingle{}, \textquotesingle{}C7S0\_end1\textquotesingle{}].
\end{flushleft}
\end{minipage}

\medskip

After the last crossing change, we obtain the unknot.

\newpage

\printbibliography[
heading=bibliography,
title={Bibliography}
]

\end{document}